\documentclass[english,11pt, a4paper]{article}

\usepackage[T1]{fontenc}

\usepackage{amssymb}
\usepackage{amsmath}
\usepackage{amsfonts}
\usepackage{amsthm}
\usepackage{mathtools}
\usepackage{mathabx}  

\usepackage[dvipsnames]{xcolor}
\usepackage{graphicx}
\usepackage{float}

\usepackage{mdframed}
\usepackage{enumitem} 

\usepackage[export]{adjustbox}
\usepackage{bbold}
\usepackage[percent]{overpic}
\usepackage{units}

\usepackage{booktabs,array}
\usepackage{url} 
\usepackage{hyperref}

\let\OLDthebibliography\thebibliography
\renewcommand\thebibliography[1]{
  \OLDthebibliography{#1}
  \setlength{\parskip}{0pt}
  \setlength{\itemsep}{0pt plus 0.3ex}
}

\hypersetup{
    colorlinks=true,
    citecolor=purple!70!blue,
    linkcolor=black,
    filecolor=black,      
    urlcolor=black,
    pdftitle={},
    }

\usepackage{pgfplots}
\pgfplotsset{compat=1.6}

\usetikzlibrary{calc,quotes,angles,arrows.meta,positioning, shapes}
\usetikzlibrary{decorations.pathreplacing,decorations.markings}

\usepackage{float}

\usepackage{caption}
\usepackage{subcaption}

\usepackage[nocompress]{cite}

\theoremstyle{plain}%
\newtheorem{theorem}{Theorem}[section]
\newtheorem{lemma}[theorem]{Lemma}
\newtheorem{proposition}[theorem]{Proposition}

\newtheorem*{conjecture*}{Conjecture}

 \numberwithin{equation}{section}

\theoremstyle{definition}
\newtheorem{definition}[theorem]{Definition}

\theoremstyle{remark}
\newtheorem{remark}[theorem]{Remark}

\let \le \leqslant
 \let \leq \leqslant
 \let \geq \geqslant
 \let \ge \geqslant
 
\DeclareMathOperator{\tr}{tr}
\DeclareMathOperator{\real}{Re}

\DeclareMathOperator{\support}{supp}

\DeclareMathOperator{\dist}{dist} 
 
\DeclareMathOperator*{\argmin}{arg\, min}

\usepackage{fancyhdr}
\usepackage{graphics}

\definecolor{detailcolor00}{rgb}{0.4405, 0.204, 0.343}
\definecolor{detailcolor01}{rgb}{0.546, 0.215, 0.352}
\definecolor{detailcolor02}{rgb}{0.675, 0.247, 0.387} 
\definecolor{detailcolor03}{rgb}{0.775, 0.317, 0.455}
\definecolor{detailcolor04}{rgb}{0.830, 0.421, 0.553} 
\definecolor{detailcolor05}{rgb}{0.831, 0.533, 0.663}
\definecolor{detailcolor06}{rgb}{0.779, 0.619, 0.775}
\definecolor{detailcolor07}{rgb}{0.724, 0.694, 0.827}
\definecolor{detailcolor08}{rgb}{0.687, 0.770, 0.880}
\definecolor{detailcolor09}{rgb}{0.671, 0.839, 0.904}
\definecolor{detailcolor10}{rgb}{0.659, 0.872, 0.882}

\usepackage{scalerel,stackengine}
\newcommand\pig[1]{\scalerel*[5.5pt]{\Big#1}{%
  \ensurestackMath{\addstackgap[1.5pt]{\big#1}}}}
\newcommand\pigl[1]{\mathopen{\pig{#1}}}
\newcommand\pigr[1]{\mathclose{\pig{#1}}}

\ifdim\overfullrule>0pt 
  \usepackage{environ}

  \NewEnviron{tikzpicture}{%
    \begin{pgfpicture}
    \pgfpathrectanglecorners{\pgfpointorigin}{\pgfpoint{3cm}{3cm}}%
     \pgfusepath{stroke}\end{pgfpicture}%
  }
\fi

\title{The phase transition of the Marcu-Fredenhagen ratio in the abelian lattice Higgs model}

\author{Malin P. Forsstr\"om}
\author{Malin P. Forsstr\"om\thanks{Email: palo@chalmers.se\newline\indent Address: Mathematical Sciences, Chalmers University of Technology and University of Gothenburg, SE-412 96 Göteborg, Sweden}}

\begin{document}

\maketitle

\begin{abstract}
The Marcu-Fredenhagen ratio is a quantity used in the physics literature to differentiate between phases in lattice Higgs models. It is defined as the limit of a ratio of expectations of Wilson line observables as the length of these lines go to infinity while the parameters of the model are kept fixed. In this paper, we show that the Marcu-Fredenhagen ratio exists in all predicted phases of the model, and show that it indeed undergoes a phase transition. In the Higgs phase of the model we do a more careful analysis of the ratio to deduce its first order behaviour and also give an upper bound on its rate of convergence. Finally, we also present a short and concise proof of the exponential decay of correlations in the Higgs phase.
\end{abstract}

\section{Introduction}

Lattice gauge theories are spin models on the directed edges of lattices, which takes spins in some group \( G,\) referred to as the structure group or gauge group. Lattice gauge theories were introduced independently by Wilson~\cite{w1974}, as lattice approximations of the quantum field theories that appear in the standard model (known as Yang-Mills theories), and by Wegner in~\cite{w1971}, as an example of a spin system with a phase transition without a local order parameter. The lattice Higgs model is a lattice gauge theory coupled to an external field.
%
Since their introduction, lattice gauge theories and the lattice Higgs model have attracted great interest in the physics community, and have been successfully used both for simulations and as toy models for the Yang-Mills model~\cite{fs1979,s1988}.

The natural observables in lattice Higgs models are Wilson loop observables, Wilson line observables, and ratios of such observables, such as the Marcu-Fredenhagen ratio \( \rho \) (see, e.g., \cite{bf1983, hgjjkn1987, fmf1986, m, s1988, bf1987, g2006, gr2002,fv2023}), which is the main focus of this paper.  These are all natural observables from a physics perspective (see, e.g.~\cite{hgjjkn1987, bf1987}), but are also interesting from a mathematical standpoint since they are believed to undergo phase transitions~\cite{s1988}. We draw the conjectured phase diagram (see, e.g.,~\cite{c1980,fs1979,ghms2011}) of the lattice Higgs model with gauge group \( \mathbb{Z}_2, \) also known as the Ising lattice Higgs model, in~Figure~\ref{fig: phase diagram}.
\begin{figure}[htp]\centering
		\begin{tikzpicture}[scale=1] 
  		\begin{axis}[xmin=-0.05,xmax=0.8,ymin=-0.05,ymax=0.8,xticklabel=\empty,yticklabel=\empty,axis lines = middle,xlabel=$\beta$,ylabel=$\kappa$,label style =
               {at={(ticklabel cs:1.0)}}, color=detailcolor07, ticks=none]

               \addplot[domain=0:0.32, samples=100, color=detailcolor00, dashed] {0.61-2*x^1.8};
               
               \addplot[domain=0.32:0.42, samples=100, color=detailcolor00] {0.61-2*x^1.8};
               
               \addplot[domain=0.42:0.44, samples=100, color=detailcolor00]{0.61-2*0.42^1.8-100*(x-0.42)^1.6};
               \addplot[domain=0.42:0.9, samples=100, color=detailcolor00]{0.61-2*0.42^1.8-0.1*(x-0.42)+0.14*(x-0.42)^2};
  		\end{axis}
  		
  		\draw (1.9,1.9) node[] {\scriptsize Confinement};
  		\draw (1.9,1.6) node[] {\scriptsize phase};
  		\draw (5.3,0.9) node[] {\scriptsize Free phase};
  		\draw (4.2,3.7) node[] {\scriptsize Higgs phase};

  		\draw (4.2,3.35) node[] {\tiny \(\rho \neq 0\)};
  		\draw (1.85,1.25) node[] {\tiny \(\rho \neq 0\)};
  		\draw (5.3,0.55) node[] {\tiny \(\rho = 0\)};

	\end{tikzpicture}

	\caption{The conjectured phase diagram of the Ising lattice Higgs model. In the Higgs phase and the confinement phase, the Marcu-Fredenhagen ratio is believed to be non-zero, and one expects exponential decay of correlations. In contrast, in the free phase, the Marcu-Fredenhagen ratio is believed to be identically zero, and expects exponential decay of correlations with polynomial correction.}\label{fig: phase diagram}
\end{figure}
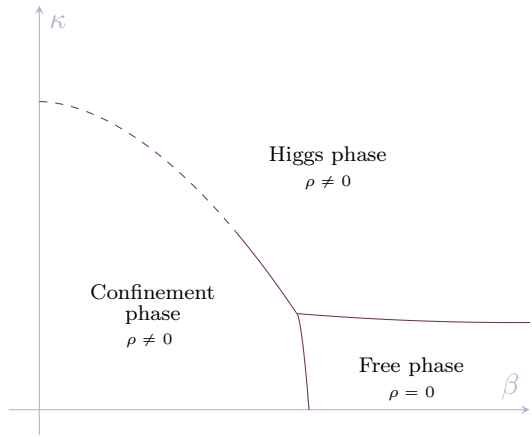
For further background, as well as more references, we refer the reader to~\cite{fs1979} and~\cite{s1988}.

In recent years, there has been a renewed interest in both lattice gauge theories and the lattice Higgs model in the mathematical community. In particular, in~\cite{a2021,flv2020,flv2023,c2019,sc2019, fv2023}, the asymptotic behavior of Wilson loop observables was described, and in~\cite{f2022b}, similar results were obtained for Wilson line observables. Further, ideas from disagreement percolation were used to understand the rate of the decay of correlations in~\cite{f2022,ac2022}. Unfortunately, the methods applied in these papers cannot be used to understand the Marcu-Fredenhagen ratio, which requires letting the length of the involved Wilson lines tend to infinity while the parameters of the models are kept fixed. 
This problem was the main motivation for the current paper. 
Our main results show that the Marcu-Fredenhagen ratio is non-zero in non-trivial subsets of the Higgs and confinement phases, while identically zero in a non-trivial subset of the free phase. As a consequence, it follows that the model undergo at least one phase transition. 
One of the main tools of the paper are various special cases of the cluster expansion in~\cite{fv2017}. This was inspired by the use of such expansions for pure lattice gauge theories in~\cite{fv2023}, but the use of these are more complicated when a Higgs field is added to the model and also needs to be different for the different phases of the model. To our knowledge, cluster expansions have not been used to study neither the Marcu-Fredenhagen ratio nor Wilson line observables prior to this paper. In particular, Wilson line observables need special handling in the free phase, where the natural cluster expansion does not converge.
Finally, we obtain a very short proof of exponential decay of correlations, which gives an alternative proof of the main results of~\cite{f2022} and~\cite{ac2022} in the case \( G = \mathbb{Z}_2,\) and also extends these from Wilson loops to the more general Wilson lines. For simplicity we state and prove all our results for \( G = \mathbb{Z}_2,\) but expect the proof ideas to work in more general settings, such as for finite abelian structure groups, with small modifications.

\subsection{Preliminary notation}

For \( m \geq 2 \), a graph naturally associated to $\mathbb{Z}^m$ has a vertex at each point \( x \in \mathbb{Z}^m \) with integer coordinates and oriented edges between nearest neighbors. When \( e_1\) and \( e_2\) are two oriented edges between the same vertices but with opposite orientation, we write \( e_2 = -e_1.\)

Let \( d\mathbf{e}_1 \coloneqq (1,0,0,\ldots,0)\), \( d\mathbf{e}_2 \coloneqq (0,1,0,\ldots, 0) \), \ldots, \( d\mathbf{e}_m \coloneqq (0,\ldots,0,1) \) be oriented edges corresponding to the unit vectors in \( \mathbb{Z}^m \). We say that an oriented edge \( e \) is \emph{positively oriented} if it is equal to a translation of one of these unit vectors, i.e.,\ if there is a \( v \in \mathbb{Z}^m \) and a \( j \in \{ 1,2, \ldots, m\} \) such that \( e = v + d{\mathbf{e}}_j \). 
If \( v \in \mathbb{Z}^m \) and \( j_1 <   j_2 \), then \( p = (v +  d\mathbf{e}_{j_1}) \land  (v+ d\mathbf{e}_{j_2}) \) is a positively oriented 2-cell, also known as a  \emph{positively oriented plaquette}. We let \( C_0(\mathbb{Z}^m) \), \( C_1(\mathbb{Z}^m)\), and \( C_2(\mathbb{Z}^m) \) denote the sets of oriented vertices, edges, and plaquettes.
Next, we let \( B_N \) denote the set \(   [-N,N]^m \cap \mathbb{Z}^m \), and we let \( C_0(B_N) \), \( C_1(B_N)\), and \( C_2(B_N) \) denote the sets of oriented vertices, edges, and plaquettes, respectively, whose endpoints are all in \( B_N \).

Whenever we talk about a lattice gauge theory we do so with respect to some (abelian) group \( (G,+)  \), referred to as the \emph{structure group}, together with a unitary and faithful representation \( \rho \) of \( (G,+) \).  

Now assume that a structure group \( (G,+) \), a unitary representation \( \rho \) of \( (G,+) \), and an integer \( N\geq 1 \) are given.
We let \( \Omega^1(B_N,G) \) denote the set of all  \( G \)-valued  1-forms \( \sigma \) on \( C_1(B_N) \), i.e., the set of all \( G \)-valued functions \(\sigma \colon  e \mapsto \sigma(e) \) on \( C_1(B_N) \) such that \( \sigma(e) =  -\sigma(-e) \) for all \( e \in C_1(B_N) \).
Similarly, we let \( \Omega^0(B_N,G) \) denote the set of all \( G\)-valued functions \( \phi \colon x \mapsto \phi(x)\) on \( C_0(B_N) \) which are such that \( \phi(x) = - \phi(-x) \) for all \( x \in C_1(B_N). \) 
When \( \sigma \in \Omega^1(B_N,G) \) and \( p \in C_2(B_N) \), we let \( \partial p \) denote the formal sum of the four edges \( e_1,\) \( e_2,\) \( e_3,\) and \( e_4 \) in the oriented boundary of \( p \), and define
\begin{equation*}
    d\sigma(p) \coloneqq \sigma(\partial p) \coloneqq \sum_{e \in \partial p} \sigma(e) \coloneqq \sigma(e_1) + \sigma(e_2) + \sigma(e_3) + \sigma(e_4).
\end{equation*} 
Similarly, when \( \phi \in \Omega^0(B_N,G) \) and \( e \in C_1(B_N) \) is an edge from \( x_1 \) to \( x_2 \), we let \( \partial e \) denote the formal sum \( x_2-x_1, \) and define \( d\phi(e) \coloneqq \phi(\partial e) \coloneqq \phi(x_2) - \phi(x_1). \)

 For \( k \in \{ 0,1,\dots, m\},\) a \( k \)-chain is a formal sum of positively oriented k-cells with integer coefficients. The support of a \(1\)-chain \(\gamma\), written \(\support \gamma\), is the set of directed edges with non-zero coefficient in \( \gamma.\)

\subsection{The abelian lattice Higgs model}

In this paper, we will consider the abelian lattice Higgs model in the fixed length limit (also known as the London limit). 
Given \( \beta, \kappa \geq 0 \), the action \( S_{N,\beta,\kappa } \) for the abelian lattice Higgs model on \( B_N \) (in the fixed length limit) is, for \( \sigma \in \Omega^1(E_N,G), \) and \( \phi \in \Omega^0(B_N,G), \) defined by
\begin{equation}\label{eq: general action}
    \begin{split}
        S_{N,\beta,\kappa}(\sigma, \phi)  &\coloneqq 
        -\beta \sum_{p \in C_2(B_N)}  \tr \rho\bigl( d\sigma(p)\bigr) 
        - \kappa\sum_{e  \in C_1(B_N) }  \tr \rho\bigl( \sigma(e)-\phi(\partial e)\bigr)  .
    \end{split}
\end{equation}
Elements \( \sigma \in \Omega^1(B_N,G) \) will be referred to as \emph{gauge field configurations}, and elements  \( \phi \in \Omega^0(B_N,G) \) will be referred to as \emph{Higgs field configurations}.
The quantity \( \beta \) is known as the \emph{gauge coupling constant}, and \( \kappa \) is known as the \emph{hopping parameter}. For a discussion of this action, see~\cite{flv2023}.

The Gibbs measure \( \mu_{N,\beta,\kappa} \) on \(\Omega^1(B_N,G) \times \Omega^0(B_N,G)\) corresponding to the action \( S_{N,\beta,\kappa} \) is given by
\begin{equation*}
    \mu_{N,\beta, \kappa}(\sigma, \phi)  \coloneqq
    Z_{N,\beta,\kappa}^{-1} e^{-S_{N,\beta,\kappa}(\sigma, \phi)} , \qquad \sigma \in \Omega^1(B_N,G) ,\, \phi \in \Omega^0(B_N,G),
\end{equation*}
where \( Z_{N,\beta,\kappa}\) is a normalizing constant. 
We refer to this lattice gauge theory as the  \emph{(fixed length) lattice Higgs model}. 
We let \( \mathbb{E}_{N,\beta,\kappa} \) denote the expectation corresponding to \( \mu_{N,\beta, \kappa}.\) 

Whenever \( f \colon \Omega^1(B_m,G) \times \Omega^0(B_m,G) \to \mathbb{R}\) for some \( m \geq 1,\) then, as a consequence of the Ginibre inequalities (see, e.g.,~\cite{f2022b}[Section 2.6]), the infinite volume limit 
\begin{equation*}
    \bigl\langle f(\sigma,\phi) \bigr\rangle_{\beta,\kappa} \coloneqq \lim_{N \to \infty} \mathbb{E}_{N,\beta,\kappa} \bigl[f(\sigma,\phi) \bigr]
\end{equation*}
exists and is translation invariant.

We say that a \(1\)-chain with finite support is a path if it has coefficients in \(\{-1,0,1\}.\)
We say that a path is a \emph{loop} if it has empty boundary \( \partial \gamma\) (see Section~\ref{section: preliminaries}). For example, any rectangular loop, as well as any finite disjoint union of such loops, corresponds to such a loop.
We say that a path is an \emph{open path} from \( x_1\in C_0^+(B_N) \) to \( x_2 \in C_0^+(B_N)\) if it has boundary \( \partial \gamma \coloneqq x_2 - x_1. \)

Given a path \( \gamma \), a gauge field configuration \( \sigma\in \Omega^1(B_N,G) ,\) and a Higgs field configuration \(  \phi\in \Omega^0(B_N,G) ,\)
the \emph{Wilson line observable} \( W_\gamma(\sigma,\phi) \) is defined by 
\begin{equation*}
    W_\gamma(\sigma,\phi)
    \coloneqq \tr \rho \bigl( \sigma(\gamma) - \phi(\partial \gamma) \bigr)
    =
    \tr \rho\Bigl( \sum_{e \in \gamma} \sigma(e) - \sum_{v \in \partial \gamma} \phi(v)\Bigr).
\end{equation*}
If \( \gamma \) is an open path from \( x_1 \) to \( x_2 \), then \( \phi( \partial \gamma) = \phi(x_2)-\phi(x_1), \) and if \( \gamma \) is a closed loop, then  \(\phi(\partial \gamma) = 0. \) If \( \gamma \) is a  loop, then \(  W_\gamma(\sigma) \coloneqq W_\gamma(\sigma,\phi)\) is referred to as a \emph{Wilson loop observable}.

\subsection{The Marcu-Fredenhagen ratio}

Assume that \( (R_n)_{n\geq 1}\) and \( (T_n)_{n\geq 1}\) are increasing sequences of positive integers. For each \( n \geq 1,\) let \( \gamma^{(n)} \) be a rectangular loop with side lengths \( 2R_n\) and \( T_n\) (see Figure~\ref{figure: U Wilson line b}). Let \( \gamma_1^{(n)}\) be as in Figure~\ref{figure: U Wilson line a}, and let \( \gamma_2^{(n)} = \gamma^{(n)}-\gamma_1^{(n)}.\) 

Define
\begin{equation}\label{eq: line loop ratio}
    \rho(\gamma_1^{(n)},\gamma_2^{(n)}) \coloneqq \frac{\langle W_{\gamma_1^{(n)}}(\sigma,\phi) \rangle_{\beta,\kappa}\langle W_{\gamma_2^{(n)}}(\sigma,\phi) \rangle_{\beta,\kappa} }{\langle W_{\gamma^{(n)}}(\sigma)\rangle_{\beta,\kappa}}.
\end{equation}
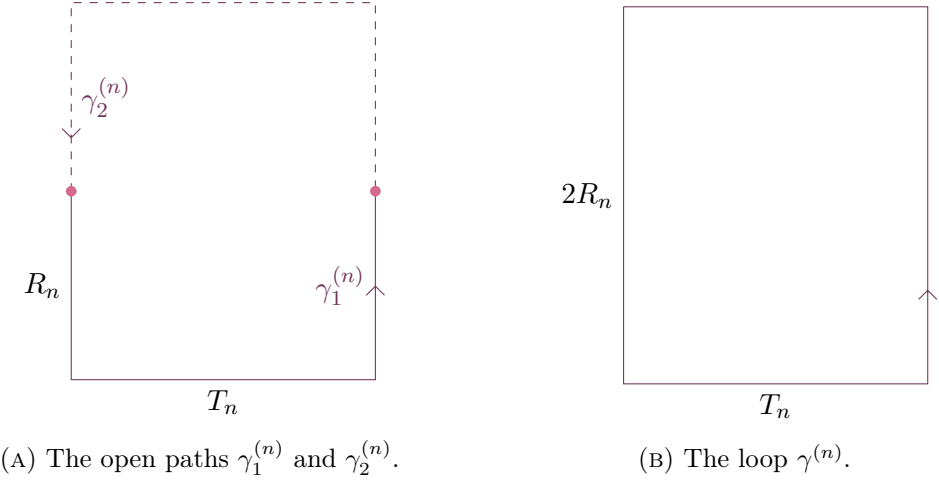
\begin{figure}[!htp]
    \centering
    \begin{subfigure}[b]{0.45\textwidth}
        \centering
        \begin{tikzpicture}
            \draw[detailcolor00] (0,2.5) -- (0,0) node[midway, left] {\color{black}\( R_n \)} -- (4,0) node[midway,anchor=north] {\color{black}\( T_n\)} -- (4,2.5) node[midway,anchor=east] {\(\gamma_1^{(n)}\)};

            \draw[detailcolor00, dashed] (0,2.5) -- (0,5)  node[midway,anchor=west] {\(\gamma_2^{(n)}\)}  -- (4,5)   -- (4,2.5);

           \draw[detailcolor00,-{Straight Barb[length=1.2mm,color=detailcolor00!40!black]}] (4,1.2) -- (4,1.25);
            
           \draw[detailcolor00,-{Straight Barb[length=1.2mm,color=detailcolor00!40!black]}] (0,3.25) -- (0,3.2);
           
            \fill[detailcolor04] (0,2.5) circle (2pt) node[anchor=south] {\color{black}\( \)};
            \fill[detailcolor04] (4,2.5) circle (2pt) node[anchor=south] {\color{black}\( \)};
        \end{tikzpicture}
        \caption{The open paths \(  \gamma_1^{(n)} \) and  \(  \gamma_2^{(n)}. \)} \label{figure: U Wilson line a} 
    \end{subfigure}
    \hfil
    \begin{subfigure}[b]{0.45\textwidth}
        \centering
        \begin{tikzpicture}
            \draw[detailcolor00] (0,5) -- (0,0) node[midway, left] {\color{black}\( 2R_n\)} --  (4,0) node[midway,anchor=north] {\color{black}\( T_n\)}  -- (4,5) -- (0,5); 
            
            \draw[detailcolor00,-{Straight Barb[length=1.2mm,color=detailcolor00!40!black]}] (4,1.2) -- (4,1.25);
        \end{tikzpicture}
        \caption{The loop \( \gamma^{(n)} \).} \label{figure: U Wilson line b} 
    \end{subfigure}
    \caption{The open paths \( \gamma_1^{(n)} \) and  \( \gamma_2^{(n)} \) and the rectangular loop \( \gamma^{(n)} \) 
 that appear in~\eqref{eq: line loop ratio}.}
    \label{figure: U Wilson line}
\end{figure}%
The limit  \( \lim_{n \to \infty} \rho(\gamma_1^{(n)},\gamma_2^{(n)})\) is referred to as the \emph{Marcu-Fredenhagen order parameter} in the physics literature (see, e.g., ~\cite{fmf1986, m}). Note that it is not obvious that this limit exists, nor that it is independent of the choice of \( (R_n)_{n\geq 1}\) and \( (T_n)_{n \geq 1}.\)  
If this limit (assuming it exists) is zero, the corresponding model is argued to have charged states, and no confinement, whereas if the limit is non-zero, then there should be no charged states and confinement, see, e.g.,~\cite{m,s1988,bf1987,fm1988,ghms2011}.

Several ratios similar to~\eqref{eq: line loop ratio} has been considered in the physics literature, see, e.g.,~\cite{g2006}, and the main ideas in this paper can be adapted to cover also these cases.

\subsection{Main results}
 
Our first main result considers the Marcu-Fredenhagen ratio in the Higgs phase.
\begin{theorem}\label{theorem: mf ratio Higgs} 
    Let \( G = \mathbb{Z}_2, \) \( \beta \geq 0,\) and \( \kappa \geq \kappa_0^{(\textrm{Higgs})},\) where \( \kappa_0^{(\textrm{Higgs})} = \kappa_0^{(\textrm{Higgs})}(m)>0\) is  defined in~\eqref{eq: kappa0}.
     Further, let \( (R_n)_{n\geq 1}\) and \( (T_n)_{n\geq 1}\) be increasing sequences of positive integers such that \( \limsup_{n\to \infty} T_n/R_n < \infty ,\) and for each \( n\geq 1,\) let \( \gamma_1^{(n)}\) and \( \gamma_2^{(n)}\) be as in Figure~\ref{figure: U Wilson line}.
    Then the following hold.
    \begin{enumerate}[label=(\roman*)]
        \item\label{item: thm 1} The limit
        \begin{equation} 
            \rho = \rho_{\beta,\kappa}\coloneqq \lim_{n \to \infty} \rho(\gamma_1^{(n)},\gamma_2^{(n)})
        \end{equation}
        exists and is independent of \( (R_n)_{n\geq 1} \) and \( (T_n)_{n\geq 1}.\)
        
        \item\label{item: thm 2} The limit \( \rho \) is strictly positive, i.e., \(\rho>0.\)
        
        \item\label{item: thm 3} For all  \(n \geq 1\) and \( \varepsilon>0, \) there is \( C_\varepsilon >0 \) such that 
        \begin{equation}\label{eq: decay rate}
            \begin{split}
                \bigl| \log \rho_{n}
                -
                \log \rho
                \bigr|
                &\leq   
                4C_{\varepsilon} \sum_{j=1}^\infty   e^{-4\max(j,\min(R_n,T_n))(\kappa-\kappa_0-\varepsilon) }
                \\&\quad+
                2C_{\varepsilon}\max(T_n-2R_n,0)
                e^{-4\max(2R_n,\min(R_n,T_n))(\kappa-\kappa_0-\varepsilon) }
                . 
            \end{split}
        \end{equation} 
        Here \( C_\varepsilon \) is defined in~\eqref{eq: cbeta*} and does not depend on \( \beta\) nor \( \kappa.\)  
    \end{enumerate} 
\end{theorem}

In other words, Theorem~\ref{theorem: mf ratio Higgs} says that the Marcu-Fredenhagen parameter exists and is strictly positive when \( \kappa \geq \kappa_0^{\text{(Higgs)}}. \) Also, it gives a upper bound on the convergence rate, thus stating how large an estimate for \( -\log \rho_n\) has to be for one to be able to conclude that \( \rho > 0.\)

We note that the assumption that \( \limsup_{n\to \infty} T_n/R_n < \infty\) is needed to guarantee that the right-hand side of~\eqref{eq: decay rate} goes to zero as \( n \to \infty.\)

Our next result complements our first theorem, Theorem~\ref{theorem: mf ratio Higgs}, by showing that the Marcu-Fredenhagen ratio is  non-zero also in parts of the confinement regime, i.e., when \( \beta \) and \( \kappa \) are both sufficiently small.  

\begin{theorem}\label{theorem: mf ratio confinement}
    Let \( G = \mathbb{Z}_2, \) \( 0 < \beta < \beta_0^{(\text{conf})}, \) and \(  \kappa >0, \) where \( \beta_0^{(\text{conf})} = \beta_0^{(\text{conf})}(m)>0\) is defined in~\eqref{eq: beta0 def}. Further, let \( (R_n)_{n\geq 1}\) and \( (T_n)_{n\geq 1}\) be increasing sequences of positive integers such that \( \limsup_{n\to \infty} T_n/R_n < \infty ,\) and for each \( n\geq 1,\) let \( \gamma_1^{(n)}\) and \( \gamma_2^{(n)}\) be as in Figure~\ref{figure: U Wilson line}. 
    Then the following hold.
    \begin{enumerate}[label=(\roman*)]
    	\item The limit
    	\begin{equation*}
    		\rho = \rho_{\beta,\kappa} \coloneqq \lim_{n\to \infty} \rho(\gamma_1^{(n)},\gamma_2^{(n)})
    	\end{equation*}
    	exists and is independent of \( (R_n)_{n\geq 1} \) and \( (T_n)_{n\geq 1}. \)
    	\item The limit \( \rho \) is strictly positive, i.e., \( \rho>0. \)
    \end{enumerate}
\end{theorem} 
Theorem~\ref{theorem: mf ratio confinement} shows that in at least parts of the confinement phase of the lattice Higgs model, the Marcu-Fredenhagen ratio is strictly positive.

Our next result concerns the Marcu-Fredenhagen ratio in the free phase.

\begin{theorem}\label{theorem: mf ratio free}
	Let \( G = \mathbb{Z}_2, \) \( \beta >0\) be suffciently large, and \(  \kappa > 0 \) be sufficiently small. Further, let \( (R_n)_{n\geq 1}\) and \( (T_n)_{n\geq 1}\) be strictly increasing sequences of positive integers, and for each \( n\geq 1,\) let \( \gamma_1^{(n)}\) and \( \gamma_2^{(n)}\) be as in Figure~\ref{figure: U Wilson line}. 
	Then 
    \begin{enumerate}[label=(\roman*)]
    	\item The limit
    	\begin{equation*}
    		\rho = \rho_{\beta,\kappa} \coloneqq \lim_{n\to \infty} \rho(\gamma_1^{(n)},\gamma_2^{(n)})
    	\end{equation*}
    	exists and is independent of \( (R_n)_{n\geq 1} \) and \( (T_n)_{n\geq 1}. \)
    	\item The limit \( \rho \) is identically zero, i.e., \( \rho=0. \)
    \end{enumerate} 
\end{theorem}
The most important consequence of Theorems~\ref{theorem: mf ratio Higgs}, \ref{theorem: mf ratio confinement} and~\ref{theorem: mf ratio free} is that they together prove that the Marcu-Fredenhagen ratio indeed has a phase transition, implying in particular that it can be used as an order parameter.

Our last result gives an upper bound on the decay of correlation in the Higgs phase of the abelian lattice Higgs model. This result extends the results in~\cite{ac2022} and~\cite{f2022} to Wilson line observables in the case \( G = \mathbb{Z}^2.\) However, the main reason we include this result here is that the methods used in this paper yields a very short proof which is very different to the proofs in~\cite{ac2022} and~\cite{f2022}.

\begin{theorem}\label{theorem: correlation}
    Let \( G = \mathbb{Z}_2, \) \( \beta \geq 0,\) and \( \kappa \geq \kappa_0^{\text{(Higgs)}}.\) Further, let \( \gamma_1 \) and \( \gamma_2 \) be two paths. Then, for any \( \varepsilon > 0\) there is \( C_\varepsilon > 0 \)  such that  
    \begin{align*} 
        &\bigl| \langle W_{\gamma_1+\gamma_2} \rangle_{\beta,\kappa} -\langle W_{\gamma_1}\rangle_{\beta,\kappa} \langle W_{\gamma_2}\rangle_{\beta,\kappa}  \bigr|
        \leq
        C_{\varepsilon}\bigl|\support \gamma_1\bigr| e^{-4(\kappa-\kappa_0-\varepsilon) \dist(\gamma_1,\gamma_2)},
    \end{align*} 
    Here \( C_\varepsilon \) is defined in~\eqref{eq: cbeta*} and does not depend on \( \beta\) nor \( \kappa,\) and \( \dist(\gamma_1,\gamma_2)\) is the \( \ell_0 \)-distance between the supports of \( \gamma_1 \) and \( \gamma_2.\)
\end{theorem}

\begin{remark}
	 The proof of~Theorem~\ref{theorem: correlation} can easily be adapted to show that the covariance decays at most exponentially also in the confinement phase and the free phase. However, we note that in the free phase such a result would not be sharp since the rate of decay in the free phase is believed to be exponential with polynomial corrections.
\end{remark}

\subsection{Relation to other work}

The main result of~\cite{f2022b} gives the asymptotic decay rate of Wilson loops and lines \( \gamma \) in the Higgs phase under the two additional assumptions that  \( 6\beta>\kappa > \kappa_0\)  and   \( {|\support \gamma|e^{-24\beta-4\kappa} \ll \infty}.\) In~\cite{flv2022}, similar results are given in the confinement phase. However, we are aware of no results about the decay of Wilson lines in the free phase, other than the well known universal lower bound \( (\tanh 2\kappa)^{|\gamma|}. \) 
For Wilson loops, several papers contain similar results, see e.g.~\cite{a2021, sc2019,flv2020}.
In all papers mentioned above, assumptions are made on the parameters so that at least one of them tend to either infinity or zero as \( |\gamma| \) grows.
As a consequence, these results cannot be applied to deduce anything about the Marcu-Fredenhagen parameter or similar ratios since this limit involves letting \( |\gamma|\) tend to infinity while keeping the parameters \( \beta \) and \( \kappa\) fixed.

In this paper, we use high temperature expansions and cluster expansions for the Higgs phase, the confinement phase, and the free phase respectively.
The cluster expansions are special cases of the cluster expansion presented in~\cite{fv2017}. 
In~\cite{fv2023}, we used a similar cluster expansion for a lattice gauge theory, but there only the case \( \kappa=0\) was considered.  This expansion is similar to the expansion we use here in the Higgs phase. However, in both the confinement and the Higgs phase we here first use high temperature expansions, while the free phase requires additional work when setting up the cluster expansion.
 
In the mathematical literature, the decay of correlations in lattice gauge theories has been studied in~\cite{ac2022} and~\cite{f2022}. In both of these papers, only observables consisting of combinations of Wilson loops were considered, and the proofs rely on couplings and giving upper bounds on events describing the vortices in the model. The proof method here is very different from this approach, and yields a much shorter proof. In~\cite{f2022,ac2022}, decay of correlations was proven for any finite structure group and any finite abelian structure group respectively. In this paper, for simplicity, we only give a proof for the structure group \( \mathbb{Z}_2,\) but the same ideas should with some work be possible to translate to any finite abelian structure group.

\subsection{Structure of paper}

In Section~\ref{section: preliminaries}, we introduce the notation and definitions we will use throughout the rest of the paper. 
Section~\ref{section: higgs}, Section~\ref{section: confinement}, and Section~\ref{section: free} contains our results for the three conjectured phases of the model; the Higgs phase, the confinement phase, and the free phase respectively.

Section~\ref{section: higgs} contains the relevant expansion and the proofs of our main results in the Higgs phase. 
In Section~\ref{sec: cluster expansions higgs}, we present the cluster expansion we will use to prove our main result. We also use this cluster expansion to express the Marcu-Fredenhagen ratio and covariance in terms of the cluster expansion.
In Section~\ref{sec: upper bounds higgs}, we give upper bounds of natural events in terms of the cluster expansion.
In Sections~\ref{sec: proof of decay} and~\ref{section: proof of main result higgs}, we give proofs of Theorem~\ref{theorem: mf ratio Higgs} and Theorem~\ref{theorem: correlation}.

Section~\ref{section: confinement} contains the relevant expansions and  proofs of our main results in the confinement phase. 
In Section~\ref{sec: high temperature expansion confinement}, we present a high temperature expansion in both parameters which will be useful in this phase.
In Section~\ref{section: cluster confinement}, we present a cluster expansion of the model obtained from the high temperature expansion.
Section~\ref{sec: upper bounds confinement} contains a upper bounds which will be useful in the proof of the main result of this section, and, finally, Section~\ref{section: proof of main result confinement} contains the proof of Theorem~\ref{theorem: mf ratio confinement}.

Section~\ref{section: free} contains the relevant expansions and  proofs of our main results in the free phase.
In Section~\ref{sec: high temperature expansion free}, we present a high temperature expansion in \( \kappa \) which will be useful in this phase.
In Section~\ref{sec: cluster expansions free}, we present a cluster expansion of a model related to the model obtained from the high temperature expansion.
In Section~\ref{section: upper bounds free}, we give upper bounds of natural events in terms of the cluster expansion.
Finally, Section~\ref{section: proof of main result free} contains the proof of Theorem~\ref{theorem: mf ratio free}.

\subsection{Acknowledgements}
The author is grateful Fredrik Viklund for many interesting discussions, and also for comments on an earlier version of this manuscript.

\section{Preliminaries}\label{section: preliminaries}

Even though we later work with $G = \mathbb{Z}_2$, in this section we allow $G$ to be a general finite abelian group since this entails no additional work. We assume that a one-dimensional unitary representation of \( G\) has been fixed.

\color{black}

\subsection{Discrete exterior calculus}

Below we present the notation from discrete exterior calculus that we need in this paper.
In order to keep the background section of this paper short, and since these definitions have appeared in several recent papers, we will refer the reader to~\cite{flv2023} for further details.

\begin{itemize}
\item  We will work with the square lattice $\mathbb{Z}^m$, where we assume that the dimension $m \ge 3$ throughout. We write $B_N = [-N,N]^m \cap \mathbb{Z}^m$. Since \( m\) will always be fixed, we suppress the dependency on \( m\) in this notation.
\item{For \( k=0,1,\dots, m,\) write $C_k(B_N)$ and $C_k(B_N)^+$ for the set of unoriented and positively oriented $k$-cells, respectively (see \cite[Sect. 2.1.2]{flv2020}).}
\item{Formal sums of positively oriented \( k \)-cells with integer coefficients are called $k$-chains, and the space of $k$-chains is denoted by \( C_k(B_N,\mathbb{Z}) \), 
(see \cite[Sect. 2.1.2]{flv2020})}
\item{Let \( k \geq 2 \) and \( c = \frac{\partial}{\partial x^{j_1}}\big|_a \wedge \dots \wedge \frac{\partial}{\partial x^{j_k}}\big|_a \in C_k(B_N)\). The \emph{boundary} of $c$ is the $(k-1)$-chain \(\partial c \in C_{k-1}(B_N, \mathbb{Z})\) defined as the formal sum of the \( (k-1)\)-cells in the (oriented) boundary of \( c.\) 
The definition is extended to $k$-chains by linearity. See \cite[Sect.~2.1.4]{flv2020}.}

\item If \( k \in \{ 0,1, \ldots, n-1 \} \) and \( c \in C_k(B_N)\) is an oriented \( k \)-cell, we define the \emph{coboundary} \( \hat \partial c \in C_{k+1}(B_N)\) of \( c \) as the \( (k+1) \)-chain $\hat \partial c \coloneqq \sum_{c' \in C_{k+1}(B_N)} \bigl(\partial c'[c] \bigr) c'.$ See \cite[Sect.~2.1.5]{flv2020}.

\item We let $\Omega^k(B_N, G)$ denote the set of $G$-valued (discrete differential) $k$-forms (see \cite[Sect 2.3.1]{flv2020}); the exterior derivative $d : \Omega^k(B_N, G)  \to \Omega^{k+1}(B_N, G)$ is defined for $0 \le k \le m-1$ (see \cite[Sect. 2.3.2]{flv2020}).

\item We write $\support \omega = \{c \in C_k(B_N): \omega(c) \neq 0\}$ for the support of a $k$-form $\omega$. Similarly, we write $(\support \omega)^+ = \{c \in C_k(B_N)^+: \omega(c) \neq 0\}.$

\item
    A 1-chain \( \gamma \in C_1(B_N,\mathbb{Z}) \) with finite support \( \support \gamma \) is called a \emph{path} if 
    for all \( e \in C_1(B_N) \), we have that \( \gamma[e] \in \{ -1,0,1 \}. \) We write \( |\gamma| = |\support \gamma|,\) and let \( \Lambda_0 \subseteq C^1(B_N,\mathbb{Z}) \) be the set of all paths.
    A path \( \gamma \) is said to be a \emph{closed path} or a \emph{loop} if  \( \partial \gamma = 0. \) 
    A path \( \gamma \) is said to be an \emph{open path} if \( |\partial \gamma| = 2. \) 
    
    \item When \( \gamma_1\) and \( \gamma_2\) are two paths, we let \( \dist(\gamma_1, \gamma_2)\) be \( \ell_0\)-distance between \( \support \gamma_1\) and \( \support \gamma_2.\) Equivalently, \( \dist(\gamma_1, \gamma_2)\) is the length of the shortest path that connects the supports of \( \gamma_1\) and \( \gamma_2.\)
\end{itemize}

\subsection{Unitary gauge}\label{sec: unitary gauge} 
In this section, we recall the notion of gauge transforms, and describe how these can be used to rewrite the Wilson line expectation as an expectation with respect to a slightly simpler probability measure. For more details on gauge transforms and unitary gauge, we refer the reader to~\cite{f2022b}.
 
For \( \eta \in \Omega^0(B_N,G) \), let the map
\[
    \tau \coloneqq \tau_\eta \coloneqq \tau_\eta^{(1)} \times \tau_\eta^{(2)} \colon \Omega^1(B_N,G)  \times \Omega^0(B_N,G)  \to  \Omega^1(B_N,G) \times \Omega^0(B_N,G) 
\]
be defined by
\begin{equation}\label{eq: gauge transform}
    \begin{cases}
     \sigma(e) \mapsto -\eta(x) +\sigma(e) + \eta(y), & e=(x,y)\in C_1(B_N), \cr
     \phi(x) \mapsto  \phi(x) + \eta(x), & x \in C_0(B_N),
    \end{cases}
\end{equation}
where \( \sigma \in \Omega^1(B_N,G) \) and \( \phi \in \Omega^0(B_N,G). \) 
A map \( \tau \) of this form is called a \emph{gauge transformation}, and functions \( f: \Omega^1(B_N,G) \times \Omega^0(B_N,G)  \to \mathbb{C} \) which are invariant under such mappings in the sense that \(f= f \circ \tau\) are said to be \emph{gauge invariant}. 

For \( \beta, \kappa \geq   0 \) and  \( \sigma \in \Omega^1(B_N,G) \), we  define the probability measure 
\begin{equation}\label{eq: fixed length unitary measure}
    \mu_{N,\beta,\kappa}^{(U)}(\sigma) 
    \coloneqq
    (Z_{N,\beta,\kappa}^{(U)})^{-1}\exp\pigl(\beta \sum_{p \in C_2(B_N)}      \rho\bigl(d  \sigma(p)\bigr) + \kappa \sum_{e \in C_1(B_N)}   \rho \bigl( \sigma(e)\bigr)\pigr)  ,
\end{equation} 
where \( Z_{N,\beta,\kappa}^{(U)} \) is a normalizing constant. We let \( \mathbb{E}^{(U)}_{N,\beta,\kappa} \) denote the corresponding expectation.

The following lemma, which is considered well-known in the physics literature, will be crucial in the analysis of the lattice Higgs model.

\begin{lemma}[Corollary 2.17 in~\cite{f2022b}]\label{lemma: unitary gauge}
    Let \( \beta \geq 0,\) \( \kappa \geq 0 \), and let \( \gamma \) be a path in \( C_1(B_N) \). Then
    \begin{equation*}
        \mathbb{E}_{N,\beta,\kappa}\bigl[W_\gamma(\sigma,\phi)\bigr] =
        \mathbb{E}^{(U)}_{N,\beta,\kappa}\bigl[W_\gamma(\sigma,1)\bigr] =
        \mathbb{E}^{(U)}_{N,\beta,\kappa}\pigl[\rho\bigl(\sigma(\gamma)\bigr)\pigr]. 
    \end{equation*} 
\end{lemma}

The main idea of the proof of Lemma~\ref{lemma: unitary gauge} is to perform a change of variables, where we for each pair \( (\sigma,\phi) \) apply the gauge transformation \( \tau_{-\phi}, \) thus mapping \( \phi \) to \( 0 \). This maps \( \mu_{N,\beta,\kappa}\) to \( \mu^{(U)}_{N,\beta,\kappa} \) and \( W_\gamma(\sigma,\phi)\) to \( W_\gamma(\tau_\phi \sigma,1).\) Using the Poincaré lemma (see, e.g.~\cite[Lemma 2.2]{c2019}), one can show that this map is \( k \)-to-\(1\) for some \( k \in \mathbb{N} \) that depends on \( N, \) and from this the conclusion of the lemma follows.  
After having applied this gauge transformation, we say we are working in \emph{unitary gauge}.

We point out that Lemma~\ref{lemma: unitary gauge} holds more generally also for 1-chains in the sense that for any 1-chain \( c \in C_1(B_N), \) one has 
    \begin{equation*}
        \mathbb{E}_{N,\beta,\kappa}\pigl[ \rho \bigl( \sigma(c) -\phi(\partial c)\bigr) \pigr] = 
        \mathbb{E}^{(U)}_{N,\beta,\kappa}\pigl[\rho\bigl(\sigma(c)\bigr)\pigr]. 
    \end{equation*} 
    However, this more general version of Lemma~\ref{lemma: unitary gauge} will not be used in this paper.

With the current section in mind, we will work with \( \sigma \sim \mu_{N,\beta, \kappa}^{(U)} \) rather than with \( {(\sigma,\phi) \sim \mu_{N,\beta,\kappa}}\) throughout the rest of this paper, together with the observable
\begin{equation*}
    W_\gamma(\sigma) \coloneqq W_\gamma(\sigma,1) = \prod_{e \in \gamma} \rho\bigl(\sigma(e)\bigr) = \rho(\sigma(\gamma)).
\end{equation*}

\subsection{Existence of the infinite volume limit}\label{sec: ginibre}

In this section, we recall a result which shows existence and translation invariance of the infinite volume limit \( \langle W_\gamma(\sigma,\phi) \rangle_{\beta,\kappa} \) defined in the introduction. This result is well-known, and is often mentioned in the literature as a direct consequence of the Ginibre inequalities. A full proof of this result in the special case \( \kappa = 0 \) was included in~\cite{flv2020}, and the general case can be proven completely analogously, hence we omit the proof here and the refer the reader to~\cite{flv2020}. 

\begin{proposition}\label{proposition: limit exists}
    Let \( G = \mathbb{Z}_n \), \( M \geq 1 \), and let \( f \colon \Omega^1(B_M,G) \times \Omega^0(B_M,G) \to \mathbb{R}\).
    For \( M' \geq M \), we abuse notation and let \( f \) denote the natural extension of \( f \) to \( \Omega^1(B_{M'},G) \times \Omega^0(B_{M'},G) \), i.e., the unique function such that \( f(\sigma) = f(\sigma|_{C_1(B_M)},\phi|_{C_0(B_M)}) \) for all \( (\sigma,\phi) \in \Omega^1(B_{M'},G) \times \Omega^0(B_{M'},G) \).
    Further, let \( \beta \in [0, \infty] \) and \( \kappa \geq 0 \). Then the limit \( \langle  f(\sigma,\phi)\rangle_{\beta,\kappa} = \lim_{N \to \infty} \mathbb{E}_{N,\beta,\kappa} \bigl[ f(\sigma,\phi) \bigr] \) exists and is translation invariant. 
\end{proposition}

\subsection{The activity}\label{sec: activity}

For \( a\geq 0 \) and \( g \in G ,\) we set
\begin{equation*}
    \phi_{a}(g) \coloneqq  e^{a \real  (\rho(g)-\rho(0))} .
\end{equation*} 
Since \( \rho \) is unitary, for any \( g \in G \) we have \( \rho(g) = \overline{\rho(-g)} \), and hence \( \real  \rho(g) = \real  \rho(-g) \). In particular,  for any \( g \in G \) 
\begin{equation} \label{eq: phi is symmetric}
    \phi_{a}(g)
    =
    e^{ a (\real  \rho(g)-\rho(0))  }
    =
    e^{a (\real  \rho(-g)-\rho(0)) }
    =
    \phi_{a}(-g).
\end{equation}
For \( \sigma \in \Omega^1(B_N,G) \) and \( \beta,\kappa \geq 0 \) we define the \emph{activity} of $\sigma$ by
\begin{equation*}
    \phi_{\beta,\kappa}(\sigma) \coloneqq \prod_{p \in C_2(B_N)}\phi_\beta\bigl(d\sigma(p)\bigr) \prod_{e \in C_1(B_N)}\phi_\kappa\bigl(\sigma(e)\bigr).
\end{equation*}

Note that for \( \sigma \in \Omega^1(B_N,G), \) the probability measure corresponding to the Wilson action lattice gauge theory can be written an
\begin{equation}\label{eq: mubetakappaphi}
    \mu^{(U)}_{N,\beta,\kappa}(\sigma) = \frac{\phi_{\beta,\kappa} (\sigma)}{\sum_{\sigma \in \Omega^1(B_N,G)} \phi_{\beta,\kappa}(\sigma)}.
\end{equation}
Moreover, in the case  $G = \mathbb{Z}_2,$ for \( \sigma \in \Omega^1(B_N,\mathbb{Z}_2)\) we have
\begin{equation*}
    \begin{split}
        &\phi_{\beta,\kappa}(\sigma) = \prod_{p \in C_2(B_N)} \phi_{\beta}\bigl( d\sigma(p) \bigr)\prod_{e \in C_1(B_N)} \phi_{\kappa}\bigl( \sigma(e) \bigr) 
        \\&\qquad= \prod_{p \in C_2(B_N)} e^{-2\beta \mathbb{1}\bigl( d\sigma(p) = 1\bigr)}\prod_{e \in C_1(B_N)} e^{-2\kappa \mathbb{1}( \sigma(e) = 1)}
        \\&\qquad=
        e^{-2\beta \sum_{p \in C_2(B_N)} \mathbb{1}( d\sigma(p) = 1)}
        e^{-2\kappa \sum_{e \in C_1(B_N)} \mathbb{1}( \sigma(e) = 1)}
        \\&\qquad=
        e^{-2\beta |\support d\sigma|}e^{-2\kappa |\support \sigma|}.
    \end{split}
\end{equation*}
We note that, by definition, if \( \sigma,\eta \in \Omega^1(B_N,G),\) \( \support \eta \subseteq \support \sigma, \) \( \sigma|_{\support \eta} = \eta \)  and  \( (d\sigma)|_{\support d\eta} = d\eta,\) then \[ \phi_{\beta,\kappa} (\sigma) = \phi_{\beta,\kappa}(\eta) \phi_{\beta,\kappa}(\sigma-\eta).\]

\subsection{Additional notation}\label{section: notation}

Let \( D_0 = D_0(m) \) be a universal constant such that for any \( e \in C_1(B_N)^+ \) and any \( j \geq 1, \) there are at most \( D_0j^{m-1} \) positively oriented plaquettes at distance \( j \) from \( e. \)

When \( (\gamma_1^{(n)})_{n\geq 1} \) and \( (\gamma_2^{(n)})_{n\geq 1} \) are as in Figure~\ref{figure: U Wilson line}, we let
\begin{equation}
	\Gamma_n \coloneqq \bigl\{ 0,\gamma_1^{(n)},\gamma_2^{(n)},\gamma_1^{(n)}+\gamma_2^{(n)} \bigr\}.
\end{equation}

Let \( \mathcal{G}_0\) be the graph with vertex set \(  C_1(B_N)^+\) and an edge between two distinct edges \( e_1,e_2 \in C_1(B_N)^+ \) if \( \support  \partial e_1 \cap \support  \partial e_2 \neq \emptyset\) (written \( e_1 \sim e_2\)).
Note that any \( e \in C_1(B_N)^+ \) has degree at most \( M_0 \coloneqq 4m-1\) in \( \mathcal{G}_0.\)
When \( \gamma \in \Lambda_0,\) we let \( \mathcal{G}_0 (\gamma)\) be the subgraph of \( \mathcal{G}_0 \) induced by \( \support \gamma. \)
We say that a path \( \gamma \in \Lambda_0 \) is connected if  \( \mathcal{G}_1(\gamma) \) is a connected graph, and let
\begin{equation*}
	\Lambda_1 \coloneqq \bigl\{ \gamma \in \Lambda_0 \colon \gamma \text{ is connected} \bigr\}.
\end{equation*}

\section{The Higgs phase (\( \kappa\) large)}\label{section: higgs}

\subsection{A cluster expansion}\label{sec: cluster expansions higgs}

In this section we describe a cluster expansion for the lattice Higgs model on a finite box $B_N$ in the Higgs phase. The material here is for the most part well-known and is a natural special case of the cluster expansion as presented in~\cite{fv2017}. The expansion we use here is similar to the expansion in~\cite{fv2023, seiler} but uses polymers in the Higgs field instead of polymers in the gauge field. This is the reason that we need \( \kappa \geq \kappa_0^{\text{(Higgs)}}\) here instead of \( \beta \geq \beta_0 \) for some \( \beta_0>0\) as in~\cite{fv2023}.

Throughout this section and in the rest of the paper, we will assume that \( G = \mathbb{Z}_2.\)

\subsubsection{Polymers}\label{sec: the graphs}

Let  \( \mathcal{G}_1 \) be the graph with vertex set \( C_1(B_N)^+\) and an edge between two distinct vertices $e_1,e_2$ iff  $\support \hat \partial e_1 \cap \support \hat \partial e_2  \neq \emptyset$, i.e. if \( e_1\) and \( \pm e_2\) are both in boundary of some common plaquette.

Since any edge \( e \in C_1(B_N)^+\) in \( B_N\) is in the boundary of at most \( 2(m-1)\) plaquettes, and any such plaquettes has exactly three edges in its boundary that are not equal to \( e,\) it follows that there are at most 
	\begin{equation}\label{eq: M1}
		M_1 \coloneqq 3 \cdot 2(m-1) = 6(m-1)
	\end{equation} 
	edges \( e' \in C_1(B_N)^+\smallsetminus \{ e \} \) with \( {\support \hat \partial e \cap \support \hat \partial e' \neq \emptyset.}\) As a consequence, it follows that each vertex in \( \mathcal{G}_1 \) has degree at most \( M_1.\) 

The graph \( \mathcal{G}_1\) will be useful when we, in the following sections, describe the cluster expansion we will use in the Higgs phase. This graph is analog to the graph introduced in~\cite[Section 2.3]{fv2023} but has~\( C_1(B_N)^+\) as its vertex set instead of \( C_2(B_N)^+\) which was used in~\cite{fv2023}.

When \( \sigma \in \Omega^1(B_N,G),\) we let \( \mathcal{G}_1 (\sigma ) \) be the subgraph of \( \mathcal{G}_1 \) induced by \( (\support \sigma)^+. \)
We let \( \Lambda \) be the set of all \( \sigma \in \Omega^1(B_N,G)\) such that  \( \mathcal{G}_1 (\sigma )  \) has exactly one connected component. 
The spin configurations in \( \Lambda \) will be referred to as \emph{polymers}.

\subsubsection{Polymer interaction}

For \( \sigma,\sigma' \in \Lambda,\) we write \( \sigma \sim \sigma'\) if   \( \mathcal{G}_1(\sigma) \cup \mathcal{G}_1(\sigma')\) is a connected subgraph of \( \mathcal{G}_1.\)

In the notation of~\cite[Chapter 3]{fv2017}, the model given by~\eqref{eq: hat Z confinement} corresponds to a model of polymers with polymers described in~Section~\ref{sec: the graphs} and interaction function  \(  \iota(\sigma_1,\sigma_2) \coloneqq \zeta(\sigma_1,\sigma_2) +1 , \) where 
\begin{equation*}
	\zeta(\sigma_1,\sigma_2) \coloneqq  
	\begin{cases}  
		-1 &\text{if } \sigma_1,\sigma_2 \in \Lambda \text{ and } \sigma_1 \sim \sigma_2 \cr   
		0 &\text{else.} 
	\end{cases}
\end{equation*}

\subsubsection{Clusters of polymers}\label{section: clusters higgs}

Consider a multiset 
\begin{align*}
    &\mathcal{S} = \{ \underbrace{\eta_1,\dots, \eta_1}_{n_{\mathcal{S}}(\eta_1) \text{ times}}, \underbrace{\eta_2, \dots, \eta_2}_{n_{\mathcal{S}}(\eta_2) \text{ times}},\dots, \underbrace{\eta_k,\dots,\eta_k}_{n_{\mathcal{S}}(\eta_k) \text{ times}} \} = \{\eta_1^{n(\eta_1)}, \ldots, \eta_k^{n(\eta_k)}\}, 
\end{align*}
where \( \eta_1,\dots,\eta_k \in \Lambda  \) are distinct and $n(\eta)=n_{\mathcal{S}}(\eta)$ denotes the number of times $\eta$ occurs in~$\mathcal{S}$. Following \cite[Chapter 3]{fv2017}, we say that $\mathcal{S}$ is \emph{decomposable} if it is possible to partition $\mathcal{S}$ into disjoint multisets. That is, if there exist non-empty and disjoint multisets $\mathcal{S}_1,\mathcal{S}_2 \subset \mathcal{S}$ such that $\mathcal{S} = \mathcal{S}_1 \cup \mathcal{S}_2$ and such that for each pair $(\eta_1,\eta_2) \in \mathcal{S}_1 \times \mathcal{S}_2$,  $ \eta_1 \nsim \eta_2.$  
If $\mathcal{S}$ is not decomposable, we say that $\mathcal{S}$  is a \emph{cluster}. We stress that a cluster is unordered and may contain several copies of the same polymer. Given a cluster $\mathcal{S} $, we define
\begin{equation*}
    \|\mathcal{S} \|_1 = \sum_{\eta \in \Lambda } n_{\mathcal{S}} (\eta) \bigl| (\support \eta )^+ \bigr|,\quad 
    \|\mathcal{S} \|_2 = \sum_{\eta \in \Lambda } n_{\mathcal{S}} (\eta) \bigl| (\support d\eta )^+ \bigr| ,
\end{equation*}
\begin{equation*}
    n(\mathcal{S}) = \sum_{\eta \in \Lambda } n_{\mathcal{S}} (\eta),\quad \text{and} \quad
    \support \mathcal{S} = \bigcup_{\eta \in \mathcal{S}} \support \eta.
\end{equation*}
For a $1-$chain \( c \in C_1(B_N,\mathbb{Z}),\) we define 
\[ \mathcal{S}(c) = \sum_{\eta \in \Lambda } n_{\mathcal{S}}(\eta) \eta(c).\]

We let \( \Xi\) be the set of all clusters.

To simplify notation in the rest of this section, for \( e \in C_1(B_N)^+,\) \( \gamma \in C^1(B_N,\mathbb{Z}), \) and \( i ,j,k \geq 1 ,\) we define
\begin{equation*}
    \Xi_{i,j,k,e} \coloneqq \bigl\{ \mathcal{S} \in \Xi \colon 
    n(\mathcal{S}) = i,\, \| \mathcal{S}\|_1 = j,\, \| \mathcal{S}\|_2 = k,\,   e \in \support \mathcal{S} \bigr\},
\end{equation*}

\begin{equation*}
    \Xi_{i} \coloneqq \bigl\{ \mathcal{S} \in \Xi \colon 
    n(\mathcal{S}) = i \bigr\},
\end{equation*} 
\begin{equation*}
    \Xi_{i,j} \coloneqq \bigl\{ \mathcal{S} \in \Xi \colon 
    n(\mathcal{S}) = i,\, \|\mathcal{S}\|_1 = j \bigr\},
\end{equation*} 
\begin{equation*}
    \Xi_{i,j,e} \coloneqq \bigl\{ \mathcal{S} \in \Xi \colon 
    n(\mathcal{S}) = i,\, \|\mathcal{S}\|_1 = j, \, e \in \support \mathcal{S} \bigr\},
\end{equation*} 
and
\begin{equation*}
    \Xi_{i,j,\gamma} \coloneqq \bigl\{ \mathcal{S} \in \Xi \colon 
    n(\mathcal{S}) = i,\, \|\mathcal{S}\|_1 = j, \, \mathcal{S}(\gamma) \neq 0 \bigr\}.
\end{equation*} 
Further, we let
\begin{equation*}
    \Xi_{i^+ } \coloneqq \bigl\{ \mathcal{S} \in \Xi \colon 
    n(\mathcal{S}) \geq i\},
\end{equation*} 
and define \( \Xi_{i^+,j}, \) \( \Xi_{i,j^+}, \) \( \Xi_{i^+,j^+}, \) etc. analogously.

We note that the sets defined above depend on \( N\) but we usually suppress this in the notation. When we want to emphasize this dependence, we write \( \Xi(B_N), \) \( \Xi_i(B_N),\) \( \Xi_{i,j}(B_N), \) etc.

The following lemma gives an upper bound on the number of clusters of a given size and with a given edge in its support.

\begin{lemma}\label{lemma: Xi1ke bound}
For any \( k \geq 1\) and \( e \in C_1(B_N)^+,\) we have \( |\Xi_{1,k,e}| \leq M_1^{2k-2}.\)
\end{lemma}

\begin{proof}
    Let \( k \geq 1\) and \( e \in C_1(B_N).\)
    Let \( \mathcal{P}\) be the set of all paths in \( \mathcal{G}_1 \) that starts at \( e \) and have length \( 2k-1.\) Since each vertex in \( \mathcal{G}_1\) has degree at most \(M_1,\) we have  \( |\mathcal{P}| \leq M_1^{2k-2}.\) 

    For \( \{ \eta \} \in \Xi_{1,k,e},\) let \( G_\eta\) be the subgraph of \( \mathcal{G}_1 \) induced by the set \( (\support \eta)^+.\)
    Then \( G_\eta\) is connected, and hence \( G_\eta \) has a spanning path \( T_\eta \in \mathcal{P}\) of length \( 2k-1=2\bigl| (\support \eta)^+ \bigr|-1\) which starts at \( e .\) 
    Since the map \( \eta \mapsto T_\eta\) is an injective map from \( \Xi_{1,k,e}\) to \( \mathcal{P}\) and \( |\mathcal{P}| \leq  M_1^{2k-2},\) the desired conclusion immediately follows.
\end{proof}

\subsubsection{The activity of clusters}
  
  We extend the notion of activity from Section~\ref{sec: activity} to clusters \( \mathcal{S} \in \Xi\) by letting
\begin{equation*}
    \phi_{\beta,\kappa}(\mathcal{S}) = \prod_{\eta \in \mathcal{S}} \phi_{\beta,\kappa}(\eta)^{n_{\mathcal{S}}(\eta)} = e^{-4\beta \| \mathcal{S}\|_2-4\kappa \| \mathcal{S}\|_1}.
\end{equation*}

\subsubsection{Ursell functions}\label{sec: ursell 1}
In the cluster expansion in~\cite{fv2017}, which is valid for many different spin models, special functions known as Ursell functions play an important role. We define the Ursell function relevant for our setting below.


For \( k \geq 1,\) we let \( \mathcal{G}^k\) be the set of all connected graphs \( G\) with vertex set \( V(G) = \{ 1,2,\dots, k\}.\) Whenever \( G\) is a graph, we let \( E(G)\) be the (undirected) edge set of \( G.\)

\begin{definition}[The Ursell functions]\label{def:ursell}
    For \( k \geq 1 \) and  \( \eta_1,\eta_2,\dots , \eta_k \in \Lambda ,\) we let
    \begin{equation*}
        U(  \eta_1, \ldots, \eta_k  ) \coloneqq \frac{1}{k!} \sum_{G \in \mathcal{G}^k} (-1)^{|E(\mathcal{G} )|} \prod_{(i,j) \in E(\mathcal{G})} \mathbb{1}( \eta_i \sim \eta_j).
    \end{equation*} 
    Note that this definition is invariant under permutations of the polymers \( \eta_1,\eta_2,\dots,\eta_k.\)

   For \( \mathcal{S} \in \Xi_k,\) and any enumeration \( \eta_1,\dots, \eta_k \) (with multiplicities) of the polymers in \( \mathcal{S},\) we define
   \begin{equation}\label{eq: ursell functions} 
        U(\mathcal{S}) = k! \, U(\eta_1,\dots, \eta_k).
   \end{equation}
\end{definition} 
Note that for any \( \mathcal{S} \in \Xi_1,\) we have \( U(\mathcal{S})=1,\) and for any \( \mathcal{S} \in \Xi_2,\)  we have
\( U(\mathcal{S}) = -1. \)

\subsubsection{Cluster expansion of the partition function}
The partition function for the abelian lattice Higgs model with parameters \( \beta,\kappa \geq 0 \) is given by
\[
Z^{(U)}_{N,\beta,\kappa} =  \sum_{\sigma \in \Omega^1(B_N,G)} \phi_{\beta,\kappa}(\sigma). 
\]
Since \( N \) is finite, this is a finite sum.
An alternative representation of $ Z_{\beta,\kappa,N}$ is given by the \emph{cluster partition function} which is defined by the following (formal) expression:
\begin{equation}\label{eq: spin-cluster-partition}
     Z_{N,\beta,\kappa}^{*} = \exp\left(
    \sum_{\mathcal{S}\in \Xi} \Psi_{\beta,\kappa}(\mathcal{S})\right), 
\end{equation}
where for \( \mathcal{S} \in \Xi, \) we define
\begin{equation*}
    \begin{split}
        \Psi_{\beta,\kappa}(\mathcal{S}) \coloneqq U(\mathcal{S})
        \phi_{\beta,\kappa}(\mathcal{S}) 
    \end{split}
\end{equation*}
and $U$ is the Ursell function as defined in Section~\ref{sec: ursell 1}.

It is not obvious that the series in the exponent of~\eqref{eq: spin-cluster-partition} is convergent but this follows from the next lemma, assuming \( \kappa \) is sufficiently large, and we verify below that in this case $\log Z_{N,\beta,\kappa}^{(U)} = \log Z_{N,\beta,\kappa}^*$.  
In this lemma, the following constants will be used. Recalling the definition of \( M_1 \) from~\eqref{eq: M1}, we define
\begin{equation}\label{eq: kappa0} 
    \alpha = \alpha^{(\textrm{Higgs})}\coloneqq             \argmin_{\alpha'\in (0,1)}\frac{\log (M_1^2+1/\alpha' )}{4(1-\alpha')} 
    \quad \text{and} \quad 
    \kappa_0^{(\textrm{Higgs})} \coloneqq \frac{\log (M_1^2+1/\alpha )}{4(1-\alpha)}.
\end{equation}
 
\begin{lemma}\label{lemma: the assumption on a}
    Let  $\beta\geq 0$ and \( \kappa  > \kappa_0^{(\textrm{Higgs})}  . \) Then, for any \( \eta \in \Lambda, \) we have 
    \[
    \sum_{\mathcal{S} \in \Xi \colon \eta \in \mathcal{S}} \Psi_{\beta,\kappa}(\mathcal{S})
    \leq \bigl(  \phi_{0,\kappa}(\eta) \bigr)^{-\alpha}\phi_{\beta,\kappa}(\eta).
    \]
    Moreover, the series in \eqref{eq: spin-cluster-partition} is absolutely convergent. 
\end{lemma}
\begin{proof}
Let $\alpha >0$ be such that \(\kappa  \geq \frac{\log (M_1^2+1/\alpha )}{4(1-\alpha)} .\) We will prove that for each \( \eta \in \Lambda \) we have 
    \begin{equation*}
        \sum_{\eta' \in \Lambda } |\phi_{\beta,\kappa}(\eta')| e^{\alpha| \support \eta'|} \mathbb{1}(\eta \sim \eta') \leq \alpha   \kappa |\support \eta|   = \alpha\log \phi_{0,\kappa}(\eta) .
    \end{equation*}
    Given this, the conclusion of the lemma follows from~\cite[Theorem~5.4]{fv2017} choosing the function $a(\eta) \coloneqq  -\alpha\log \phi_{0,\kappa}(\eta)$. 

    By the choice of \( \alpha,\) we have
    \begin{equation*}
        M_1^2 e^{-4\kappa (1-\alpha)}<1 
    \end{equation*}
    and
    \begin{equation*}
        \frac{e^{-4\kappa  (1-\alpha)}}{1-M_1^2e^{-4\kappa (1-\alpha)}} \leq \alpha 
        .
    \end{equation*}
    %
    \color{black} 
    Thus, for any \( \eta \in \Lambda ,\) we have 
    \begin{align*}
        &\sum_{\eta' \in \Lambda } \phi_{\beta,\kappa}(\eta') e^{a(\eta')} \mathbb{1}(\eta \sim \eta') 
        =
        \sum_{\eta' \in \Lambda  \colon \eta \sim \eta'} \phi_{\beta,\kappa}(\eta')   \phi_{0,\kappa}(\eta')^{-\alpha} 
        \leq 
        \sum_{\eta' \in \Lambda  \colon \eta \sim \eta'} \phi_{0,\kappa}(\eta')^{1-\alpha} 
        \\&\qquad \leq
        \sum_{e \in (\support \eta)^+} \sum_{ \{ \eta'\} \in \Xi_{1,1^+,e}} \phi_{0,\kappa}(\eta')^{1-\alpha} = 
        \sum_{e \in (\support \eta)^+} \sum_{j=1}^\infty \left| \Xi_{1,j,e} \right| \bigl( e^{-4\kappa j}\bigr)^{1-\alpha}
        .
    \end{align*} 
    Using Lemma~\ref{lemma: Xi1ke bound}, we bound the right hand side of the previous equation from above by
    \begin{align*}
        &\bigl|(\support \eta)^+ \bigr| \sum_{j=1}^\infty M_1^{2j-2}e^{-4\kappa j (1-\alpha)}
        =
        \bigl|(\support \eta)^+ \bigr| \frac{e^{-4\kappa  (1-\alpha)}}{1-M_1^2e^{-4\kappa (1-\alpha)}}.
    \end{align*}
    The desired conclusion now follows from the choice of \( \alpha.\)
\end{proof}

\begin{lemma}\label{lem:partition-functions}
Let \( \beta \geq 0 \) and \(\kappa > \kappa_0^{(\textrm{Higgs})} .\) Then
\begin{equation}\label{eq: log Z expansion}
    \log Z^{(U)}_{N,\beta,\kappa} = \log Z_{N,\beta,\kappa}^{*} = \sum_{\mathcal{S} \in \Xi} \Psi_{\beta,\kappa}(\mathcal{S}).
\end{equation} 
\end{lemma}

\begin{proof}
    The set $\Omega^1(B_N, G)$ is in bijection with the set of subsets of $\Lambda $ with the property that the polymers in each subset have pairwise disjoint supports and that their differentials have pairwise disjoint supports. Therefore, for any $\beta$ and \( \kappa\) we can write
    \[
        Z^{(U)}_{N,\beta,\kappa} = \sum_{\Lambda' \subset \Lambda} \phi_{\beta}(\Lambda') \prod_{\{\nu,\nu'\} \subset \Lambda'} \mathbb{1}(\support \nu \cap \support \nu' = \emptyset).
    \] 

    On the other hand, if $ \kappa > \kappa_0^{(\textrm{Higgs})}$, we can apply Proposition~5.3 of \cite{fv2017} to see that the right-hand side in the last display equals $\log Z_{N,\beta,\kappa}^*$ as defined in \eqref{eq: spin-cluster-partition}. From this the desired conclusion follows.
\end{proof}
 
All results in this section will assume that $\beta \geq 0 $ and \( \kappa \geq \kappa_0^{(\textrm{Higgs})},\) and using this assumption, we from now on write $Z^{(U)}_{N,\beta,\kappa}$ also for the cluster partition function $Z_{N,\beta,\kappa}^*$.

\subsubsection{Cluster expansion of Wilson line observables}

Consider the weighted cluster partition function 
\begin{equation}\label{eq:dec5.1}
   Z_{N,\beta,\kappa}^{(U)}[\gamma] \coloneqq \exp\left(\sum_{\mathcal{S} \in \Xi} \Psi_{\beta,\kappa,\gamma}(\mathcal{S})\right), 
\end{equation}
where
\[
    \Psi_{\beta,\kappa,\gamma}(\mathcal{S}) \coloneqq U(\mathcal{S}) \phi_{\beta,\kappa}(\mathcal{S})  
    \rho\bigl( \mathcal{S}(\gamma) \bigr). 
\]
The series on the right-hand side is absolutely convergent when $\beta\geq 0$ and \( \kappa  \geq \kappa_0^{(\textrm{Higgs})} \) by the proof of Lemma~\ref{lemma: the assumption on a} since $ \pigl|\rho\bigl(\mathcal{S}(\gamma)\bigr)\pigr| = 1$ for each \( \mathcal{S} \in \Xi.\) As in the proof of Lemma~\ref{lem:partition-functions}, using \cite[Proposition 5.3]{fv2017}, replacing the weight $\phi_\beta(\mathcal{S})$ by $\phi_\beta(\mathcal{S})\rho\bigl(\mathcal{S}(\gamma)\bigr)$, we have
\[
    \log Z^{(U)}_{\beta,\kappa,N}[\gamma] = \sum_{\sigma \in \Omega^1(B_N,G)}  \phi_{\beta,\kappa}(\sigma) \rho\bigl(\sigma(\gamma) \bigr).
\]
Combining~\eqref{eq:dec5.1} and~Lemma~\ref{lem:partition-functions} we obtain the following result, which give us an alternative expression for \(\log \mathbb{E}^{(U)}_{N,\beta,\kappa}[ W_\gamma ].\)

\begin{proposition}\label{proposition: the cluster expansion}
    Let $\beta \geq 0$ and \( \kappa > \kappa_0^{(\textrm{Higgs})}.\) 
    %
    Then for all $N$ such that $\support \gamma \subset B_N$,
    \begin{equation}\label{eq: expansion equation}
        -\log \mathbb{E}^{(U)}_{N,\beta,\kappa}[ W_\gamma ] = \sum_{\mathcal{S} \in \Xi} \bigl( \Psi_{\beta,\kappa}(\mathcal{S})-\Psi_{\beta,\kappa,\gamma}(\mathcal{S})\bigr) = 
        \sum_{\mathcal{S} \in \Xi} \Psi_{\beta,\kappa}(\mathcal{S})\pigl( 1-\rho \bigl(\mathcal{S}(\gamma)\bigr) \pigr) .
    \end{equation}
\end{proposition}

\begin{proof}
    Using~\eqref{eq:dec5.1} and~Lemma~\ref{lem:partition-functions}   we conclude that
    \[
        \log \mathbb{E}^{(U)}_{N,\beta,\kappa}[W_\gamma] = \log \frac{Z^{(U)}_{N,\beta,\kappa}[\gamma]}{Z^{(U)}_{N,\beta,\kappa}},
    \]
    which is what we wanted to prove.
\end{proof}

\begin{remark}
   Notice that Proposition~\ref{proposition: the cluster expansion} implies that $\mathbb{E}^{(U)}_{N,\beta,\kappa}[W_\gamma] \in (0,1]$ when \( \beta \geq 0\) and \( \kappa \geq \kappa_0^{(\textrm{Higgs})}. \) This fact is not a priori clear since $W_\gamma \in \{-1,1 \}$ for every $\sigma \in \Omega^1(B_N,\mathbb{Z_2}) $.
\end{remark}

\subsubsection{Cluster expansion of the Marcu-Fredenhagen ratio}\label{section: FK cluster}

In this section, we assume that two non-trivial paths \( \gamma_1\) and \( \gamma_2\) are given, with disjoint support and with the same endpoints so that \( \gamma \coloneqq \gamma_1 + \gamma_2\) is a loop. The goal of this section is to use the cluster expansions of~\eqref{eq: spin-cluster-partition} and~\eqref{eq:dec5.1} to give an expression of the Marcu-Fredenhagen ratio which uses the cluster expansions.

To simplify notation, we define 
\[
\rho_{N}(\gamma_1,\gamma_2) \coloneqq   \frac{\mathbb{E}^{(U)}_{N,\beta,\kappa} [W_{\gamma_1}] \,\mathbb{E}^{(U)}_{N,\beta,\kappa} [W_{\gamma_2} ]}{\mathbb{E}^{(U)}_{N,\beta,\kappa} [ W_{\gamma_1+\gamma_2} ] }.
\]
We note that by the Ginibre inequality (see~\cite[Section 2.6]{f2022b}), the limit 
\[
    \rho(\gamma_1,\gamma_2) \coloneqq \lim_{N \to \infty} \rho_{N}(\gamma_1,\gamma_2)
\]
exists and is translation invariant.

\begin{lemma}\label{lemma: ratio sum}
     Let $\beta \geq 0$ and \( \kappa > \kappa_0^{(\textrm{Higgs})}.\) Then
     \begin{equation}\label{eq: ratio sum}
        \log \rho_N(\gamma_1,\gamma_2)
        =
        -
        4\sum_{\mathcal{S} \in \Xi } \Psi_{\beta,\kappa}(\mathcal{S}) \mathbb{1} \pigl(  \mathcal{S}(\gamma_1) = \mathcal{S}(\gamma_2)  =  1 \pigr) .
    \end{equation}
\end{lemma}

\begin{proof} 
    Using~Proposition~\ref{proposition: the cluster expansion}, we obtain
\begin{align*}
    &\log \rho_{N}(\gamma_1,\gamma_2)
    =
    \log \mathbb{E}^{(U)}_{N,\beta,\kappa} [W_{\gamma_1}]  + \log \mathbb{E}^{(U)}_{N,\beta,\kappa} [W_{\gamma_2}] - \log \mathbb{E}^{(U)}_{N,\beta,\kappa} [W_{\gamma}] 
    \\&\qquad=
    \sum_{\mathcal{S} \in \Xi} \Psi_{\beta,\kappa}(\mathcal{S}) \Bigl( \rho \bigl(\mathcal{S}(\gamma_1+\gamma_2)\bigr)  -\rho \bigl(\mathcal{S}(\gamma_1)\bigr)  -\rho \bigl(\mathcal{S}(\gamma_2)\bigr) +1\Bigr)
        \\&\qquad=  \sum_{\mathcal{S} \in \Xi   }  \Psi_{\beta,\kappa}(\mathcal{S}) \Bigl( \rho \bigl(\mathcal{S}(\gamma_1)\bigr)\rho \bigl(\mathcal{S}(\gamma_2)\bigr)  -\rho \bigl(\mathcal{S}(\gamma_1)\bigr)  -\rho \bigl(\mathcal{S}(\gamma_2)\bigr) +1\Bigr)
        \\&\qquad=  
        \sum_{\mathcal{S} \in \Xi   } \Psi_{\beta,\kappa}(\mathcal{S})  \pigl( \rho \bigl(\mathcal{S}(\gamma_1)\bigr) - 1 \pigr) \pigl( \rho \bigl(\mathcal{S}(\gamma_2)\bigr) - 1 \pigr) 
        \\&\qquad= 
        -
        4\sum_{\mathcal{S} \in \Xi  } \Psi_{\beta,\kappa}(\mathcal{S}) \mathbb{1} \pigl(  \mathcal{S}(\gamma_1) = \mathcal{S}(\gamma_2)  =  1 \pigr) .
\end{align*}
This concludes the proof.
\end{proof}

\subsubsection{Cluster expansion of the covariance of two Wilson lines}

In this section, we use Lemma~\ref{lemma: ratio sum} to give an upper bound of the covariance of two Wilson lines as a sum over clusters. This result, Lemma~\ref{lemma: covariance expansion} below, will be the main ingredient in the proof of Theorem~\ref{theorem: correlation}.

\begin{lemma}\label{lemma: covariance expansion}
    Let $\beta \geq 0$ and \( \kappa > \kappa_0^{(\textrm{Higgs})}.\) Further, let \( \gamma_1\) and \( \gamma_2\) be two paths with disjoint support. Then
    \begin{equation}\label{eq: covariance expansion}
        \bigl| \mathbb{E}^{(U)}_{N,\beta,\kappa}[W_{\gamma_1+\gamma_2}] -\mathbb{E}^{(U)}_{N,\beta,\kappa}[W_{\gamma_1}] \mathbb{E}^{(U)}_{N,\beta,\kappa}[W_{\gamma_2}]  \bigr|
        \leq 
        4
        \Bigl| \sum_{\mathcal{V}\in \Xi} \Psi_\beta(\mathcal{V}) \mathbb{1}\bigl(\mathcal{S}(\gamma_1) = \mathcal{S}(\gamma_2) = 1\bigr)
        \Bigr| .
    \end{equation} 
\end{lemma}

\begin{proof}
    Since the function \( x \mapsto e^x \) is convex for all \( x \in \mathbb{R},\) we have
    \begin{align*}
        &\bigl| \mathbb{E}^{(U)}_{N,\beta,\kappa}[W_{\gamma_1+\gamma_2}] -\mathbb{E}^{(U)}_{N,\beta,\kappa}[W_{\gamma_1}] \mathbb{E}^{(U)}_{N,\beta,\kappa}[W_{\gamma_2}]  \bigr|
        \\&\qquad\leq 
        \bigl| \log \mathbb{E}^{(U)}_{N,\beta,\kappa}[W_{\gamma_1+\gamma_2}] -\log  \mathbb{E}^{(U)}_{N,\beta,\kappa}[W_{\gamma_1}] -\log \mathbb{E}^{(U)}_{N,\beta,\kappa}[W_{\gamma_2}]  \bigr|
        =
        \bigl|\log \rho_N(\gamma_1,\gamma_2)\bigr|.
    \end{align*}
    Using Lemma~\ref{lemma: ratio sum}, the desired conclusion immediately follows.
\end{proof}

One interpretation of the right-hand-side of~\eqref{eq: covariance expansion} is that a cluster \( \mathcal{S} \in \Xi\)  makes a non-zero contribution to the covariance on the left-hand-side of~\eqref{eq: covariance expansion} only if its support connects the two paths \( \gamma_1\) and \( \gamma_2.\)

\subsection{Upper bound for clusters}\label{sec: upper bounds higgs}
In this section we give upper bounds on sums over the activity of sets of clusters which naturally arise in the proofs of Theorem~\ref{theorem: mf ratio Higgs} and Theorem~\ref{theorem: correlation}. 

\begin{lemma}\label{lemma: general Cbeta}
    Let \( \beta \geq 0 \) and  $\kappa > \kappa_0^{(\textrm{Higgs})},$ and let \( e \in C_1(B_N)^+.\) Then 
    \begin{equation*}\label{eq: general Cbeta}
        \sum_{\mathcal{S} \in \Xi_{1^+,1^+,e}}
        \bigl| \Psi_{\beta,\kappa}(\mathcal{S}) \bigr|
        \leq
        \frac{ e^{-(2\kappa-\alpha) 2}}{ 1-4M_1^2 e^{-2(2\kappa-\alpha) }},
    \end{equation*}
    where \( \alpha\) is as in~\eqref{eq: kappa0}.
\end{lemma}

\begin{proof}
    By Lemma~\ref{lemma: the assumption on a}, for any \( \eta \in \Lambda,\) we have 
    \begin{equation*}
        \sum_{\mathcal{S} \in \Xi \colon \eta \in \mathcal{S}} \bigl|\Psi_{\beta,\kappa}(\mathcal{S}) \bigr|\leq  \phi_{\beta,\kappa}(\eta) \bigl( \phi_{0,\kappa}(\eta) \bigr)^{-\alpha},
    \end{equation*}
    and hence
    \begin{equation}\label{eq: b1}
        \begin{split}
            &\sum_{\mathcal{S} \in \Xi_{1^+,1^+,e}}
            \bigl| \Psi_{\beta,\kappa}(\mathcal{S}) \bigr|
            \leq 
            \sum_{\{ \eta \} \in \Xi_{1,1+,e}}
            \sum_{\mathcal{S} \in \Xi \colon \eta \in \mathcal{S}} \bigl| \Psi_{\beta,\kappa}(\mathcal{S}) \bigr|
            \leq 
            \sum_{\{ \eta \} \in \Xi_{1,1+,e}}
            \phi_{\beta,\kappa}(\eta) \bigl( \phi_{0,\kappa}(\eta) \bigr)^{-\alpha},
        \end{split}
    \end{equation}
    By Lemma~\ref{lemma: Xi1ke bound}, we have
    \begin{equation}\label{eq: b2}
        \begin{split}
            &
            \sum_{\{ \eta \} \in \Xi_{1,1+,e}}
            \phi_{\beta,\kappa}(\eta) \bigl( \phi_{0,\kappa}(\eta) \bigr)^{-\alpha}
            \leq
            \sum_{j=1}^{\infty}
            |\Xi_{1,j,e}|
            e^{-2(2\kappa-\alpha)j}
            =
            \sum_{j=1}^{\infty}
            (2M_1)^{2j-2}
            e^{-(2\kappa-\alpha) 2j}
            .
        \end{split}
    \end{equation}
    Combining~\eqref{eq: b1} and~\eqref{eq: b2}, the desired conclusion immediately follows.
\end{proof}

\begin{lemma}\label{lemma: new upper bound}
    Let \( \beta \geq 0 \) and \( \kappa > \kappa_0^{(\textrm{Higgs})}.\) Further, let \( k \geq 1\) and \( e \in C_1(B_N).\)  
    Then, for any \( \varepsilon>0,\) we have
    \begin{equation*}
        \begin{split}
            & \sum_{\mathcal{S} \in \Xi_{1+,k+,e}} \bigl| \Psi_{\beta,\kappa}(\mathcal{S}) \bigr|   
            \leq  
            4^{-1}C_{\varepsilon}e^{-4k(\kappa-\kappa_0^{(\textrm{Higgs})}-\varepsilon) },
        \end{split}
    \end{equation*} 
    where \( C_{\varepsilon}\) is defined by
    \begin{equation}\label{eq: cbeta*}
        C_{\varepsilon} \coloneqq 4\sup_{N \geq 1} \sup_{e \in C_1(B_N)}\sum_{\mathcal{S} \in \Xi_{1+,1+,e}} \bigl| \Psi_{0,\kappa_0^{(\textrm{Higgs})}+\varepsilon}(\mathcal{S}) \bigr| < \frac{ 4e^{-(2\kappa-\alpha) 2}}{ 1-4M_1^2 e^{-2(2\kappa-\alpha) }}.
    \end{equation}
\end{lemma}

\begin{proof}
    Let \( \varepsilon>0. \) From Lemma~\ref{lemma: general Cbeta} we immediately get the upper bound on \( C_{\varepsilon} \) in~\eqref{eq: cbeta*}, and hence \( C_{\varepsilon} \)  is well defined.
    
    For any \( \mathcal{S}\in \Xi,\) we have \( \phi_{\beta,\kappa}(\mathcal{S}) = e^{-2\beta\|\mathcal{S}\|_2 - 2\kappa \|\mathcal{S}\|_1}\) and \( \Psi_{\beta,\kappa}(\mathcal{S}) = U(\mathcal{S}) \psi_{\beta,\kappa}(\mathcal{S}),\) where \( U(\mathcal{S})\) does not depend on \( \beta \) and \( \kappa,\) and hence \[
    \Psi_{\beta,\kappa}(\mathcal{S}) = e^{-2(\kappa-\kappa_0^{(\textrm{Higgs})}-\varepsilon)\| \mathcal{S}\|_1} \Psi_{\beta,\kappa_0^{(\textrm{Higgs})}+\varepsilon}(\mathcal{S}). \]
    Using this observation, we obtain
    \begin{equation*}
        \begin{split}
            & \sum_{\mathcal{S} \in \Xi_{1+,k+,e}} \bigl| \Psi_{\beta,\kappa}(\mathcal{S}) \bigr|   
            \leq  
            e^{-4(\kappa-\kappa_0^{(\textrm{Higgs})}-\varepsilon) k}\sum_{\mathcal{S} \in \Xi_{1+,k+,e}} \bigl| \Psi_{\beta,\kappa_0^{(\textrm{Higgs})}+\varepsilon}(\mathcal{S}) \bigr|     
            \\&\qquad\leq  
            e^{-4(\kappa-\kappa_0^{(\textrm{Higgs})}-\varepsilon) k}\sum_{\mathcal{S} \in \Xi_{1+,1+,e}} \bigl| \Psi_{\beta,\kappa_0^{(\textrm{Higgs})}+\varepsilon}(\mathcal{S}) \bigr|.
        \end{split}
    \end{equation*} 
    This concludes the proof.
\end{proof}

\subsection{Proof of Theorem~\ref{theorem: correlation}}\label{sec: proof of decay}

In this section we give a proof of~Theorem~\ref{theorem: correlation}, which gives an exponential upper bound on the decay of correlations.

\begin{proof}[Proof of Theorem~\ref{theorem: correlation}] 
    If \( \mathcal{S} \in \Xi\) is such that \( \mathcal{S}(\gamma_1) = \mathcal{S}(\gamma_2) = 1, \) then there is  some edge \( e \in \support \gamma_1\) such that \( \mathcal{S} \in \Xi_{1^+,\dist(\gamma_1,\gamma_2),e}.\) Using Lemma~\ref{lemma: covariance expansion}, it follows that  
    \begin{align*} 
        &\bigl| \mathbb{E}^{(U)}_{N,\beta,\kappa}[W_{\gamma_1+\gamma_2}] -\mathbb{E}^{(U)}_{N,\beta,\kappa}[W_{\gamma_1}] \mathbb{E}^{(U)}_{N,\beta,\kappa}[W_{\gamma_2}]  \bigr|
        \leq 
        4
        \Bigl| \sum_{\mathcal{S}\in \Xi} \Psi_{\beta,\kappa}(\mathcal{S}) \mathbb{1}\bigl(\mathcal{S}(\gamma_1) = \mathcal{S}(\gamma_2) = 1\bigr)
        \Bigr|
        \\&\qquad\leq
        4
        \sum_{e \in \support \gamma} \sum_{\mathcal{S}\in \Xi_{1^+,\dist(\gamma_1,\gamma_2),e}} \bigl| \Psi_{\beta,\kappa}(\mathcal{S}) 
        \bigr|. 
    \end{align*} 
    Using~Lemma~\ref{lemma: new upper bound}, the desired conclusion immediately follows. 
\end{proof}

\subsection{Proof of Theorem~\ref{theorem: mf ratio Higgs}}\label{section: proof of main result higgs}

Before giving a proof of Theorem~\ref{theorem: mf ratio Higgs} we state and prove a useful lemma.  

\begin{lemma}\label{lemma: gamma0 tail bound}
    In the setting of Theorem~\ref{theorem: mf ratio Higgs}, for any \( n \geq 1 \) and \( k \geq 1,\) we have
    \begin{align*}
        &\sum_{\mathcal{S} \in \Xi_{1^+,k^+}  } \Bigl| \Psi_{\beta,\kappa}(\mathcal{S}) \mathbb{1} \pigl(  \mathcal{S}(\gamma_1^{(n)}) = \mathcal{S}(\gamma_2^{(n)})  =  1 \pigr)\Bigr| 
        \\&\qquad\leq   
        2^{-1}C_{\varepsilon} \sum_{j=1}^\infty   e^{-4\max(j,k)(\kappa-\kappa_0^{(\textrm{Higgs})}-\varepsilon) }
        +
        \max(T_n-2R_n,0)
        4^{-1}C_{\varepsilon}e^{-4\max(2R_n,k)(\kappa-\kappa_0^{(\textrm{Higgs})}-\varepsilon) }
    \end{align*}
    where \( C_\varepsilon \) is defined in~\eqref{eq: cbeta*}.
\end{lemma}

\begin{proof}
    Let \( k \geq 1,\) \( n \geq 1,\) and let \( \gamma_1 = \gamma_1^{(n)}\) and \( \gamma_2 = \gamma_2^{(n)}.\)
    Further, let be such that \( \mathcal{S} \in \Xi_{1^+,k+} \) and \( \mathcal{S}(\gamma_1) = \mathcal{S}(\gamma_2) = 1.\) Since \( \mathcal{S}(\gamma_1) = 1,\) there is \( e \in \support \gamma_1 \cap \support \mathcal{S}.\) Since \( \mathcal{S}(\gamma_2) = 1, \) we also have \( \support \gamma_2 \cap \support \mathcal{S} \neq \emptyset . \) Consequently, we must have \( \bigl|(\support \mathcal{S})^+\bigr| \geq \dist(e,\gamma_2) ,\) and hence
    \begin{align*}
        &\sum_{\mathcal{S} \in \Xi_{1^+,k^+}  } \Bigl| \Psi_{\beta,\kappa}(\mathcal{S}) \mathbb{1} \pigl(  \mathcal{S}(\gamma_1) = \mathcal{S}(\gamma_2)  =  1 \pigr)\Bigr|
        \leq 
        \sum_{\mathcal{S} \in \Xi_{1^+,k^+}  } \bigl| \Psi_{\beta,\kappa}(\mathcal{S}) \bigr| \mathbb{1} \pigl(  \mathcal{S}(\gamma_1) = \mathcal{S}(\gamma_2)  =  1 \pigr)
        \\&\qquad\leq 
        \sum_{e \in \gamma_1}
        \sum_{\mathcal{S} \in \Xi_{1^+,k^+,e}  } \bigl| \Psi_{\beta,\kappa}(\mathcal{S}) \bigr| \mathbb{1} \pigl(  \bigl| (\support \mathcal{S})^+ \bigr|> \dist(e,\gamma_2) \pigr)
        \\&\qquad
        \leq 
        \sum_{j=1}^\infty \sum_{\substack{e \in \gamma_1 \mathrlap{\colon}\\ \dist(e,\gamma_2)=j}} \sum_{\mathcal{S} \in \Xi_{1^+,\max(j,k)^+,e}  } \bigl| \Psi_{\beta,\kappa}(\mathcal{S}) \bigr| .
    \end{align*} 
    Note that for \( j \geq 1,\) if \( j \neq 2R_n\) then there are at most two edges \( e \in \gamma_1\) such that \( \dist(e,\gamma_2)=j.\) Also, there are at most \( \max(T_n-2R_n,0)\) edges \( e \in \gamma_1\) such that \( \dist(e,\gamma_2)=2R_n.\)
    Using Lemma~\ref{lemma: new upper bound}, we thus obtain the desired conclusion.
\end{proof}

\begin{proof}[Proof of Theorem~\ref{theorem: mf ratio Higgs}]
    Let \( k \geq 1,\) and let \( n_k \geq 1\) be such that \( R_n,T_n > k\) for all \( n \geq n_k.\) Further, assume that \( N\) is large enough to guarantee that \( \dist (\gamma^{(n)},\partial B_N)>k.\) %
    Then, for any \( n \geq n_k,\) we have that
    \begin{align*}
        a_k \coloneqq \sum_{\mathcal{S} \in \Xi_{1^+,k^-} } \Psi_{\beta,\kappa}(\mathcal{S}) \mathbb{1} \pigl(  \mathcal{S}(\gamma_1^{(n)}) = \mathcal{S}(\gamma_2^{(n)})  =  1 \pigr)
    \end{align*} 
    is well defined, and does not depend on \( n\) nor the choice of \( (R_n)_{n\geq 1}\) and \( (T_n)_{n\geq 1}.\)  Moreover, by Lemma~\ref{lemma: gamma0 tail bound},  the limit \( \lim_{k \to \infty} a_k\) exists. Since we for all \( k \geq 1\) have
    \begin{equation*}
        a_k< \delta_{1,\varepsilon}< \infty,
    \end{equation*}
    it follows that  \( \lim_{k \to \infty} a_k < \infty.\)

     We now show that~\ref{item: thm 1} and~\ref{item: thm 2} holds. To this end, let \( n \geq n_k. \) Then, by Lemma~\ref{lemma: ratio sum}, we have
    \begin{align*}
        -4^{-1}\log \rho_N(\gamma_1^{(n)},\gamma_2^{(n)}) = a_k +  \sum_{\mathcal{S} \in \Xi_{1^+,k^+}  }  \Psi_{\beta,\kappa}(\mathcal{S}) \mathbb{1} \pigl(  \mathcal{S}(\gamma_1^{(n)}) = \mathcal{S}(\gamma_2^{(n)})  =  1 \pigr) .
    \end{align*}
    Using Lemma~\ref{lemma: gamma0 tail bound} and the Ginibre inequality, it follows that
    \begin{align*}
        \lim_{n\to \infty} -4^{-1}\log \rho(\gamma_1^{(n)},\gamma_2^{(n)})  
    \end{align*}
    exists and is equal to \( \lim_{k\to \infty} a_k, \) and hence
    \begin{equation*}
        \rho = \lim_{n\to \infty}  \rho(\gamma_1^{(n)},\gamma_2^{(n)})  
    \end{equation*}
    exists and is equal to \( e^{-4 \lim_{k\to \infty} a_k}.\) Since \(  \lim_{k\to \infty} a_k < \infty, \) it follows that \( \rho>0. \) This completes the proof of~\ref{item: thm 1} and~\ref{item: thm 2}.
    
    We now show that the~\ref{item: thm 3} holds. To this end, let \( n \geq 1, \) and let \( k = \min(R_n,T_n)\) (this guarantees that \( n \geq n_k\)). Then
    \begin{align*}
        &
        \bigl| \log \rho(\gamma_1^{(n)},\gamma_2^{(n)}) - \log \rho \bigr|
        =
        \bigl| \log \rho(\gamma_1^{(n)},\gamma_2^{(n)}) - (-4\lim_{k\to \infty} a_k) \bigr|
        \\&\qquad=
        4\bigl| -4^{-1}\log \rho(\gamma_1^{(n)},\gamma_2^{(n)}) - \lim_{k\to \infty} a_k \bigr|
        \\&\qquad \leq 
        4\Bigl( \bigl| -4^{-1}\log \rho(\gamma_1^{(n)},\gamma_2^{(n)}) - a_k \bigr| + \bigl| a_k- \lim_{k\to \infty} a_k\bigr|\Bigr)
    \end{align*}
    Using~Lemma~\ref{lemma: gamma0 tail bound}, we thus obtain
    \begin{align*}
        &\bigl| \log \rho(\gamma_1^{(n)},\gamma_2^{(n)}) - \log \rho \bigr| 
        \\&\qquad\leq   
        4C_{\varepsilon} \sum_{j=1}^\infty   e^{-4\max(j,k)(\kappa-\kappa_0^{(\textrm{Higgs})}-\varepsilon) }
        +
        2C_{\varepsilon}\max(T_n-2R_n,0)
        e^{-4\max(2R_n,k)(\kappa-\kappa_0^{(\textrm{Higgs})}-\varepsilon) }
        .
    \end{align*} 
    This shows that~\ref{item: thm 3} holds, and thus completes the proof.
\end{proof}

\section{The confinement phase (\( \beta\) and \( \kappa\) both small)}\label{section: confinement}
 
In this section, we prove our main result for the confinement phase, Theorem~\ref{theorem: mf ratio confinement}. The proof strategy is similar to that of Theorem~\ref{theorem: mf ratio Higgs}, with the main difference that we here need to use a high temperature expansion before we can find a convergent cluster expansion.

\subsection{A high temperature expansion}\label{sec: high temperature expansion confinement}

In this section, we recall the high temperature expansion of Ising lattice gauge theory from~\cite{flv2022}. 
To this end, for \( a \geq 0\) and \( j \in \{ 0,1 \},\) we define
\begin{equation*}
    \varphi_a(j) \coloneqq 
    \begin{cases}
        1&\text{if } j = 0 \cr 
        \tanh 2 a  &\text{if } j = 1.
    \end{cases}
\end{equation*}

Let \( \gamma \) be a path. For \( \omega \in \Omega^2(B_N,\mathbb{Z}_2),\) we let
\begin{equation}\label{eq: ht action}
    \varphi^\gamma_{\beta,\kappa} (\omega)  \coloneqq   
     \prod_{p \in C_2(B_N)^+} 
     \varphi_\beta\bigl(\omega(p)\bigr) \prod_{e \in C_1(B_N)^+}  
     \frac{\varphi_\kappa\bigl( \delta\omega(e) + \gamma[e] \bigr)}{\varphi_\kappa(\gamma[e])} .
\end{equation} 
Further, we let
\begin{equation}\label{eq: hat Z confinement}
	\hat Z_{N,\beta,\kappa}[\gamma] \coloneqq \varphi_\kappa\bigl( 1\bigr)^{|\gamma|} \sum_{\omega \in \Omega^2(B_N,\mathbb{Z}_2)} \varphi^\gamma_{\beta,\kappa}(\omega)   .
\end{equation}

The high temperature expansion of the lattice Higgs model is given by the following lemma.
\begin{lemma}[Proposition 4.1 in~\cite{flv2022}]\label{lemma: hte confinement}
    Let \( \beta,\kappa \geq 0.\) Then, for any path \( \gamma, \) we have
    \begin{equation*}
        \mathbb{E}_{N,\beta,\kappa}^{(U)}[W_\gamma] = \frac{\varphi_\kappa\bigl( 1\bigr)^{|\gamma|} \hat Z_{N,\beta,\kappa}[\gamma] }{\hat Z_{N,\beta,\kappa}[0] },
    \end{equation*} 
    and thus
    \begin{align*}
        \rho_N(\gamma_1,\gamma_2 ) =
        \frac{ \hat Z_{N,\beta,\kappa} [\gamma_1]   \hat Z_{N,\beta,\kappa} [\gamma_2] }{ \hat Z_{N,\beta,\kappa} [\gamma_1+\gamma_2]  \hat Z_{N,\beta,\kappa}[0]}.
    \end{align*}
\end{lemma}

For a path \( \gamma,\) we also define 
\[
    \hat Z_{N,\beta,\kappa}^\gamma \coloneqq \sum_{\substack{\omega \in \Omega^2(B_N,\mathbb{Z}_2)  \mathrlap{\colon}\\  \not \exists p \in \support \omega \colon p \sim \gamma    }}\varphi^0_{\beta,\kappa}(\sigma) .
\]
We note that for any path \( \gamma,\) we have
\begin{equation*}
    \hat Z_{N,\beta,\kappa}^\gamma \leq  \hat Z_{N,\beta,\kappa}. 
\end{equation*}

\subsection{A cluster expansion}\label{section: cluster confinement}

In this section we will describe the cluster expansion of~\cite{fv2017} for the model obtained from the high temperature expansion in the previous section. The setup for this expansion will be very similar with the corresponding setup in the Higgs phase, and for this reason we will aim to keep the exposition short here.

\subsubsection{Polymers}\label{sec: the graphs cp}

In this section, we assume that a path  \( \gamma \in \Lambda_0 \) is given.

Let \( \mathcal{G}_2\) be the graph with vertex set \(  C_2(B_N)^+,\) an edge between two distinct plaquettes \( p_1,p_2 \in C_2(B_N)^+ \) if \( \support  \partial p \cap \support  \partial p' \neq \emptyset\) (written \( p_1 \sim p_2\)). In other words, \( p_1 \sim p_2 \) if they have some common edge in their boundaries.
Note that any \( p \in C_2(B_N)^+ \) has degree at most \( M_2 \coloneqq 4 \cdot (2(m-1)-1) =    8m-12\) in \( \mathcal{G}_2.\)

For \( \omega \in \Omega^2(B_N,\mathbb{Z}_2), \) we let \( \mathcal{G}_2(\omega)\) be the subgraph of \( \mathcal{G}_2 \) induced by \(  (\support \omega)^+.\) In other words, \( \mathcal{G}_2 \) is the graph with vertex set  \(  (\support \omega)^+\) and an edge between two vertices \( p_1,p_2 \in (\support \omega)^+ \) if \( p_1 \sim p_2. \)

In this section, we let
\begin{equation*}
    \Lambda \coloneqq \bigl\{ \omega \in \Omega^2(B_N,\mathbb{Z}_2) \colon  \mathcal{G}_2(\omega) \text{ has exactly one connected component} \bigr\}.
\end{equation*} 
The spin configurations in \( \Lambda  \) will be referred to as \emph{polymers}.

\subsubsection{Polymer interaction}

For \( \omega,\omega' \in \Lambda,\) we write \( \omega \sim \omega'\) if   \( \mathcal{G}_2(\omega) \cup \mathcal{G}_2(\omega')\) is a connected subgraph of \( \mathcal{G}_2.\)

In the notation of~\cite[Chapter 3]{fv2017}, the model given by~\eqref{eq: hat Z confinement} corresponds to a model of polymers with polymers described in~Section~\ref{sec: the graphs cp}  and interaction function  \(  \iota(\omega_1,\omega_2) \coloneqq \zeta(\omega_1,\omega_2) +1 , \) where 
\begin{equation*}
	\zeta(\omega_1,\omega_2) \coloneqq  
	\begin{cases}  
		-1 &\text{if }  \omega_1 \sim \omega_2 \cr   
		0 &\text{else,} 
	\end{cases}
	\qquad \omega_1,\omega_2 \in \Lambda .
\end{equation*}

For \( \omega \in \Omega^2(B_N,\mathbb{Z}_2),\) we write \( \omega \sim \gamma\) if there is \( p \in \support \omega\) such that \( \support \partial p \cap \support  \gamma \neq \emptyset.\) Finally, we also write \( \gamma \sim \gamma.\)

\subsubsection{Clusters of polymers}\label{section: clusters confinement}

Consider a multiset 
\begin{align*}
    &\mathcal{S} = \{ \underbrace{\eta_1,\dots, \eta_1}_{n_{\mathcal{S}}(\eta_1) \text{ times}}, \underbrace{\eta_2, \dots, \eta_2}_{n_{\mathcal{S}}(\eta_2) \text{ times}},\dots, \underbrace{\eta_k,\dots,\eta_k}_{n_{\mathcal{S}}(\eta_k) \text{ times}} \} = \{\eta_1^{n(\eta_1)}, \ldots, \eta_k^{n(\eta_k)}\}, 
\end{align*}
where \( \eta_1,\dots,\eta_k \in \Lambda \) are distinct and $n(\eta)=n_{\mathcal{S}}(\eta)$ denotes the number of times $\eta$ occurs in~$\mathcal{S}$. Following \cite[Chapter 3]{fv2017}, we say that $\mathcal{S}$ is \emph{decomposable} if it is possible to partition $\mathcal{S}$ into disjoint multisets. That is, if there exist non-empty and disjoint multisets $\mathcal{S}_1,\mathcal{S}_2 \subset \mathcal{S}$ such that $\mathcal{S} = \mathcal{S}_1 \cup \mathcal{S}_2$ and such that for each pair $(\eta_1,\eta_2) \in \mathcal{S}_1 \times \mathcal{S}_2$,  $\eta_1 \nsim \eta_2.$  
If $\mathcal{S}$ is not decomposable, we say that $\mathcal{S}$  is a \emph{cluster}. We stress that a cluster is unordered and may contain several copies of the same polymer.

We let \(  \Xi\) be the set of all clusters.

For \( \mathcal{S} \in  \Xi,\) we let \( \mathcal{S}^0\) be the the set of all  \(\eta\in  \mathcal{S}\) such that \( \omega \nsim \gamma .\)  We let \( \mathcal{S}^\gamma \coloneqq \mathcal{S}\smallsetminus \mathcal{S}^0.\)

Given a cluster $\mathcal{S} \in  \Xi$, we let
\begin{equation*}
    \|\mathcal{S} \| = \sum_{\eta \in \Lambda  } n_{\mathcal{S}} (\eta) \bigl| (\support \eta )^+ \bigr| ,
\end{equation*} 
and
\begin{equation*}
    \|\mathcal{S} \|_\gamma \coloneqq \sum_{\eta \in \Lambda  } n_{\mathcal{S}^\gamma} (\eta) \bigl| \support \delta \eta \cap \support \gamma \bigr|  .
\end{equation*}
We also let
\begin{equation*}
    \support \mathcal{S}   = \bigcup_{\eta \in \Lambda  }   (\support \eta )^+ .
\end{equation*}

For \( p \in C_2(B_N)^+ \) and \( j \geq 1, \) we let
\begin{equation*}
    \Xi_{1,j,p}   \coloneqq \bigl\{\{ \eta \} \colon \eta \in \Lambda  \text{ and } p \in \support \eta \text{ and } \| \{ \omega\} \| = j\bigr\}
\end{equation*}  
and
\begin{equation*}
    \Xi_{1+,j+,p}   \coloneqq \bigl\{\mathcal{S} \in \Xi \colon   p \in \support \mathcal{S} \text{ and } \| S\| \geq j\bigr\}.
\end{equation*}  
 As in Section~\ref{section: clusters higgs}, these sets depends on \( N,\) and when this is important, we write \( \Xi(B_N), \) \(\Xi_{1,j,p})(B_N), \) and \(\Xi_{1+,j+,p})(B_N) \) respectively.

The following lemma is analogous to Lemma~\ref{lemma: Xi1ke bound} and gives an upper bound on the number of clusters of a given size and with a given plaquette in its support.

\begin{lemma}\label{lemma: no cl upper bound ht}
	For any \( k \geq 1 \) and  \( p \in C_2(B_N)^+ ,\) we have \( | \Xi_{1,k,p} | \leq M_2^{2k-2} .\)
\end{lemma}

\begin{proof}
    Let \( k \geq 1\) and \( p \in C_2(B_N)^+.\)
    Let \( \mathcal{P}\) be the set of all paths in \( \mathcal{G}_2 \) that starts at \( p \) and have length \( 2k-1.\) Since each vertex in \( \mathcal{G}_2\) has degree at most \(M_2,\) we have  \( |\mathcal{P}| \leq M_2^{2k-2}.\) 
     
     For any \( \{ \omega \} \in \Xi_{1,k,p},\) the graph \( \mathcal{G}_2(\omega)\)  is connected, and hence \(  \mathcal{G}_2(\omega)\) has a spanning path \( T_\omega \in \mathcal{P}\) of length \( 2k-1=2\bigl| (\support \omega)^+ \bigr|-1\) which starts at \( p .\) 
    Since the map \( \omega \mapsto T_\omega\) is an injective map from \( \Xi_{1,k,p}\) to \( \mathcal{P}\) and \( |\mathcal{P}| \leq  M_2^{2k-2},\) the desired conclusion immediately follows.
\end{proof}

\subsubsection{The activity of clusters}

We extend the notion of activity from~\eqref{eq: ht action} to clusters \( \mathcal{S} \in \Xi\) as follows. 
\begin{align*}
    \varphi^\gamma_{\beta,\kappa}(\mathcal{S}) \coloneqq \prod_{\eta \in \mathcal{S}} \varphi^\gamma_{\beta,\kappa}(\eta).
\end{align*}
Note that 
\[
\varphi^\gamma_{\beta,\kappa}(\mathcal{S}) = \varphi^0_{\beta,\kappa}(\mathcal{\mathcal{S}}) \varphi_\kappa(1)^{-\| \mathcal{S}\|_\gamma}.
\]

The following lemma, which we now state and prove, will be useful later.

\begin{lemma}\label{lemma: geometry upper bound}
 	Let \(\beta,\kappa >0, \) let \( \gamma \in \Gamma_n, \) and let \( \omega \in \Lambda. \) Then
	\[
		\varphi^\gamma_{\beta,\kappa}(\omega) \leq \varphi_\beta(1)^{(1-o_{n}(1))|(\support \omega)^+|}.
	\]
\end{lemma}

\begin{proof}
	For any \( \omega \in \Lambda, \) we have
	\begin{align*}
		&\varphi^\gamma_{\beta,\kappa}(\omega) = 
		\varphi_\beta(1)^{|(\support \omega)^+|}\varphi_\kappa(1)^{
		|(\support \delta \omega )^+ \smallsetminus \support \gamma|-|(\support \delta \omega )^+ \cap\support \gamma|}.
	\end{align*}
	
	Assume first that \( \omega \in \Lambda \) is such that \( |(\support \omega)^+| \geq  R_nT_n.\) In this case we have
	\begin{align*}
		&\varphi^\gamma_{\beta,\kappa}(\omega) \leq 
		\varphi_\beta(1)^{|(\support \omega)^+|}\varphi_\kappa(1)^{-|\gamma|}
		\\&\qquad\leq 
		\varphi_\beta(1)^{(1-o_{n})(1)|(\support \omega)^+|} 
		\varphi_\beta(1)^{-o_{n}(1) R_nT_n}\varphi_\kappa(1)^{-(R_n+T_n)}.
	\end{align*}
	Since \( R_n,T_n \to \infty, \) \( R_nT_n \) tends to infinity much quicker than \( R_n+T_n, \) and hence the desired conclusion follows if \( o_n(1) \) goes to zero slowly enough.

	Now instead assume that we are given \( |(\support \omega)^+| <   R_nT_n \) and \( |\support \delta \omega \cap \support \gamma | \neq 0. \) Then \(|\support  \delta \omega |\) is minimized if  \( \support \omega \) is a subset of the flat surface \( q \) that spans \( \gamma_1^{(n)} + \gamma_2^{(n)}.\) Assume that this is the case. For each edge \( e \in \support \delta \omega \cap \support \gamma, \) either all plaquettes in \( \support q \) between \( e \) and the opposite side of \( \gamma_1 + \gamma_2 \) are in \( \support \omega, \) or there is some such plaquette that is not in \( \support \omega, \) and in this case there there is at leat two edges parallell to \( e \) that is in \( (\support\delta \omega)^+\smallsetminus \support \gamma.\) Note that in the first of these cases, there are a total of at least \( \min(R_n,T_n) \) plaquettes that are in \( \support \omega. \) Moreover, each plaquette appearing in one of these sets are in sets corresponding to at most four edges. Hence
	\begin{equation*}
		|(\support \delta \omega )^+ \cap\support \gamma | - |(\support \delta \omega)^+ \smallsetminus \support \gamma | \leq \frac{4|(\support \omega)^+|}{\min (R_n,T_n)}.
	\end{equation*}
	Consequently, in this case, we have
	\begin{align*}
		\varphi^\gamma_{\beta,\kappa}(\omega) \leq 
		\varphi_\beta(1)^{|(\support \omega)^+|}\varphi_\kappa(1)^{
		-\frac{4|(\support \omega)^+|}{\min (R_n,T_n)}}.
	\end{align*}
	From this the desired conclusion immediately follows.
\end{proof}

\subsubsection{Ursell functions}

The Ursell functions we will use in the confinement phase are analogous to the Ursell functions used for the Higgs phase (see Definition~\ref{def:ursell}), except the set \( \Lambda \) is different and we use the graph \( \mathcal{G}_2 \) instead of \( \mathcal{G}_1.\)

\subsubsection{Cluster expansion of the partition function}

For \( \beta \geq 0, \) \( \kappa \geq 0 ,\) \( \gamma \in \Lambda, \) and \( \mathcal{S} \in \Xi, \) we let 
\[
	\Psi^\gamma_{\beta,\kappa}(\mathcal{S}) \coloneqq U(\mathcal{S}) \varphi^\gamma_{\beta,\kappa}(\mathcal{S}).
\]
Further, we let
\begin{equation}\label{eq: beta0 def}
	\beta_0^{(\text{conf})} \coloneqq \sup \bigl\{ \beta \geq 0 \colon M_2^2 \varphi_\beta(1)<1  \text{ and } \exists \alpha \in (0,1) \text{ s.t. } \frac{M_2^3 \varphi_\beta(1)^{1-\alpha}}{1-M_2^2 \varphi_\beta(1)^{1-\alpha}}< 2\alpha \bigr\}  
\end{equation}
and for \( \beta > \beta_0^{(\text{conf})}, \) we let
\begin{equation}\label{eq: alpha cp}
	\alpha \coloneqq \alpha_{\beta} \coloneqq \inf \{ \alpha \in (0,1) \colon \frac{M_2^3 \varphi_\beta(1)^{1-\alpha}}{1-M_2^2 \varphi_\beta(1)^{1-\alpha}}< 2\alpha \}.
\end{equation} 

\begin{lemma}\label{lemma: upper bound and convergence}
    Let \( 0 < \beta < \beta_0^{(\text{conf})} \) and \(  \kappa > 0,\)  and let \( \gamma \in \Gamma_n.\)  Then, for any \( \omega \in \Xi,\) we have 
    \begin{equation*}
        \sum_{\mathcal{S} \in \Xi \colon \omega \in \mathcal{S}} \bigl|\Psi^\gamma_{\beta,\kappa}(\mathcal{S})\bigr| \leq \varphi^\gamma_{\beta,\kappa}(\omega)^{1-\alpha}  .
    \end{equation*} 
    Moreover, 
    \begin{equation}\label{eq: logZ}
        \log \hat Z_{N,\beta,\kappa} [\gamma]= \sum_{\mathcal{S} \in \Xi^\gamma} \Psi^\gamma_{\beta,\kappa}(\mathcal{S})
    \end{equation}
    and
    \begin{equation}\label{eq: logZgamma}
        \log \hat Z_{N,\beta,\kappa}^\gamma = \sum_{\mathcal{S} \in \Xi \colon \mathcal{S}^\gamma = \emptyset} \Psi^0_{\beta,\kappa}(\mathcal{S}).
    \end{equation}
    Furthermore, the series on the right-hand sides of~\eqref{eq: logZ} and~\eqref{eq: logZgamma} are both absolutely convergent.
\end{lemma}

\begin{proof} 
	For \( \omega \in \Omega^2(B_N,\mathbb{Z}_2)\), let \( a(\omega) \coloneqq -\alpha \log \varphi^\gamma_{\beta,\kappa}(\omega).\) 
    We need to show that for any \( \omega \in \Lambda\) we have 
    \begin{equation}\label{eq: goal}
        \sum_{ \omega' \in \Lambda } \varphi^\gamma_{\beta,\kappa} (\omega') e^{a(\omega')} \mathbf{1}(\omega \sim \omega') \leq \alpha |\support \omega|.
    \end{equation}
    Given this, the conclusion follows from~\cite[Theorem~5.4]{fv2017}.

    We now show that~\eqref{eq: goal} holds. To this end, fix any \( \omega \in \Lambda.\)  Then   
    \begin{align*}
        &\sum_{ \omega'  \in \Lambda } \varphi^\gamma_{\beta,\kappa} (\omega') e^{a(\omega')} \mathbf{1}(\omega \sim \omega') 
        =
        \sum_{  \omega'   \in \Lambda } \varphi^\gamma_{\beta,\kappa} (\omega')^{1-\alpha}\mathbf{1}(\omega \sim \omega') 
        \\&\qquad \leq 
        \sum_{p \in (\support\omega)^+} \sum_{p'\sim p}\sum_{ \substack{\omega'  \in \Lambda \colon\\ p' \in \support \omega} } \varphi^\gamma_{\beta,\kappa}(\omega')^{1-\alpha} .
    \end{align*} 
	Since any \( p \in C_2(B_N)^+ \) has degree at most \( M_2 \) in \( \mathcal{G}_2, \) we can bound the previous expression from above by
    \begin{align*}
        &
        M_2|(\support \omega)^+| \sum_{j=1}^\infty \max_{p \in C_2(B_N)^+} \sum_{\{ \omega' \} \in \Xi_{1,j,p}} \varphi^\gamma_{\beta,\kappa}(\omega')^{1-\alpha} 
        \\&\qquad\leq 
        M_2|(\support \omega)^+|  \sum_{j=1}^\infty \max_{p \in C_2(B_N)^+} |\Xi_{1,j,p}| \max_{\{ \omega' \} \in \Xi_{1,j,p}} \varphi^\gamma_{\beta,\kappa}(\omega')^{1-\alpha} .
    \end{align*} 
    Using Lemma~\ref{lemma: geometry upper bound} and Lemma~\ref{lemma: no cl upper bound ht}, we obtain
    \begin{align*}
    	\sum_{\{ \omega' \} \in \Xi } \varphi^\gamma_{\beta,\kappa}(\omega') e^{a(\omega')} \mathbf{1}(\omega \sim \omega') \leq M_2 |(\support \omega)^+| \sum_{j=1}^{\infty} M_2^{2j-2} \varphi_\beta(1)^{(1-o_{n}(1))(1-\alpha)j}.
    \end{align*}
    The desired conclusion now follows from the choice of \( \alpha. \)
\end{proof}

\subsection{Upper bounds for clusters}\label{sec: upper bounds confinement}
In this section we give upper bounds on sums over the activity of sets of clusters which naturally arise in the proofs of Theorem~\ref{theorem: mf ratio confinement}.

\begin{lemma}\label{lemma: general Cbeta cp}
    Let \( 0 < \beta < \beta_0^{(\text{conf})} \) and \(  \kappa >0 , \) let \( \gamma \in \Gamma_n , \)  and let \( p \in C_2(B_N)^+.\) Then  
    \begin{equation}\label{eq: cbeta* cp}
        \sum_{\mathcal{S} \in \Xi_{1^+,1^+,p}}
        \bigl| \Psi^\gamma_{\beta,\kappa}(\mathcal{S}) \bigr|
        \leq 
        \frac{ 
            \varphi_{\beta}(1)^{(1-o_n(1))(1-\alpha)}}{1-M_2^{2}
            \varphi_{\beta}(1)^{(1-o_n(1))(1-\alpha)}} \eqqcolon C_{\varepsilon,n}^{(2)}
    \end{equation}
    where \( \alpha\) is as in~\eqref{eq: alpha cp}.
\end{lemma}

\begin{proof}
    By Lemma~\ref{lemma: upper bound and convergence}, for any \( \omega \in \Xi,\) we have 
    \begin{equation*}
        \sum_{\mathcal{S} \in \Xi \colon \omega \in \mathcal{S}} \bigl|\Psi^\gamma_{\beta,\kappa}(\mathcal{S})\bigr| \leq \varphi^\gamma_{\beta,\kappa} (\omega)^{1-\alpha}
    \end{equation*} 
    and hence
    \begin{equation}\label{eq: b1 cp}
        \begin{split}
            &\sum_{\mathcal{S} \in \Xi_{1^+,1^+,p}}
            \bigl| \Psi^\gamma_{\beta,\kappa}(\mathcal{S}) \bigr|
            \leq 
            \sum_{\{ \omega \} \in \Xi_{1,1+,p}}
            \sum_{\mathcal{S} \in \Xi \colon \omega \in \mathcal{S}} \bigl| \Psi^\gamma_{\beta,\kappa}(\mathcal{S}) \bigr|
            \leq 
            \sum_{\{ \omega \} \in \Xi_{1,1+,p}}
            \varphi^\gamma_{\beta,\kappa}(\omega)^{1-\alpha}.
        \end{split}
    \end{equation}
    Next, by Lemma~\ref{lemma: no cl upper bound ht} and Lemma~\ref{lemma: geometry upper bound}, we have
    \begin{equation}\label{eq: b2 cp}
        \begin{split}
            &
            \sum_{\{ \omega \} \in \Xi_{1,1+,p}}
            \varphi^\gamma_{\beta,\kappa}(\omega)^{1-\alpha}
            \leq
            \sum_{j=1}^{\infty}
            |\Xi_{1,j,p}|
            \max_{\omega \in \Xi_{1,j,p}} \varphi^\gamma_{\beta,\kappa}(\omega)^{1-\alpha}
            \leq
            \sum_{j=1}^{\infty}
            M_2^{2j-2}
            \varphi_{\beta}(1)^{(1-o_n(1))(1-\alpha)j}.
        \end{split}
    \end{equation}
    Combining~\eqref{eq: b1 cp} and~\eqref{eq: b2 cp}, the desired conclusion immediately follows.
\end{proof}

\begin{lemma}\label{lemma: new upper bound cp}
    Let \( 0 < \beta < \beta_0^{(\text{conf})} \) and \(  \kappa >0, \)  let \( \gamma \in \Gamma_n, \) and let \( p \in C_2(B_N)^+.\) Further, let \( k \geq 1\) and \( \varepsilon \in (0,\beta_0^{(\text{conf})}-\beta).\) Then
    \begin{equation*}
        \begin{split}
            & \sum_{\mathcal{S} \in \Xi_{1+,k+,p}} \bigl| \Psi^\gamma_{\beta,\kappa}(\mathcal{S}) \bigr|   
            \leq  
            C_{\varepsilon,n}^{(2)}\Bigl(\frac{\varphi_{\beta}(1)}{\varphi_{\beta+\varepsilon}(1)}\Bigr)^k ,
        \end{split}
    \end{equation*} 
    where \( C_{\varepsilon,n}^{(2)}\) is defined in~\eqref{eq: cbeta* cp}.
\end{lemma}

\begin{proof}
    Note first that, by definition, we have
    \[
    	 \Psi^\gamma_{\beta,\kappa}(\mathcal{S})  =  \Bigl(\frac{\varphi_{\beta}(1)}{\varphi_{\beta+\varepsilon}(1)}\Bigr)^{\| \mathcal{S} \|}  \Psi^\gamma_{\beta+\varepsilon,\kappa}(\mathcal{S})  . \]
    Using this observation, we obtain
    \begin{equation*}
        \begin{split}
            & \sum_{\mathcal{S} \in \Xi_{1+,k+,p}} \bigl| \Psi^\gamma_{\beta,\kappa}(\mathcal{S}) \bigr|   
            \leq  
            \Bigl(\frac{\varphi_{\beta}(1)}{\varphi_{\beta+\varepsilon}(1)}\Bigr)^k \sum_{\mathcal{S} \in \Xi_{1+,k+,p}} \bigl| \Psi^\gamma_{\beta+\varepsilon,\kappa}(\mathcal{S}) \bigr|     
            \\&\qquad\leq  
            \Bigl(\frac{\varphi_{\beta}(1)}{\varphi_{\beta+\varepsilon}(1)}\Bigr)^k \sum_{\mathcal{S} \in \Xi_{1+,1+,p}} \bigl| \Psi^\gamma_{\beta+\varepsilon,\kappa}(\mathcal{S}) \bigr| .  
        \end{split}
    \end{equation*} 
    Applying Lemma~\ref{lemma: general Cbeta cp}, we obtain the desired conclusion. 
\end{proof}

\subsection{Proof of Theorem~\ref{theorem: mf ratio confinement}} \label{section: proof of main result confinement}

In this section we prove Theorem~\ref{theorem: mf ratio confinement}. Before doing so we provide a useful lemma.

\begin{lemma}\label{lemma: upper bound on sum 2}
	Let \( 0 < \beta < \beta_0^{(\text{conf})} \) and \(  \kappa > 0. \) Further, let \( \gamma \in \Gamma_n ,\) let \( k \geq 0, \) and let \( \varepsilon \in (0,\beta_0^{(\text{conf})}-\beta). \) Then
	\begin{equation*} 		
        \sum_{\substack{\mathcal{S} \in \Xi  \colon \| \mathcal{S}\| \geq k, \\ \mathcal{S}^{\gamma_1} \neq \emptyset,\, \mathcal{S}^{\gamma_2} \neq \emptyset}} \bigl|\Psi^\gamma_{\beta,\kappa} (\mathcal{S}) \bigr| \leq 
        2(m-1)C_{\varepsilon,n}^{(2)} \biggl( 
		2 \sum_{j = 1}^{\infty}  
		\Bigl(\frac{\varphi_{\beta}(1)}{\varphi_{\beta+\varepsilon}(1)}\Bigr)^{\max(j,k)} 
		+ 
		\max(0,T_n-2R_n)  
		\Bigl(\frac{\varphi_{\beta}(1)}{\varphi_{\beta+\varepsilon}(1)}\Bigr)^{2R_n} 
		 \biggr),
	\end{equation*}
	where \( C_{\varepsilon,n}^{(2)}\) is defined in~\eqref{eq: cbeta* cp}.
\end{lemma}

\begin{proof}
	Note first that
	\begin{equation}\label{eq: upper bound on sum 2 eq 1}
		\sum_{\substack{\mathcal{S} \in \Xi  \colon \| \mathcal{S}\| \geq k, \\ \mathcal{S}^{\gamma_1} \neq \emptyset,\, \mathcal{S}^{\gamma_2} \neq \emptyset}} \bigl|\Psi^\gamma_{\beta,\kappa}(\mathcal{S})\bigr| 
		\leq
		\sum_{e \in \gamma_1} \sum_{p \in \support \hat \partial e} \sum_{\substack{\mathcal{S} \in \Xi_{1+,k+,p}\colon \\\| S \| \geq \dist(e,\gamma_2) }}\bigl|\Psi^\gamma_{\beta,\kappa}(\mathcal{S}) \bigr|.
	\end{equation}
	Using Lemma~\ref{lemma: new upper bound cp} and noting that \( | \support \hat \partial e| \leq 2(m-1) \) for any \( e \in C_1(B_N),\) we can upper bound the right hand side of~\eqref{eq: upper bound on sum 2 eq 1} by
	\begin{align*}
		&2(m-1)C_{\varepsilon,n}^{(2)} \sum_{e \in \gamma_1}  
		\Bigl(\frac{\varphi_{\beta}(1)}{\varphi_{\beta+\varepsilon}(1)}\Bigr)^{\max(k,\dist(e,\gamma_2))}  
		\\&\qquad\leq
		2(m-1)C_{\varepsilon,n}^{(2)} \biggl( 
		2 \sum_{j = 1}^{\infty}  
		\Bigl(\frac{\varphi_{\beta}(1)}{\varphi_{\beta+\varepsilon}(1)}\Bigr)^{\max(j,k)} 
		+ 
		\max(0,T_n-2R_n)  
		\Bigl(\frac{\varphi_{\beta}(1)}{\varphi_{\beta+\varepsilon}(1)}\Bigr)^{2R_n} 
		 \biggr).
	\end{align*}
	This concludes the proof.
\end{proof}

\begin{proof}[Proof of Theorem~\ref{theorem: mf ratio confinement}]
	By combining Lemma~\ref{lemma: hte confinement} and Lemma~\ref{lemma: upper bound and convergence}, we see that
	\begin{equation*}
		\log \rho_N(\gamma_1^{(n)},\gamma_2^{(n)}) = \log \frac{ \hat Z_{N,\beta,\kappa} [\gamma_1]   \hat  Z_{N,\beta,\kappa} [\gamma_2] }{ \hat Z_{N,\beta,\kappa} [\gamma_1+\gamma_2] \hat  Z_{N,\beta,\kappa}}.
	\end{equation*}
	
	To simplify notation, we now let \( \gamma_1 \coloneqq \gamma_1^{(n)} \) and \( \gamma_2 \coloneqq \gamma_2^{(n)} .\)

	Let  \( \varepsilon \in (0,\beta_0^{(\text{conf})}-\beta). \) 
	For each path \( \gamma \in \Gamma_n \) we have	
    \begin{align*} 
        \log  \hat  Z_{N,\beta,\kappa} [\gamma]    
        &=  
        \sum_{\mathcal{S} \in \Xi  \colon \mathcal{S}^{\gamma_1+\gamma_2} = \emptyset} U(\mathcal{S})\varphi^{\gamma}_{\beta,\kappa} (\mathcal{S}) 
        +
        \sum_{\mathcal{S} \in \Xi  \colon \mathcal{S}^{\gamma_1} = \emptyset,\, \mathcal{S}^{\gamma_2} \neq \emptyset} U(\mathcal{S}) \varphi^{\gamma}_{\beta,\kappa}(\mathcal{S}) 
        \\&\qquad +
        \sum_{\mathcal{S} \in \Xi  \colon \mathcal{S}^{\gamma_1} \neq \emptyset,\, \mathcal{S}^{\gamma_2} = \emptyset} U(\mathcal{S}) \varphi^{\gamma}_{\beta,\kappa} (\mathcal{S}) 
        +
        \sum_{\mathcal{S} \in \Xi  \colon \mathcal{S}^{\gamma_1} \neq \emptyset,\, \mathcal{S}^{\gamma_2} \neq \emptyset} U(\mathcal{S}) \varphi^{\gamma}_{\beta,\kappa} (\mathcal{S}) .
	\end{align*}  
	Note that if we e.g. have \( \mathcal{S}^{\gamma_1} = \emptyset, \) then \( \varphi^{\gamma_1}_{\beta,\kappa}(\mathcal{S}) = \varphi^{0}_{\beta,\kappa}(\mathcal{S}) \) and \(\varphi^{\gamma_1+\gamma_2}_{\beta,\kappa}(\mathcal{S}) = \varphi^{\gamma_2}_{\beta,\kappa}(\mathcal{S}). \) Using these observations, it follows that 
	\begin{align}
        &\log \frac{ \hat Z_{N,\beta,\kappa} [\gamma_1]   \hat  Z_{N,\beta,\kappa} [\gamma_2] }{ \hat Z_{N,\beta,\kappa} [\gamma_1+\gamma_2] \hat  Z_{N,\beta,\kappa}}\nonumber
        \\&\qquad= 
        \sum_{\mathcal{S} \in \Xi  \colon \mathcal{S}^{\gamma_1} \neq \emptyset,\, \mathcal{S}^{\gamma_2} \neq \emptyset} U(\mathcal{S}) \bigl( \varphi^{\gamma_1}_{\beta,\kappa} (\mathcal{S}) + \varphi^{\gamma_2}_{\beta,\kappa} (\mathcal{S}) -\varphi^{\gamma_1+\gamma_2}_{\beta,\kappa} (\mathcal{S}) - \varphi^{0}_{\beta,\kappa} (\mathcal{S}) \bigr)\nonumber
        \\&\qquad=  
        \sum_{\mathcal{S} \in \Xi  \colon \mathcal{S}^{\gamma_1} \neq \emptyset,\, \mathcal{S}^{\gamma_2} \neq \emptyset} \bigl( \Psi^{\gamma_1}_{\beta,\kappa} (\mathcal{S}) +\Psi^{\gamma_2}_{\beta,\kappa} (\mathcal{S}) -\Psi^{\gamma_1+\gamma_2}_{\beta,\kappa}  (\mathcal{S})-\Psi^{0}_{\beta,\kappa}  (\mathcal{S})\bigr).\nonumber
    \end{align}
    
    Now note that for any \( k \geq 1 ,\) \( n \geq 1 \) such that \( R_n,T_n>k, \)  and \( N \) sufficiently large,
    \begin{equation*}\label{eq: log MFR small k}
    	\sum_{\substack{\mathcal{S} \in \Xi  \colon \| \mathcal{S} \| \leq k \\ \mathcal{S}^{\gamma_1} \neq \emptyset,\, \mathcal{S}^{\gamma_2} \neq \emptyset}} \bigl( \Psi^{\gamma_1}_{\beta,\kappa} (\mathcal{S}) +\Psi^{\gamma_2}_{\beta,\kappa}  (\mathcal{S}) -\Psi^{\gamma_1+\gamma_2}_{\beta,\kappa}  (\mathcal{S})-\Psi^{0}_{\beta,\kappa}  (\mathcal{S})\bigr)
    \end{equation*}
    is independent on \(n .\) Using Lemma~\ref{lemma: upper bound on sum 2}, it follows that the limit \( \lim_{n\to \infty} \lim_{N\to \infty} \rho_N(\gamma_1^{(n)},\gamma_2^{(n}) \) exists. This completes the proof of the first part of Theorem~\ref{theorem: mf ratio confinement}.
    
    Next, note that by Lemma~\ref{lemma: upper bound on sum 2}, we have 
    \begin{align*}
    	&\sum_{\mathcal{S} \in \Xi  \colon \mathcal{S}^{\gamma_1} \neq \emptyset,\, \mathcal{S}^{\gamma_2} \neq \emptyset} \bigl( \Psi^{\gamma_1}_{\beta,\kappa}  (\mathcal{S}) +\Psi^{\gamma_2}_{\beta,\kappa}  (\mathcal{S}) -\Psi^{\gamma_1+\gamma_2}_{\beta,\kappa} (\mathcal{S})-\Psi^{0}_{\beta,\kappa}  (\mathcal{S})\bigr)
    	\\&\qquad\geq 
    	-8(m-1)C_{\varepsilon,n}^{(2)}  \biggl( 2 \sum_{j = 1}^{\infty}  
		\Bigl(\frac{\varphi_{\beta}(1)}{\varphi_{\beta+\varepsilon}(1)}\Bigr)^{j} 
		+ 
		\max(0,T_n-2R_n)  
		\Bigl(\frac{\varphi_{\beta}(1)}{\varphi_{\beta+\varepsilon}(1)}\Bigr)^{2R_n} 
		 \biggr).
    \end{align*}

    Define
    \begin{equation}\label{eq: cbeta* cp lim}
    	C_\varepsilon^{(2)} \coloneqq \lim_{n \to \infty}C_{\varepsilon,n}^{(2)} .
    \end{equation}
    Letting first \( N \to \infty \) and then \( n \to \infty, \) we obtain 
    \begin{align*}
        \lim_{n \to \infty } \log  \rho(\gamma_1^{(n)},\gamma_2^{(n)}) >
        \frac{-16(m-1)C_{\varepsilon}^{(2)} \tanh 2\beta }{\tanh(2\beta+ 2\varepsilon)- \tanh 2\beta} .
    \end{align*}
    This concludes the proof. 
\end{proof}

\section{The free phase (\( \beta \) large and \( \kappa \) small)}\label{section: free}

In this section, we provide a proof of Theorem~\ref{theorem: mf ratio free}. The proof strategy is similar to that of Theorem~\ref{theorem: mf ratio confinement}, but here we use a different high temperature expansion and also have to deal with a few additional complications before we can use a cluster expansion.
 
\subsection{A high temperature expansion}\label{sec: high temperature expansion free}

In this section, we use a high temperature expansion of the Ising lattice Higgs model to obtain alternative expressions for  \( Z_{N,\beta,\kappa}[\gamma]\) that are useful when \( \kappa \) is small and \( \beta \) is large.

When \( \gamma \) is a closed loop, we let \( q_{\gamma} \) be a corresponding oriented surface. We note that by the discrete Stoke's theorem (see, e.g.,~\cite[Section 2.3.2]{flv2020}), when \( \omega \in \Omega^2(B_N,\mathbb{Z}_2) \) is such that \( d \omega = 0, \) then \( \omega(q_{\gamma}) \) does not depend on the choice of \( q_{{\gamma}}. \)

For a path \( \gamma \in \Lambda_0,\) we define 
\begin{equation}\label{eq: check Z}
    \check Z_{N,\beta,\kappa}[\gamma ] 
    \coloneqq   
    \sum_{\substack{\omega \in \Omega^2(B_N,\mathbb{Z}_2) \colon\\ d\omega=0}} \sum_{\substack{\gamma' \in \Lambda_0 \colon\\ \delta (\gamma+\gamma')=0}} e^{\beta \sum_{p \in C_2(B_N)} \rho(\omega(p))  }  (\tanh 2\kappa)^{| \gamma'|}  \rho(\omega(q_{\gamma +\gamma'})).
\end{equation}

 The following lemma gives a connection between~\( Z_{N,\beta,\kappa}^{(U)}[\gamma]\) and~\(\check Z_{N,\beta,\kappa}[\gamma].\)

\begin{lemma}\label{lemma: hce 3}
    Let \( \beta,\kappa \geq 0.\) Then
    \[
    Z_{N,\beta,\kappa}^{(U)}[\gamma] = (\cosh 2\kappa)^{|C_1(B_N)^+|} \frac{| \Omega^1(B_N,\mathbb{Z}_2) |}{|\{ \omega \in \Omega^2(B_N,\mathbb{Z}_2) \colon d\omega=0 \}|}   \check Z_{N,\beta,\kappa}[\gamma].
    \] 
\end{lemma}

\begin{proof}
	For any \( \sigma \in \Omega^1(B_N,\mathbb{Z}_2) \), we have
    \begin{align*}
    	&
    	e^{ \kappa \sum_{e \in C_1(B_N)}\rho(\sigma(e))}
    	=
    	e^{ 2\kappa \sum_{e \in C_1(B_N)^+}\rho(\sigma(e))}
    	=
        \prod_{e \in C_1(B_N)^+} (\cosh 2\kappa + \rho(\sigma(e)) \sinh 2\kappa )  
        \\&\qquad=
        (\cosh 2\kappa)^{|C_1(B_N)^+|}  \prod_{e \in C_1(B_N)^+} \pigl(1 + \rho\bigl(\sigma(e)\bigr) \tanh \kappa \pigr)  
        \\&\qquad=
        (\cosh 2\kappa)^{|C_1(B_N)^+|} \sum_{\gamma' \in \Lambda_0} (\tanh 2\kappa)^{|\gamma' |} \rho\bigl(\sigma( \gamma')\bigr).
	\end{align*}
	Using the definition of \( Z_{N,\beta,\kappa}^{(U)}[\gamma],\) it follows that
    \begin{align*}
        &Z_{N,\beta,\kappa}^{(U)}[\gamma]
        =
        \sum_{\sigma \in \Omega^1(B_N,\mathbb{Z}_2)} e^{\beta \sum_{p \in C_2(B_N)} \rho(d\sigma(p)) + \kappa \sum_{e \in C_1(B_N)}\rho(\sigma(e))}\rho(\sigma(\gamma))
        \\&\qquad=
        (\cosh 2\kappa)^{|C_1(B_N)^+|}\sum_{\sigma \in \Omega^1(B_N,\mathbb{Z}_2)} e^{\beta \sum_{p \in C_2(B_N)} \rho(d\sigma(p))  } \sum_{\gamma' \in \Lambda_0} (\tanh 2\kappa)^{| \gamma'|}  \rho(\sigma(\gamma +\gamma')).
	\end{align*}
	Now note that if \( \gamma' \in \Lambda_0\) is such that  \( \support \gamma \cap \support \gamma' \neq \emptyset \) and \( \delta (\gamma+\gamma') \neq 0, \) then by gauge invariance, we have
	\begin{align*}
		\sum_{\sigma \in \Omega^2(B_N,\mathbb{Z}_2)} e^{\beta \sum_{p \in C_2(B_N)} \rho(d\sigma(p))  } (\tanh 2\kappa)^{| \gamma' |}  \rho(\sigma(\gamma+ \gamma')) = 0.
	\end{align*}
	Using this observation it follows that
    \begin{align*}
        &Z_{N,\beta,\kappa}^{(U)}[\gamma]
        = 
        (\cosh 2\kappa)^{|C_1(B_N)^+|}\sum_{\sigma \in \Omega^2(B_N, \mathbb{Z}_2)} e^{\beta \sum_{p \in C_2(B_N)} \rho(d\sigma(p))  } \sum_{\gamma' \in \Lambda_0 \colon \delta (\gamma+\gamma')=0} (\tanh 2\kappa)^{| \gamma'|}  \rho(\sigma(\gamma +\gamma')).
     \end{align*}
     Finally, using the Poincaré lemma (see, e.g.,~\cite[Lemma 2.2]{c2019}), we obtain
     \begin{align*}
        Z_{N,\beta,\kappa}^{(U)}[\gamma]
         &=
        (\cosh 2\kappa)^{|C_1(B_N)^+|} \frac{| \Omega^1(B_N,\mathbb{Z}_2) |}{|\{ \omega \in \Omega^2(B_N,\mathbb{Z}_2) \colon d\omega=0 \}|} 
        \\&\qquad \cdot \sum_{\substack{\omega \in \Omega^2(B_N,\mathbb{Z}_2) \colon\\ d\omega=0}} \sum_{\substack{\gamma' \in \Lambda_0 \colon\\ \delta (\gamma+\gamma')=0}} e^{\beta \sum_{p \in C_2(B_N)} \rho(\omega(p))  }  (\tanh 2\kappa)^{| \gamma'|}  \rho(\omega(q_{\gamma +\gamma'}))
        \end{align*}
        as desired. This concludes the proof.
\end{proof}

When \( \gamma \) is a closed path, we verify in Section~\ref{sec: cluster expansions free} that~\( \check Z_{N,\beta,\kappa}[\gamma]\) has a cluster expansion. However, when \( \gamma \) is an open path, this argument fails. For this reason, we now give an alternative expression for~\eqref{eq: check Z}. 
To this end, for a closed loop \( \gamma \) and an open path \( \gamma_0, \) we let 
\begin{align*}
	\check Z_{N,\beta,\kappa}[\gamma,\gamma_0] \coloneqq \sum_{\substack{\omega \in \Omega^2(B_N,\mathbb{Z}_2) \colon\\ d\omega=0}} \sum_{\substack{\gamma' \in \Omega^1(B_N,\mathbb{Z}_2)  \colon \\ \gamma' \nsim \gamma_0 ,\, \delta \gamma'=0}} 
     e^{\beta \sum_{p \in C_2(B_N)} (\rho(\omega(p)) -1) }  (\tanh 2\kappa)^{| \gamma'|}  \rho(\omega(q_{\gamma+\gamma'})).
\end{align*}
Further, for any open connected path, we let \(  \mathcal{L}_\gamma \) be the set of all connected paths \( \gamma_0 \) such that \( \gamma+\gamma_0 \) is closed (see Figure~\ref{figure: mirrored path a}).

\begin{lemma}\label{lemma: high temperature free alternative}
	Let \( \beta,\kappa \geq 0. \) Then the following holds.
	\begin{enumerate}[label=(\roman*)]
		\item For any closed path \( \gamma, \) we have \label{item: closed check Z}
		\begin{align*}
    		&\check Z_{N,\beta,\kappa}[\gamma ] = \check Z_{N,\beta,\kappa}[\gamma,0].
		\end{align*} 
		\item For any open connected path \( \gamma, \) we have \label{item: open check Z}
		\begin{align*}
    		\check Z_{N,\beta,\kappa}[\gamma ] 
    		=     \sum_{\gamma_0 \in \mathcal{L}_\gamma} 
    		(\tanh 2\kappa)^{| \gamma_0 |} 
     		{\check Z}_{N,\beta,\kappa}[\gamma+\gamma_0,\gamma_0]
     		.
		\end{align*} 
	\end{enumerate} 
\end{lemma}

\begin{proof}
	Assume first that \( \gamma \) is a closed path. Then, by definition, we have 
	\begin{align*}
    	&\check Z_{N,\beta,\kappa}[\gamma ] 
    	=
    	\sum_{\substack{\omega \in \Omega^2(B_N,\mathbb{Z}_2) \colon\\ d\omega=0}} \sum_{\substack{\gamma' \in \Lambda_0 \colon\\ \delta \gamma'=0}}  e^{\beta \sum_{p \in C_2(B_N)} \rho(\omega(p))  }  (\tanh 2\kappa)^{| \gamma'|}  \rho(\omega(q_{\gamma+\gamma'})) .
	\end{align*}  
	and hence~\ref{item: closed check Z} holds.
	
	Now instead assume that \( \gamma \) is a connected open path with \( |\support \hat \partial \gamma | = 2 .\) Then, by definition, we have
	\begin{align*}
    	\check Z_{N,\beta,\kappa}[\gamma ] 
    	&=
    	\sum_{\gamma_0 \in \mathcal{L}_\gamma} 
    	(\tanh 2\kappa)^{| \gamma_0|} 
    	\\&\qquad\cdot \sum_{\substack{\omega \in \Omega^2(B_N,\mathbb{Z}_2) \colon\\ d\omega=0}} \sum_{\substack{\gamma' \in \Lambda_0  \colon \gamma' \nsim \gamma_0 \\ \delta \gamma'=0}}  e^{\beta \sum_{p \in C_2(B_N)} \rho(\omega(p))  }  (\tanh 2\kappa)^{| \gamma'|}  \rho(\omega(q_{\gamma+\gamma_0+\gamma'})).
	\end{align*}  
	This completes the proof of~\ref{item: open check Z}.
\end{proof}

\subsection{A cluster expansion}\label{sec: cluster expansions free}
In this section, using the high temperature expansion of Section~\ref{sec: high temperature expansion free}, we present a cluster expansion which is useful in the free phase.

\subsubsection{Polymers}\label{section: polymers free}

Let \( \mathcal{G}_3\) be the graph with vertex set \(  C_2(B_N)^+\) and an edge between two distinct plaquettes \( p_1,p_2 \in C_2(B_N)^+ \) if \( \support \hat \partial p \cap \support  \hat \partial p' \neq \emptyset\) (written \( p_1 \sim p_2\)).
Note that any \( p \in C_2(B_N)^+ \) has degree at most \( M_3 \coloneqq 10(m-2)\) in \( \mathcal{G}_3.\)
For \( \omega,\omega' \in \Omega^2(B_N,\mathbb{Z}_2),\) we write \( \omega \sim \omega' \) if the subgraph of \( \mathcal{G}_3 \) induced by \(  (\support \omega)^+ \cup  (\support \omega')^+ \) is connected. 
When \( \omega \in \Omega^2(B_N,\mathbb{Z}_2), \) we let \( \mathcal{G}_3(\omega_2 )\)  be the subgraph of \( \mathcal{G}_3 \) induced by \( (\support \omega)^+. \)
We let  
\begin{equation*}
    \Lambda_2 \coloneqq \bigl\{ \omega \in \Omega^2(B_N,\mathbb{Z}_2) \colon  \mathcal{G}_3(\omega) \text{ has exactly one connected component} \bigr\}.
\end{equation*} 

Recall the definitions of~\( \mathcal{G}_0 ,\) \( M_0, \) and~\( \Lambda_1 \) from Section~\ref{section: notation}.

For two paths \( \gamma,\gamma' \in \Lambda_0 ,\) we write \( \gamma \sim \gamma' \) if the subgraph of \( \mathcal{G}_0 \) induced by \(  \support \gamma \cup  \support \gamma' \) is connected.

In this section, thee elements in \( \Lambda_1  \) and \( \Lambda_2 \) will be referred to as \emph{polymers}.
 
\subsubsection{Polymer interaction}

For \( \omega,\omega' \in \Lambda_2,\) we write \( \omega \sim \omega'\) if   \( \mathcal{G}_3(\omega) \cup \mathcal{G}_3(\omega')\) is a connected subset of \( \mathcal{G}_3.\)

For \( \gamma,\gamma' \in \Lambda_1,\) we recall that \( \gamma \sim \gamma'\) if   \( \mathcal{G}_0(\gamma) \cup \mathcal{G}_0(\gamma')\) is a connected subset of \( \mathcal{G}_0.\)

For \( \gamma \in \Lambda_1 \) and \( \omega \in \Lambda_2, \) we write \( \gamma \sim \omega \) and \( \omega \sim \gamma \)  if \( \rho(\omega(q_\gamma)) = -1. \)

In the notation of~\cite[Chapter 3]{fv2017}, the model described by \( \check Z [\gamma,\gamma_0]\) corresponds to a model of polymers with polymers described in~Section~\ref{section: polymers free} and interaction function  \(  \iota(\eta_1,\eta_2) \coloneqq \zeta(\eta_1,\eta_2) +1 , \) where 
\begin{equation*}
	\zeta(\eta_1,\eta_2) \coloneqq \iota(\eta_1,\eta_2) -1 = 
	\begin{cases} 
		-2 &\text{if } \eta_1 \in \Lambda_1,\, \eta_2 \in \Lambda_2  \text{ and } \rho(\eta_2(\eta_1))=-1\cr
		-2 &\text{if } \eta_1 \in \Lambda_2,\, \eta_2 \in \Lambda_1  \text{ and } \rho(\eta_1(\eta_2))=-1 \cr 
		-1 &\text{if } \eta_1,\eta_2 \in \Lambda_1 \text{ and } \eta_1 \sim \eta_2 \cr  
		-1 &\text{if } \eta_1,\eta_2 \in \Lambda_2 \text{ and } \eta_1 \sim \eta_2 \cr  
		0 &\text{else.} 
	\end{cases}
\end{equation*}

\subsubsection{Clusters of polymers}

Consider a multiset 
\begin{align*}
    &\mathcal{S} = \{ \underbrace{\eta_1,\dots, \eta_1}_{n_{\mathcal{S}}(\eta_1) \text{ times}}, \underbrace{\eta_2, \dots, \eta_2}_{n_{\mathcal{S}}(\eta_2) \text{ times}},\dots, \underbrace{\eta_k,\dots,\eta_k}_{n_{\mathcal{S}}(\eta_k) \text{ times}} \} = \{\eta_1^{n(\eta_1)}, \ldots, \eta_k^{n(\eta_k)}\}, 
\end{align*}
where \( \eta_1,\dots,\eta_k \in \Lambda_1 \cup \Lambda_2 \) are distinct and $n(\eta)=n_{\mathcal{S}}(\eta)$ denotes the number of times $\eta$ occurs in~$\mathcal{S}$. Following~\cite[Chapter 3]{fv2017}, we say that $\mathcal{S}$ is \emph{decomposable} if it is possible to partition $\mathcal{S}$ into disjoint multisets. That is, if there exist non-empty and disjoint multisets $\mathcal{S}_1,\mathcal{S}_2 \subset \mathcal{S}$ such that $\mathcal{S} = \mathcal{S}_1 \cup \mathcal{S}_2$ and such that for each pair $(\eta_1,\eta_2) \in \mathcal{S}_1 \times \mathcal{S}_2$,  $\eta_1 \nsim \eta_2.$  
If $\mathcal{S}$ is not decomposable, we say that $\mathcal{S}$  is a \emph{cluster of polymers}. We stress that such a cluster is unordered and may contain several copies of the same polymer.

In this section, we let \(  \Xi \) be the set of all clusters.

When \( \mathcal{S} \in \Xi, \) we let \( \mathcal{S}_1 \) denote the multiset \( \{ \eta^{n(\eta)} \}_{\eta \in \mathcal{S} \colon \eta \in \Lambda_1} \) and analogously let \( \mathcal{S}_2 \) denote the multiset \( \{ \eta^{n(\eta)} \}_{\eta \in \mathcal{S} \colon \eta \in \Lambda_2}. \) Further, we let
\begin{equation*}
	\| \mathcal{S} \| \coloneqq \sum_{\eta \in \mathcal{S}} n(\eta) |(\support \eta)^+| \quad \text{and} \quad n(\mathcal{S}) \coloneqq \sum_{\eta \in \mathcal{S}} n(\mathcal{S}).
\end{equation*}

When \( \gamma \in \Lambda_0 \) and \( v \in C_0(B_N)^+, \) we write \( v \sim \gamma \) if there is \( e \in \gamma \) such that \( v \in (\support \partial e)^+ .\) Similarly, if \( \omega \in \Omega^2(B_N,\mathbb{Z}_2) \) and \( p \in C_2(B_N)^+\) we write \( \omega \sim p \) if there is \( p' \in (\support \omega)^+ \) such that \( \support \hat \partial p \cap \support \hat \partial p' \neq \emptyset. \)

When \( \mathcal{S} \in \Xi \) and \( \gamma_0 \) is a path, we write \( \mathcal{S}_1 \sim \gamma_0 \) if there is \( \gamma' \in \mathcal{S}_1 \) such that \( \gamma \sim \gamma_0. \)

For \( i \geq 0, \) we let
\begin{equation*}
	\Xi_i \coloneqq \bigl\{ \mathcal{S} \in \Xi \colon n(\mathcal{S}) = 1 \bigr\}.
\end{equation*}
 As in Sections~\ref{section: clusters higgs} and~\ref{section: clusters confinement}, the sets \( \Xi  \) and \( \Xi_i \) depend on \( N, \) but we usually suppress this dependency.

\subsubsection{The activity of clusters}

For a closed path \( \gamma \in \Lambda_0 \) and a path \( \gamma_0 \in \Lambda_0, \) we define the activity of clusters \( \mathcal{S} \in \Xi\) by
\begin{align*}
    \varphi^{\gamma,\gamma_0}_{\beta,\kappa} (\mathcal{S}) 
    \coloneqq 
    \prod_{\omega \in \mathcal{S}_2} \rho(\omega(\gamma)) e^{\beta \sum_p (\rho(\omega(p))-1)} 
    \prod_{\gamma' \in \mathcal{S}_1} (\tanh 2\kappa)^{| \gamma'|}\mathbf{1}(\gamma' \nsim \gamma_0).
\end{align*}

\subsubsection{Ursell functions}

The Ursell function which is relevant in the free phase, which we define below, is slightly different than the Ursell function associated to hard-core interaction which was used in the Higgs phase and the confinement phase.

Recall the definition of \( \mathcal{G}^k\) from Section~\ref{sec: ursell 1}. 

\begin{definition}[The Ursell functions]\label{def:ursell 2}
    For \( k \geq 1 \) and  \( \eta_1,\eta_2,\dots , \eta_k \in \Lambda_1 \cup \Lambda_2 ,\) we let
    \begin{align*}
        U(  \eta_1, \ldots, \eta_k  ) 
        \coloneqq 
        \frac{1}{k!} 
        \sum_{G \in \mathcal{G}^k} 
        &(-1)^{|E(\mathcal{G} )|} 
        \prod_{\substack{(i,j) \in E(\mathcal{G}) \colon\\ \eta_i,\eta_j \in \Lambda_1}} \mathbb{1}( \eta_i \sim \eta_j)
        \prod_{\substack{(i,j) \in E(\mathcal{G}) \colon\\ \eta_i,\eta_j \in \Lambda_2}} \mathbb{1}( \eta_i \sim \eta_j)
        \\&\qquad\cdot 
        \prod_{\substack{(i,j) \in E(\mathcal{G}) \colon \\ |\{ \eta_i,\eta_j \} \cap  \Lambda_1|=1  }} 2 \cdot \mathbb{1}( \eta_i \sim \eta_j).
    \end{align*} 
    Note that this definition is invariant under permutations of the polymers \( \eta_1,\eta_2,\dots,\eta_k.\)

   For \( \mathcal{S} \in \Xi_k,\) and any enumeration \( \eta_1,\dots, \eta_k \) (with multiplicities) of the polymers in \( \mathcal{S},\) we define
   \begin{equation}\label{eq: ursell functions} 
        U(\mathcal{S}) = k! \, U(\eta_1,\dots, \eta_1).
   \end{equation}
\end{definition} 
Note that for any \( \mathcal{S} \in \Xi_1,\) we have \( U(\mathcal{S})=1,\) and for any \( \mathcal{S} \in \Xi_2,\)  we have either
\( U(\mathcal{S}) = -1 \) or \( U(\mathcal{S}) = -2. \)

\subsubsection{Cluster expansion of the partition function}

Before we state and prove the main result of this section, we will state and prove a few useful lemmas.

\begin{lemma}\label{lemma: intersection bound for path-path}
	Let \( \beta \geq 0 \) and \( \kappa \geq 0 . \) Further, let \( \gamma \) be a closed path, let \( \gamma_0 \) be a path, and let \( \gamma' \) be a non-empty path. Then, for any \( a \in (0,1), \) we have
	\begin{equation*}
		\sum_{\substack{ \gamma''  \in \Lambda_1 \mathrlap{\colon} \\ \gamma'' \sim \gamma'}} \varphi^{\gamma,\gamma_0}_{\beta,\kappa}(\{\gamma''\})^{a } 
		\leq \bigl| \{ v \in C_0(B_N)^+ \colon v \sim \gamma'  \}  \bigr| \sum_{j=2}^\infty 2j(2m)^{2j} (\tanh 2 \kappa)^{2aj}. 
	\end{equation*} 
\end{lemma}

\begin{proof}
	If \( \gamma'' \in \Lambda_1 \) is such that \( \gamma''\sim \gamma', \) then there must exist some \( v \in C_0(B_N)^+ \) such that \( v \sim \gamma' \) and \( v \in \gamma''. \) Hence
	 	\begin{align*}
		&\sum_{\substack{ \gamma''  \in \Lambda_1 \mathrlap{\colon} \\ \gamma'' \sim \gamma'}}  \varphi^{\gamma,\gamma_0}_{\beta,\kappa} (\{ \gamma''\})^{a }
		\leq
		\sum_{\substack{ \gamma''  \in \Lambda_1 \mathrlap{\colon} \\ \gamma'' \sim \gamma'}} \varphi^{0,0}_{\beta,\kappa} (\{ \gamma''\})^{a }
		= 
		\sum_{\substack{ \gamma''  \in \Lambda_1 \mathrlap{\colon} \\ \gamma'' \sim \gamma'}} 
		(\tanh 2\kappa)^{a|\gamma''|} .
	\end{align*} 
	Since any \( \gamma'' \in \Lambda_1 \) is closed, \( |\gamma''| \) is even. Moreover, any non-trivial \( \gamma'' \in \Lambda_1 \) satisfies  \( |\gamma''| \geq 4.\) Combining these observations, we obtain
	\begin{align*}
		& 
		\sum_{\substack{ \gamma''  \in \Lambda_1 \mathrlap{\colon} \\ \gamma'' \sim \gamma'}} 
		(\tanh 2\kappa)^{a|\gamma''|} 
		\leq
		\sum_{v \sim \gamma'} \sum_{j=2}^\infty \bigl| \{ \gamma'' \in \Lambda_1 \colon v \sim \gamma'' \text{ and } |\gamma''|=2j \} \bigr| (\tanh 2\kappa)^{ 2aj}.
	\end{align*}
	
	Since any \( v' \in C_0(B_N)^+ \) has degree at most \( 2m ,\) for any \( j \geq 1 \)  there can be at most \( 2j (2m)^j \) paths \( \gamma'' \in \Lambda_1 \) such that \( v \in \gamma'' \) and \( |\gamma''|=j, \) and hence
	
	\[ 
	 \bigl| \{ \gamma'' \in \Lambda_1 \colon v \sim \gamma'' \text{ and } |\gamma''|=2j \} \bigr| \leq 2j (2m)^{2j}.
	 \]
	 Combining the above inequalities, the desired conclusion now follows.
\end{proof}

\begin{lemma}\label{lemma: intersection bound for spin-spin}
	Let \( \beta \geq 0\) and \( \kappa \geq 0. \) Further, let \( \gamma \) be a closed path, let \( \gamma_0 \) be a path,  let \( \omega \in \Lambda_2 , \)  and let \( a \in (0,1) .\)   Then 
	\begin{equation*}
		\sum_{\substack{ \omega'  \in \Lambda_2 \mathrlap{\colon} \\\omega' \sim \omega}} \bigl|\varphi^{\gamma,\gamma_0}_{\beta,\kappa}(\{\omega'\}) \bigr|^{a}  \leq 
		\bigl|(\support \omega)^+\bigr|\sum_{j=2(m-1)}^\infty 
		M_3^{2j+1} e^{-4\beta a j} .
	\end{equation*}
\end{lemma}

\begin{proof}
	Since any non-trivial \( \omega' \in \Lambda_2 \) satisfies \( |(\support \omega')^+| \geq 2(m-1)\) (see, e.g.,~\cite{c2019}), we can write
	\begin{align*}
		&\sum_{\substack{ \omega'  \in \Lambda_2 \mathrlap{\colon} \\\omega' \sim \omega}} \bigl|\varphi^{\gamma,\gamma_0}_{\beta,\kappa} (\{\omega'\}) \bigr|^{a} 
		=
		\sum_{\substack{ \omega'  \in \Lambda_2 \mathrlap{\colon} \\\omega' \sim \omega}} e^{-2\beta a|\support \omega'|}  
		\\&\qquad\leq
		\sum_{p \in (\support \omega)^+} \sum_{j=2(m-1)}^\infty 
		\bigl| \{ \omega' \in \Lambda_2 \colon \omega' \sim p \text{ and } |(\support \omega')^+|=j\} \bigr| e^{-4\beta aj}   .
	\end{align*} 
	If \( \omega' \in \Lambda_2 \) and \( p \in (\support \omega)^+ \) are such that \( \omega' \sim p, \) then  \( \{ p \} \cup (\support \omega')^+ \) is a induces a connected subgraph \( G \) of \( \mathcal{G}_3. \) Since \( G \) is a connected graph, it has a spanning path of length at most \( 2|(\support \omega')^+|+1 \) that starts at \( p. \) The number of paths in \( \mathcal{G}_3 \) of length \( 2|(\support \omega')^+|+1 \) that starts at \( p \) is at most \( M_3^{ 2|(\support \omega')^+|+1 }. \) Hence
	\begin{equation*}
		\bigl| \{ \omega' \in \Lambda_2 \colon \omega' \sim p \text{ and } |(\support \omega')^+|=j\} \bigr|
		\leq M_3^{2j+1}.
	\end{equation*}
	Combining the above inequalities, we obtain the desired conclusion.
\end{proof}

\begin{lemma}\label{lemma: intersection bound for path-spin}
	Let \( \beta \geq 0 \) and \( \kappa \geq 0. \) Further, let \( \gamma \) be a closed path, let \( \gamma_0 \) be a path,  let \( \gamma' \in \Lambda_1 ,\) and let \( a \in (0,1). \)  Then
	\begin{equation*} 
        \sum_{\substack{\omega'  \in \Lambda_2 \mathrlap{\colon}\\ \omega' \sim \gamma'}}
         \bigl| \varphi^{\gamma,\gamma_0}_{\beta,\kappa}(\{\omega'\})\bigr|^{a}  \leq | \gamma'|\sum_{j=1}^\infty D_0 j^3  \sum_{k=\max(4j, 2(d-1))}^\infty 
         M_3^{2k-1} e^{-4\beta a k}.
	\end{equation*}
\end{lemma}

\begin{proof}
	Note first that 
	\begin{align*}
		&\sum_{\substack{\omega'  \in \Lambda_2 \mathrlap{\colon}\\\omega' \sim \gamma' }}
         \bigl| \varphi^{\gamma,\gamma_0}_{\beta,\kappa}(\{\omega'\})\bigr|^{a}  
         \leq 
         \sum_{e \in \gamma'} \sum_{j=1}^\infty 
         \sum_{\substack{p \in C_2(B_N)^+ \mathrlap{\colon}\\ \dist(p,e)=j}}  
         \sum_{\substack{\omega' \in \Lambda_2 \colon p \in \support \omega' ,\\   \dist(\omega',\gamma') \geq j }} \bigl| \varphi^{\gamma,\gamma_0}_{\beta,\kappa}(\{\omega'\})\bigr|^{a}  .
        \end{align*}
        
         If \( \omega' \in \Lambda_2 \) is such that \( \omega' \sim \gamma', \) then  every oriented surface \( q \) with \( \partial q = \gamma' \) \( \omega \) must intersect \( q,\) and hence the support of \( \omega' \) must loop around \( \gamma' .\)  Consequently, if \( \omega' \in \Lambda_2 \) is such that  \( \dist(\omega',\gamma') = j \) and \( \omega' \sim \gamma', \) then we must have \( |(\support \omega')^+| \geq 4j. \) Moreover, for any \( \omega' \in \Lambda_2, \) we have \( |(\support \omega')| \geq 2(m-1)  \) (see, e.g.,~\cite{c2019}).
         Consequently, for any \( j\geq 1 \) and \( p \in C_2(B_N)^+ \) such that \( \dist(p,e)=j, \) we have  
         \begin{align*}
         	&\sum_{\substack{\omega' \in \Lambda_2 \colon p \in \support \omega' ,\\   \dist(\omega',\gamma') \geq j }} \bigl| \varphi^{\gamma,\gamma_0}_{\beta,\kappa}(\{\omega'\})\bigr|^{a}  
         	\\&\qquad\leq
         	\sum_{k=\max(4j, 2(d-1))}^\infty 
         \bigl| \{ \omega' \in \Lambda_2 \colon p \in \support \omega' \text{ and } |(\support \omega')^+| = k\}\bigr| e^{-2\beta a \cdot  2k} .
         \end{align*}
         Now note that any \( \omega' \in \Lambda_2 \) corresponds to a connected component in \( \mathcal{G}_3 \) which has a spanning path of length at most \( 2|(\support \omega')^+|-1 \) that starts at some given \( p \in \support \omega'. \) This implies in particular that for any \( p \in C_2(B_N)^+ \) and any \( k \geq 1, \) we have
    	\begin{align*}
    		\bigl| \{ \omega' \in \Lambda_2 \colon p \in \support \omega' \text{ and } |(\support \omega')^+| = k\}\bigr| \leq M_3^{2k-1}.
    	\end{align*}
    	Finally, we note that
    	for any edge \( e \in C_1(B_N)^+ \) and \( j \geq 1, \) we have
    	\[
    	\bigl| \{ p \in C_2(B_N)^+ \colon \dist(e,p)=j \} \bigr| \leq D_0 j^3.
    	\]
    	 Combining the above equations,  the desired conclusion follows.
\end{proof}

\begin{lemma}\label{lemma: intersection bound for spin-path}
	Let \( \beta \geq 0  \) and let \( \kappa \geq 0. \) Further, let \( \gamma  \) be a closed path, let \( \gamma_0 \) be a path, and let \( \omega \in \Lambda_2 .\) Further, let \( a \in (0,1). \)  Then
	\begin{equation*}
		\sum_{\substack{ \gamma'  \in \Lambda_1 \mathrlap{\colon}\\ \gamma' \sim \omega}}  \bigl|\varphi^{\gamma,\gamma_0}_{\beta,\kappa}(\{\gamma'\})  \bigr|^{a} \leq |(\support \omega)^+| \sum_{j=1}^\infty D_0j^3  \sum_{k=4j}^\infty (2m)^k (\tanh 2\kappa)^{ak}.
	\end{equation*}
\end{lemma}

As the proof of Lemma~\ref{lemma: intersection bound for spin-path} is completely analogous to the proof of Lemma~\ref{lemma: intersection bound for path-spin}, we omit it here.

For \( \beta,\kappa \geq 0, \) a closed loop \( \gamma ,\) a path \( \gamma_0 ,\) and \( \mathcal{S} \in \Xi, \) we let 
\[
	\Psi^{\gamma,\gamma_0}_{\beta,\kappa}(\mathcal{S}) \coloneqq U(\mathcal{S}) \varphi^{\gamma,\gamma_0}_{\beta,\kappa}(\mathcal{S}).
\]
Note that with this notation, we have
\begin{equation*}
	\Psi^{\gamma,\gamma_0}_{\beta,\kappa} (\mathcal{S}) = \Psi^{0,0}_{\beta,\kappa} (\mathcal{S}) \prod_{\gamma' \in \mathcal{S}} \mathbf{1}(\gamma' \nsim \gamma_0) \prod_{\omega \in \mathcal{S}}\rho(\omega(\gamma)).
\end{equation*}

\begin{proposition}\label{proposition: cluster convergence 3}
	For \( \alpha \in (0,1), \) there are \( \beta_0^{(\textrm{free})}(\alpha)>0 \) and  \( \kappa_0^{(\textrm{free})} (\alpha)>0\)  such that the following holds.
	
	\begin{enumerate}
		\item For all \( \alpha\in (0,1) ,\) \( \beta>   \beta_0^{(\textrm{free})}(\alpha) , \)  \( \kappa < \kappa_0^{(\textrm{free})} (\alpha), \)  \( \gamma \in \Lambda_1\),  \( \gamma_0 \in \Lambda_0, \) and \( \eta \in \Xi,\) we have 
    \begin{equation*}
        \sum_{\mathcal{S} \in \Xi \colon \eta \in \mathcal{S}} \bigl|\Psi^{\gamma,\gamma_0}_{\beta,\kappa} (\mathcal{S})\bigr| \leq \bigl| \varphi^{\gamma,\gamma_0}_{\beta,\kappa}(\eta)\bigr|^{1-\alpha }  .
    \end{equation*}  
    \item Let \( \beta>  \beta_0^{(\textrm{free})}(\alpha)  \)  and \( \kappa <  \kappa_0^{(\textrm{free})}(\alpha)  \) for some \( \alpha \in (0,1). \)  Then, for any    \( \gamma \in \Lambda_,\) and  \( \gamma_0 \in \Lambda_0\) we have
	\begin{equation}\label{eq: cluster expansion 3}
		\log \check Z[\gamma,\gamma_0]
		=
		\sum_{\mathcal{S}\in \Xi} \Psi^{\gamma,\gamma_0}_{\beta,\kappa} (\mathcal{S}).
	\end{equation} 
    Furthermore, series on the right-hand side of~\eqref{eq: cluster expansion 3} is absolutely convergent.
	\end{enumerate}  
\end{proposition}

\begin{proof} 
	We will show that for all \( \alpha \in (0,1), \) if \( \beta \) is sufficiently large and \( \kappa \) is sufficiently small, then, if we for \( \eta \in \Lambda_1 \cup \Lambda_2 \) let   \( a(\eta) \coloneqq -\alpha \log \bigl| \varphi^{\gamma,\gamma_0}_{\beta,\kappa}(\eta) \bigr|,\)  
	we have 
	\begin{equation}\label{eq: goal ineq free}
		\sum_{\substack{ \eta'  \in \Lambda_1 \cup \Lambda_2 \mathrlap{\colon} \\ \eta' \sim \eta}} \bigl| \varphi^{\gamma,\gamma_0}_{\beta,\kappa}( \{\eta' \})  \zeta(\eta,\eta')\bigr| e^{a(\eta')} \leq \alpha |(\support \eta')^+|.
	\end{equation} 
	Given this, the conclusion of the proposition follows from~\cite[Theorem~5.4]{fv2017}. 
	To this end, let  \( \gamma' \in \Lambda_1\) and \( \omega \in \Lambda_2. \) Note that there are four different cases in~\eqref{eq: goal ineq free}, corresponding to (i) \( \eta,\eta' \in \Lambda_1, \) (ii) \( \eta_1 \in \Lambda_1 \) and \( \eta_2 \in \Lambda_2, \) (iii) \( \eta_1 \in \Lambda_2 \) and \( \eta_2 \in \Lambda_1, \) and (iv) \( \eta_1,\eta_2 \in \Lambda_2 \) respectively. We now treat these cases separately.

	(i) Let \( \gamma' \in \Lambda_1. \) Since \( \gamma' \) is closed, we have \( \bigl| \{ v \colon v \sim \gamma'  \}  \bigr| \leq |\gamma'|.\) Using, by Lemma~\ref{lemma: intersection bound for path-path}, is follows that for any \( \alpha \in (0,1) \) we have
    \begin{align*} 
        &\sum_{\substack{ \gamma''  \in \Lambda_1 \mathrlap{\colon} \\ \gamma'' \sim \gamma'}} \bigl| \varphi^{\gamma,\gamma_0}_{\beta,\kappa}(\{ \gamma''\})  \zeta(\gamma',\gamma'')\bigr| e^{a(\gamma'')} 
        =\sum_{\substack{ \gamma''  \in \Lambda_1 \mathrlap{\colon} \\ \gamma'' \sim \gamma'}} \varphi^{\gamma,\gamma_0}_{\beta,\kappa}(\{ \gamma''\})^{1-\alpha}  
        \\&\qquad\leq
        |\gamma'| \sum_{j=2}^\infty 2j(2m)^{2j} (\tanh 2 \kappa)^{2(1-\alpha)j}. 
    \end{align*} 
    
    (ii) Let \( \omega \in \Lambda_2. \) Then, by Lemma~\ref{lemma: intersection bound for spin-spin}, for any \( \alpha \in (0,1) \) we have
    \begin{align*} 
        &\sum_{ \omega'  \in \Lambda_2 } \bigl| \varphi^{\gamma,\gamma_0}_{\beta,\kappa}( \{\omega'\}) \zeta(\omega,\omega') \bigr|e^{a(\omega')}  
        =
        \sum_{\substack{ \omega'  \in \Lambda_2 \mathrlap{\colon} \\\omega' \sim \omega}} \bigl|\varphi^{\gamma,\gamma_0}_{\beta,\kappa}(\{\omega'\}) \bigr|^{1-\alpha} 
        \\&\qquad\leq 
         \bigl|(\support \omega)^+\bigr|\sum_{j=2(m-1)}^\infty 
		M_3^{2j+1} e^{-4\beta (1-\alpha) j} .
    \end{align*}

    (iii) Let \( \gamma'\in \Lambda_1. \) Then, by Lemma~\ref{lemma: intersection bound for path-spin}, for any \( \alpha \in (0,1) \) we have
    \begin{align*} 
        &\sum_{ \omega'  \in \Lambda_2 } \bigl| \varphi^{\gamma,\gamma_0}_{\beta,\kappa}(\{\omega'\}) \zeta(\gamma,\omega') \bigr| e^{a(\omega')}  
        =
        2
        \sum_{\substack{\omega'  \in \Lambda_2 \mathrlap{\colon}\\\gamma' \sim \omega' }}
         \bigl| \varphi^{\gamma,\gamma_0}_{\beta,\kappa}(\{\omega'\})\bigr|^{1-\alpha}   
        \\&\qquad\leq  
        2 | \gamma'|\sum_{j=1}^\infty D_0 j^3  \sum_{k=\max(4j, 2(d-1))}^\infty 
         M_3^{2k-1} e^{-4\beta a k}
\end{align*}

    (iv) Let \( \omega \in \Lambda_2. \) Then, by Lemma~\ref{lemma: intersection bound for spin-path}, for any \( \alpha \in (0,1) \) we have
    \begin{align*} 
        &\sum_{ \gamma'  \in \Lambda_1 } \bigl|\varphi^{\gamma,\gamma_0}_{\beta,\kappa}(\{\gamma'\}) \zeta(\omega,\gamma') \bigr|e^{a(\gamma')}  
        =
        2\sum_{ \substack{\gamma'  \in \Lambda_1 \mathrlap{\colon} \\\gamma' \sim \omega}}  \varphi^{0,0}_{\beta,\kappa}(\{\gamma'\})^{1-\alpha}  
        \\&\qquad\leq 
         2|(\support \omega)^+| \sum_{j=1}^\infty D_0j^3  \sum_{k=4j}^\infty (2m)^k (\tanh 2\kappa)^{ak}.
    \end{align*}
    
    Note that for any \( \alpha \in (0,1), \)  if \( \beta \) is sufficiently large and \( \kappa \) is sufficiently small, then the upper bounds for all cases are finite. In particular, for any \( \alpha \in (0,1), \)  if \( \beta \) is sufficiently large and \( \kappa \) is sufficiently small, then for any \( \eta \in \Lambda_1 \cup \Lambda_2 \) we have the upper bound
    \begin{align*} 
        &\sum_{ \eta'  \in \Lambda_1 \cup \Lambda_2} \bigl|\varphi^{\gamma,\gamma_0}_{\beta,\kappa} (\eta') \zeta(\eta,\eta') \bigr|e^{a(\eta')}  
        \leq \alpha |(\support \eta)^+|.
    \end{align*} 
    This concludes the proof.
\end{proof}

\subsection{Upper bounds for clusters}\label{section: upper bounds free}

Since the limiting measure \( \langle \cdot \rangle_{\beta,\kappa} \) is translation invariant, without loss of generality we can and will assume that for each \( n \geq 1 \) and all \( N \) sufficiently large, we have \( \check Z_{N,\beta,\kappa} [\gamma_1]  = \check Z_{N,\beta,\kappa} [\gamma_2] .\) 
Also, without loss of generality, we can and will assume that for each \( n \geq 1, \) \( \gamma_1^{(n)} \) and \( \gamma_2^{(n)} \) lie in the \( x_1x_2 \)-plane and has its endpoints at the \( x_1\)-axis. When \( \gamma \in \mathcal{L}_{\gamma_1^{(n)}}, \) we will let \( \hat \gamma \) denote the path defined for each edge \( ((x_1,x_2,\dots, x_m),((x_1',x_2',\dots, x_m'))) \in C_1(B_N)^+ \) by
\begin{equation*}
	\hat \gamma [((x_1,x_2,\dots, x_m),(x_1',x_2',\dots, x_m'))] = -\gamma[((x_1,-x_2,\dots, x_m),(x_1',-x_2',\dots, x_m'))]  
\end{equation*}
(see Figure~\ref{figure: mirrored path}). 

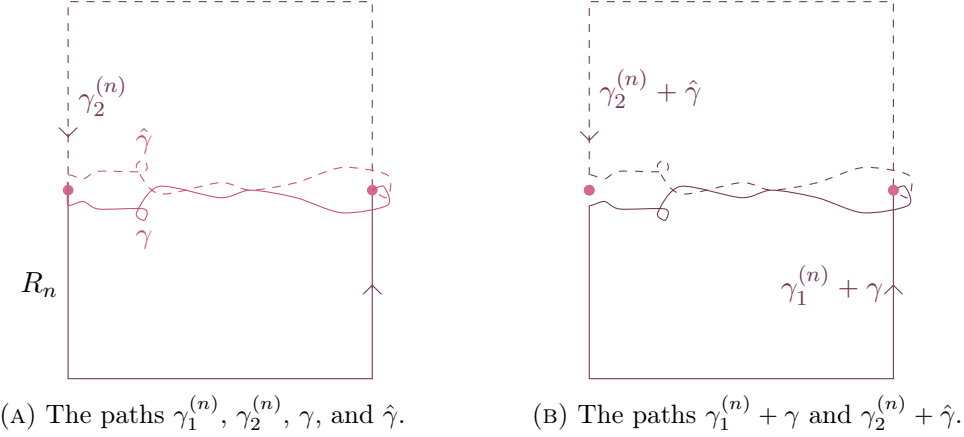
\begin{figure}[!htp]
    \centering
    \begin{subfigure}[b]{0.45\textwidth}
        \centering
        \begin{tikzpicture}
            \draw[detailcolor00] (0,2.5) -- (0,0) node[midway, left] {\color{black}\( R_n \)} -- (4,0) -- (4,2.5);

            \draw[detailcolor00, dashed] (0,2.5) -- (0,5)  node[midway,anchor=west] {\(\gamma_2^{(n)}\)}  -- (4,5)   -- (4,2.5);

           \draw[detailcolor00,-{Straight Barb[length=1.2mm,color=detailcolor00!40!black]}] (4,1.2) -- (4,1.25);
            
           \draw[detailcolor00,-{Straight Barb[length=1.2mm,color=detailcolor00!40!black]}] (0,3.25) -- (0,3.2);
           
            \fill[detailcolor04] (0,2.5) circle (2pt) node[anchor=south] {\color{black}\( \)};
            \fill[detailcolor04] (4,2.5) circle (2pt) node[anchor=south] {\color{black}\( \)};
             
             \draw[detailcolor03] (1,1.85) node[] {\( \gamma \)};
             
            \draw[detailcolor03] plot[smooth] coordinates {(0,2.5) (0,2.3) (0.2,2.35) (0.4,2.25) (1,2.25) (1,2.1) (0.9,2.2) (1.2,2.55) (2,2.4) (2.4,2.5) (3,2.4) (3.6,2.2) (4.2,2.3) (4.15,2.55) (4.0,2.5)};
              
             \begin{scope}[yscale=-1,xscale=1,yshift=-5cm]
              	\draw[detailcolor03, dashed,line cap=round] plot[smooth] coordinates {(0,2.5) (0,2.3) (0.2,2.35) (0.4,2.25) (1,2.25) (1,2.1) (0.9,2.2) (1.2,2.55) (2,2.4) (2.4,2.5) (3,2.4) (3.6,2.2) (4.2,2.3) (4.2,2.6) (4.0,2.5)};
              	
              	\draw[detailcolor03] (1,1.85) node[] {\( \hat \gamma \)};
              \end{scope}

        \end{tikzpicture}
        \caption{The paths \( \gamma_1^{(n)} ,\)   \( \gamma_2^{(n)} ,\) \( \gamma ,  \) and \( \hat \gamma. \)} \label{figure: mirrored path a}
    \end{subfigure}
    \hfil
    \begin{subfigure}[b]{0.45\textwidth}
        \centering
        \begin{tikzpicture}
             
            \draw[detailcolor00] plot[smooth] coordinates {(0,2.5) (0,2.3) (0.2,2.35) (0.4,2.25) (1,2.25) (1,2.1) (0.9,2.2) (1.2,2.55) (2,2.4) (2.4,2.5) (3,2.4) (3.6,2.2) (4.2,2.3) (4.15,2.55) (4.0,2.5)};

            \draw[white,thick] (-0.01,2.5) -- (-0.01,2.303);
            \draw[white,thick] (-0.01,2.4) -- (-0.01,2.2);
            \draw[white,thick] (0.015,2.317) -- (-0.01,2.312);
            
            \draw[detailcolor00] (0,2.3) -- (0,0)  -- (4,0)  -- (4,2.5) node[midway,anchor=east] {\(\gamma_1^{(n)} + \gamma\)}; 

           \draw[detailcolor00,-{Straight Barb[length=1.2mm,color=detailcolor00!40!black]}] (4,1.2) -- (4,1.25);
             
            \fill[detailcolor04] (0,2.5) circle (2pt) node[anchor=south] {\color{black}\( \)};
            \fill[detailcolor04] (4,2.5) circle (2pt) node[anchor=south] {\color{black}\( \)};

            \begin{scope}[yscale=-1,xscale=1,yshift=-5cm]
            	\draw[detailcolor00,dashed] plot[smooth] coordinates {(0,2.5) (0,2.3) (0.2,2.35) (0.4,2.25) (1,2.25) (1,2.1) (0.9,2.2) (1.2,2.55) (2,2.4) (2.4,2.5) (3,2.4) (3.6,2.2) (4.2,2.3) (4.2,2.6) (4.0,2.5)};  
            
            	\draw[white,thick] (-0.01,2.5) -- (-0.01,2.303);
            	\draw[white,thick] (-0.01,2.4) -- (-0.01,2.2);
            	\draw[white,thick] (0.015,2.317) -- (-0.01,2.312);
            
            	\draw[detailcolor00,dashed] (0,2.3) -- (0,0) node[midway,anchor=west] {\(\gamma_2^{(n)} + \hat\gamma\)} -- (4,0)  -- (4,2.5); 

           		\draw[detailcolor00,-{Straight Barb[length=1.2mm,color=detailcolor00!40!black]}] (0,1.8) -- (0,1.85); 
           
            	\fill[detailcolor04] (0,2.5) circle (2pt) node[anchor=south] {\color{black}\( \)};
            	\fill[detailcolor04] (4,2.5) circle (2pt) node[anchor=south] {\color{black}\( \)};
            	
            \end{scope} 
        \end{tikzpicture}
        \caption{The paths \(  \gamma_1^{(n)}+ \gamma \) and \(  \gamma_2^{(n)} + \hat \gamma. \) } \label{figure: mirrored path b}
    \end{subfigure}
    \caption{In the two figures above we illustrate the setting of the Proof of Theorem~\ref{theorem: mf ratio free}. In both pictures, we draw the open paths \( \gamma_1^{(n)} \) and  \( \gamma_2^{(n)} ,\) and also draw a path \( \gamma \in \mathcal{L}_{\gamma_1^{(n)}} .\) }
    \label{figure: mirrored path}
\end{figure}%

Before we prove Theorem~\ref{theorem: mf ratio free} we will state and prove three lemmas.

\begin{lemma}\label{lemma: convergence for gamma0 term}
	Let \( \alpha \in (0,1), \) and let \( \beta > \beta_0^{\text{(free)}} (\alpha)\) and \( \kappa <  \kappa_0^{\text{(free)}}(\alpha). \) Further, let \( \gamma \in \Lambda_0\)  Then
	\begin{align*}
		\sum_{ \mathcal{S}\in \Xi \colon  \mathcal{S}_1 \sim \gamma } \bigl|\Psi^{0,0}_{\beta,\kappa} (\mathcal{S}) \bigr|
		\leq 
		\bigl( |\gamma|+1\bigr) \sum_{j=2}^\infty 2j(2m)^{2j} (\tanh 2 \kappa)^{2(1-\alpha)j}. 
	\end{align*}
\end{lemma}

\begin{proof}
	By Proposition~\ref{proposition: cluster convergence 3}, we have
	\begin{align*}
		&\sum_{ \mathcal{S}\in \Xi \colon   \mathcal{S}_1  \sim \gamma } \bigl|\Psi^{0,0}_{\beta,\kappa}(\mathcal{S}) \bigr|
		\leq
		\sum_{\gamma' \in \Lambda_1 \colon \gamma' \sim \gamma} \sum_{\mathcal{S}\in \Xi \colon \gamma' \in \mathcal{S}} \bigl|\Psi^{0,0}_{\beta,\kappa}(\mathcal{S}) \bigr|
		\leq
		\sum_{\gamma' \in \Lambda_1 \colon \gamma' \sim \gamma} \bigl| \varphi^{0,0}_{\beta,\kappa}(\{ \gamma' \})\bigr|^{1-\alpha}.
	\end{align*}
	Applying Lemma~\ref{lemma: intersection bound for path-path} with \( a = 1-\alpha, \) we obtain the desired conclusion.  
\end{proof}

\begin{lemma}\label{lemma: convergence for gamma0 term 2}
	Let \( \alpha \in (0,1), \) and let \( \beta > \beta_0^{\text{(free)}} (\alpha)\) and \( \kappa <  \kappa_0^{\text{(free)}}(\alpha). \)  Further, let \( \gamma \) be a closed path. Then
	\begin{align*}
		\sum_{ \mathcal{S}\in \Xi \colon  \mathcal{S}_2 \sim \gamma } \bigl|\Psi^{0,0}_{\beta,\kappa} (\mathcal{S}) \bigr|
		\leq 
		D_0 | \gamma|\sum_{j=1}^\infty  j^3  \sum_{k=\max(4j, 2(d-1))}^\infty 
         M_3^{2k-1} e^{-4\beta (1-\alpha) k}.
	\end{align*}
\end{lemma}

\begin{proof}
	By Proposition~\ref{proposition: cluster convergence 3}, we have
	\begin{align*}
		&\sum_{ \mathcal{S}\in \Xi \colon   \mathcal{S}_2  \sim \gamma } \bigl|\Psi^{0,0}_{\beta,\kappa}(\mathcal{S}) \bigr|
		\leq
		\sum_{\omega \in \Lambda_2 \colon \omega \sim \gamma} \sum_{\mathcal{S}\in \Xi \colon \omega \in \mathcal{S}} \bigl|\Psi^{0,0}_{\beta,\kappa}(\mathcal{S}) \bigr|
		\leq
		\sum_{\omega \in \Lambda_2 \colon \omega \sim \gamma} \bigl| \varphi^{0,0}_{\beta,\kappa}(\{ \omega \})\bigr|^{1-\alpha}
	\end{align*}
	Applying Lemma~\ref{lemma: intersection bound for path-spin} with \( 1=1-\alpha,\) we obtain the desired conclusion.  
\end{proof}

\begin{lemma}\label{lemma: convergence for gamma0 term 4}
	Let \( \alpha \in (0,1), \) and let \( \beta > \beta_0^{\text{(free)}} (\alpha)\) and \( \kappa <  \kappa_0^{\text{(free)}} (\alpha). \) Further, let \( \gamma \in \mathcal{L}_{\gamma_1}, \) let \( k \geq 2(m-1), \) and let \( \varepsilon>0 \) be such that \( (1-\varepsilon)\beta > \beta_0^{(free)}(\alpha) \) and \( (\tanh 2\kappa)^{1-\varepsilon} < \tanh(2\kappa_0^{(free)}(\alpha)). \) Then  
	\begin{align*}
		\sum_{\substack{\mathcal{S}\in \Xi \colon \mathcal{S}_2 \sim \gamma_1+\hat \gamma,\\ \mathcal{S}_2 \sim \gamma_2+ \hat \gamma}} \bigl| \Psi^{0,0}_{\beta,\kappa} (\mathcal{S})   \bigr|
		&\leq 
		2D_0 |\gamma|e^{-8\beta\varepsilon  (m-1)} 
		  \sum_{k=2(m-1)}^\infty M_3^{2k-1} e^{-4(1-\varepsilon)\beta k}
		  \\&\hspace{4em}\cdot \sum_{j=1}^\infty  j^3 \max(e^{-4\beta},\tanh 2\kappa)^{\varepsilon\max(4j-2(m-1),0)} .
	\end{align*} 
\end{lemma}

\begin{proof}
	To simplify notation, let \(  \beta' \coloneqq (1-\varepsilon)\beta   \) and \( \kappa' \coloneqq \tanh^{-1}((\tanh 2\kappa)^{1-\varepsilon})/2 . \) Note that, by assumption, we have
	\(  \beta'   > \beta_0^{(free)}(\alpha) \) and \( \kappa' \  < \kappa_0^{(free)}(\alpha). \) Further, note that with this notation, for any \( \mathcal{S} \in \Xi \) we have
	\[
	|\Psi^{0,0}_{\beta,\kappa} (\mathcal{S})|^{1-\varepsilon} = |\Psi^{0,0}_{\beta',\kappa'}(\mathcal{S})|.
	\]

	Let \( j \geq 1 \) and  \( \mathcal{S} \in \Xi \) be such that \(  \mathcal{S}_2 \sim \gamma_1+\gamma,\) \( \mathcal{S}_2 \sim \gamma_2+\hat \gamma, \) and \( \dist(\mathcal{S}_2,\{\gamma,\hat \gamma\})=j . \) Then we must have \( \|\mathcal{S} \| \geq 4j \) and \( \|\mathcal{S}_2 \| \geq 2(m-1), \)  and hence
	\begin{align*}
		&\bigl| \Psi^{0,0}_{\beta,\kappa}(\mathcal{S}) \bigr|
		=
		\bigl| \Psi^{0,0}_{\beta,\kappa}(\mathcal{S}) \bigr|^\varepsilon
		\bigl| \Psi^{0,0}_{\beta,\kappa}(\mathcal{S}) \bigr|^{1-\varepsilon}
		=
		\bigl| \Psi^{0,0}_{\beta,\kappa}(\mathcal{S}) \bigr|^\varepsilon
		\bigl| \Psi^{0,0}_{\beta',\kappa'}(\mathcal{S}) \bigr|
		\\&\qquad\leq
		e^{-8\beta\varepsilon (m-1)} \max(e^{-4\beta},\tanh 2\kappa )^{\varepsilon \max(4j-2(m-1),0)}
		 \bigl| \Psi^{0,0}_{\beta',\kappa'}(\mathcal{S}) \bigr|.
	\end{align*} 
	Consequently, we have 
	\begin{align*}
		&\sum_{\substack{\mathcal{S}\in \Xi \colon \mathcal{S}_2 \sim \gamma_1+\gamma,\\\mathcal{S}_2 \sim \gamma_2+\hat \gamma}} \bigl| \Psi^{\gamma,\gamma_0}_{\beta,\kappa} (\mathcal{S}) \bigr|
		=
		\sum_{j=1}^\infty \sum_{\substack{\mathcal{S}\in \Xi \colon \mathcal{S}_2 \sim \gamma_1+\gamma,\gamma_2+\hat \gamma,\\ \substack{\dist(\mathcal{S}_2,\{\gamma,\hat \gamma\})=j }}} \bigl| \Psi^{0,0}_{\beta,\kappa} (\mathcal{S}) \bigr| 
		\\&\qquad\leq 
		e^{-8\beta\varepsilon  (m-1)} \sum_{j=1}^\infty  \max(e^{-4\beta},\tanh 2\kappa )^{\varepsilon\max(4j-2(m-1),0)}
		\sum_{\substack{\mathcal{S}\in \Xi \colon \\\ \substack{\dist(\mathcal{S}_2,\{\gamma,\hat\gamma\})=j }}} \bigl| \Psi^{0,0}_{\beta',\kappa'}(\mathcal{S}) \bigr|.
	\end{align*}  
	By Proposition~\ref{proposition: cluster convergence 3}, we have
	\begin{align*}
		&\sum_{\substack{\mathcal{S}\in \Xi \colon \\ \substack{\dist(\mathcal{S}_2,\{\gamma,\hat \gamma\})=j }}} \bigl| \Psi^{0,0}_{\beta',\kappa'}(\mathcal{S}) \bigr|
		\leq 
		\sum_{\substack{\omega \in \Lambda_2 \colon  \\ \dist(\mathcal{S}_2,\{\gamma,\hat \gamma\})=j}}
		\sum_{\substack{\mathcal{S}\in \Xi \colon \omega \in \mathcal{S}}} \bigl| \Psi^{0,0}_{\beta',\kappa'}(\mathcal{S}) \bigr|
		\\&\qquad\leq 
		\sum_{\substack{\omega \in \Lambda_2 \colon \\ \dist(\mathcal{S}_2,\{\gamma,\hat \gamma\})=j}}
		 \bigl| \varphi^{0,0}_{\beta',\kappa'}(\omega) \bigr|.
	\end{align*}
	Recall that there are at most \( 2D_0 j^3 |\gamma| \) positively oriented plaquettes at distance \( j \) from \( \support  \gamma \cup \support \hat \gamma. \) Also, note that each \( \omega \in \Lambda_2 \) corresponds to a connected subgraph of \( \mathcal{G}_3 \) which has a spanning path of length at most \( M_3^{2|(\support \omega)^+|-1}.\) Combining these observations, we obtain the  upper bound 
	\begin{align*}
		&\sum_{\substack{\omega \in \Lambda_2 \colon \\ \dist(\mathcal{S}_2,\{\gamma,\hat \gamma\})=j}}
		 \bigl| \varphi^{0,0}_{\beta',\kappa'}(\omega) \bigr| 
		 \leq
		  2D_0j^3 |\gamma| \sum_{k=2(m-1)}^\infty M_3^{2k-1} e^{-4\beta' k}.
	\end{align*}
	Combining the above equations, the desired conclusion follows.
\end{proof}

\subsection{Proof of Theorem~\ref{theorem: mf ratio free}}\label{section: proof of main result free}

\begin{proof}[Proof of Theorem~\ref{theorem: mf ratio free}]
	For some \( \alpha \in (0,1), \) let \( \beta > \beta_0^{\text{(free)}}(\alpha) \) and \( \kappa < \kappa_0^{\text{(free)}}(\alpha) \) 
	
	Fix \( n \geq 0. \) To simplify notation, let \( \gamma_1 = \gamma_1^{(n)} \) and \( \gamma_2 = \gamma_2^{(n)}. \)

	By Proposition~\ref{proposition: cluster convergence 3} 
	that \( \check Z_{N,\beta,\kappa} [\gamma_1]  ,\) \(  \check Z_{N,\beta,\kappa} [\gamma_2] ,\) 
	\( \check Z_{N,\beta,\kappa} [\gamma_1+\gamma_2] , \) and \( \check Z_{N,\beta,\kappa}[0] \) are all strictly positive.
	By combining Lemma~\ref{lemma: hce 3} and Lemma~\ref{lemma: high temperature free alternative}, we can thus write
	\begin{align*}
		&\rho_N(\gamma_1,\gamma_2 )^{1/2} =
		\biggl( \frac{ \check Z_{N,\beta,\kappa} [\gamma_1]   \check Z_{N,\beta,\kappa} [\gamma_2] }{ \check Z_{N,\beta,\kappa} [\gamma_1+\gamma_2]  \check Z_{N,\beta,\kappa}} \biggr)^{1/2}
		=
		\frac{ \check Z_{N,\beta,\kappa} [\gamma_1] }{ (\check Z_{N,\beta,\kappa} [\gamma_1+\gamma_2]  \check Z_{N,\beta,\kappa})^{1/2}}
		\\&\qquad= 
    	\sum_{\gamma \in \mathcal{L}_{\gamma_1}} 
    	(\tanh 2\kappa)^{|\gamma|} 
    	\frac{  {\check Z}_{N,\beta,\kappa} [\gamma_1+\gamma,\gamma] }{ \bigl( \check Z_{N,\beta,\kappa} [\gamma_1+\gamma_2,0] \check Z_{N,\beta,\kappa}[0,0]\bigr)^{1/2}}.
	\end{align*} 
	
	Now fix any \( \gamma \in \mathcal{L}_{\gamma_1} \) and  let
	\begin{align*}
		&\rho_N(\gamma_1,\gamma_2,\gamma)  \coloneqq
    	\frac{  {\check Z}_{N,\beta,\kappa} [\gamma_1+\gamma,\gamma]^2 }{ \check Z_{N,\beta,\kappa} [\gamma_1+\gamma_2,0] {\check Z}_{N,\beta,\kappa}[0,0] }.
	\end{align*}

	 By Proposition~\ref{proposition: cluster convergence 3}, we have
	\begin{align*}
		&\log \rho_N(\gamma_1,\gamma_2,\gamma) 
		=  
    	2\log    {\check Z}_{N,\beta,\kappa} [\gamma_1+\gamma,\gamma]  -\log   {\check Z}_{N,\beta,\kappa} [\gamma_1+\gamma_2,0]-\log  {\check Z}_{N,\beta,\kappa}[0,0]
    	\\&\qquad=
    	2\sum_{\mathcal{S}\in \Xi} \Psi^{\gamma_1+\gamma,\gamma}_{\beta,\kappa}(\mathcal{S})
    	-\sum_{\mathcal{S}\in \Xi} \Psi^{\gamma_1+\gamma_2,0}_{\beta,\kappa}(\mathcal{S})
    	-\sum_{\mathcal{S}\in \Xi} \Psi^{0,0}_{\beta,\kappa}(\mathcal{S}) .
	\end{align*}
	Now note that 
		\begin{align*}
		&
		2\sum_{\mathcal{S}\in \Xi} \Psi^{\gamma_1+\gamma,\gamma}_{\beta,\kappa}(\mathcal{S})
    	-\sum_{\mathcal{S}\in \Xi} \Psi^{\gamma_1+\gamma_2,0}_{\beta,\kappa}(\mathcal{S})
    	-\sum_{\mathcal{S}\in \Xi} \Psi^{0,0}_{\beta,\kappa}(\mathcal{S}) 
    	\\&\qquad=
		\sum_{\mathcal{S}\in \Xi} \Psi^{0,0}_{\beta,\kappa}(\mathcal{S}) 
		\Bigl( 2\prod_{\omega \in \mathcal{S}}\rho\bigl(\omega(q_{\gamma_1+\gamma})\bigr)  \prod_{\gamma'\in \mathcal{S}} \mathbf{1}(\gamma' \nsim \gamma)-\prod_{\omega \in \mathcal{S}} \rho\bigl(\omega(q_{\gamma_1+\gamma_2})\bigr)-1 \Bigr) 
    	\\&\qquad=
    	-
		2\sum_{\mathcal{S}\in \Xi} \Psi^{0,0}_{\beta,\kappa}(\mathcal{S}) 
		\Bigl( \prod_{\omega \in \mathcal{S}} \rho\bigl(\omega(q_{\gamma+\gamma_1})\bigr)\Bigr) \Bigl( 1-\prod_{\gamma'\in \mathcal{S}} \mathbf{1}(\gamma' \nsim \gamma) \Bigr)  \qquad (\eqqcolon A_0  ) 
		\\&\qquad\qquad+
		\sum_{\mathcal{S}\in \Xi} \Psi^{0,0}_{\beta,\kappa}(\mathcal{S}) 
		\Bigl( 2\prod_{\omega \in \mathcal{S}}\rho\bigl(\omega(q_{\gamma_1+\gamma})\bigr)   -\prod_{\omega \in \mathcal{S}} \rho\bigl(\omega(q_{\gamma_1+\gamma_2})\bigr)-1 \Bigr) .
	\end{align*}  
	Further, we have
	\begin{align*}
		&\sum_{\mathcal{S}\in \Xi} \Psi^{0,0}_{\beta,\kappa}(\mathcal{S}) 
		\Bigl( 2\prod_{\omega \in \mathcal{S}}\rho\bigl(\omega(q_{\gamma_1+\gamma})\bigr)   -\prod_{\omega \in \mathcal{S}} \rho\bigl(\omega(q_{\gamma_1+\gamma_2})\bigr)-1 \Bigr)   
		\\&\qquad=
		\sum_{\mathcal{S}\in \Xi} \Psi^{0,0}_{\beta,\kappa}(\mathcal{S}) 
		\Bigl( \prod_{\omega \in \mathcal{S}}\rho\bigl(\omega(q_{\gamma_1+\gamma})\bigr)+\prod_{\omega \in \mathcal{S}}\rho\bigl(\omega(q_{\gamma_2+\hat \gamma})\bigr)   -\prod_{\omega \in \mathcal{S}} \rho\bigl(\omega(q_{\gamma_1+\gamma_2})\bigr)-1 \Bigr)  
		\\&\qquad=
		-\sum_{\mathcal{S}\in \Xi} \Psi^{0,0}_{\beta,\kappa}(\mathcal{S})  
		\Bigr(\prod_{\omega \in \mathcal{S}}\rho\bigl(\omega(q_{\gamma_1+\gamma})\bigr)-1\Bigr)
		\Bigr(\prod_{\omega \in \mathcal{S}}\rho\bigl(\omega(q_{\gamma_2+\hat \gamma})\bigr)-1\Bigr) \qquad (\eqqcolon A_1  )
		 \\&\qquad\qquad-
		 \sum_{\mathcal{S}\in \Xi} \Psi^{0,0}_{\beta,\kappa}(\mathcal{S}) \prod_{\omega \in \mathcal{S}} \rho\bigl(\omega(q_{\gamma_1+\gamma_2})\bigr) (1- \prod_{\omega \in \mathcal{S}_2}\rho(\omega(q_{\gamma+\hat \gamma})))\Bigr)  \qquad (\eqqcolon A_2   ) .
	\end{align*}
	Consequently, if we define \( A_0 ,\) \( A_1  , \) and \( A_2  \) as above, then 
	\begin{equation*}
		\log \rho_N(\gamma_1,\gamma_2,\gamma) = A_0 + A_1+A_2.
	\end{equation*}
	We now give upper bounds for \( A_0, \) \( A_1 ,\) and \( A_2. \) To this end, note first that by Lemma~\ref{lemma: convergence for gamma0 term}, we have 
	\begin{equation*}
		|A_0| \leq  \sum_{\mathcal{S}\in \Xi \colon \mathcal{S}_1 \sim \gamma} \bigl| \Psi^{0,0}_{\beta,\kappa}(\mathcal{S})  \bigr| 
		\leq 
		\bigl( |\gamma|+1\bigr) \sum_{j=2}^\infty 2j(2m)^{2j} (\tanh 2 \kappa)^{2(1-\alpha)j}. 
	\end{equation*}
	Next, by Lemma~\ref{lemma: convergence for gamma0 term 2}, we have
	\begin{equation*}
		|A_2| \leq D_0| \gamma|\sum_{j=1}^\infty  j^3  \sum_{k=\max(4j, 2(d-1))}^\infty 
         M_3^{2k-1} e^{-4\beta (1-\alpha) k}.
	\end{equation*}
	Finally, by Lemma~\ref{lemma: convergence for gamma0 term 4}, we have
	\begin{align*}
		&|A_1| \leq 4\sum_{\substack{\mathcal{S}\in \Xi \colon \mathcal{S}_2 \sim \gamma_1+\gamma,\\\mathcal{S}_2 \sim \gamma_2+\hat \gamma}} \bigl| \Psi^{0,0}_{\beta,\kappa}(\mathcal{S}) \bigr| 
		\\&\qquad\leq
		8D_0 |\gamma|e^{-8\beta\varepsilon  (m-1)} 
		  \sum_{k=2(m-1)}^\infty M_3^{2k-1} e^{-4(1-\varepsilon)\beta k}
		  \\&\qquad\qquad\cdot \sum_{j=1}^\infty  j^3 \max(e^{-4\beta},\tanh 2\kappa )^{\varepsilon\max(4j-2(m-1),0)} .
	\end{align*} 
	for some small \( \varepsilon>0. \)
	Combining the three previous equations, it follows that as \( \beta \to \infty \) and \( \kappa \to 0, \) we have
	\begin{equation}\label{eq: small uniform upper bound higgs}
		\frac{\log \rho_N(\gamma_1,\gamma_2,\gamma)}{|\gamma|} \leq o_{\beta}(1) + o_\kappa(1),
	\end{equation}
	where the right hand side is independent of \( \gamma_1, \) \( \gamma_2, \) \( \gamma, \) \( N, \) and \( n. \) In particular, if \( \beta \) is sufficiently large and \( \kappa \) is sufficiently small, then the right-hand-side of~\eqref{eq: small uniform upper bound higgs} is strictly smaller than \( 1/(2m \tanh 2\kappa). \)
	Finally, we note that since any vertex of \( B_N \) has degree \( 2m, \) we have
	\begin{align*}
		\lim_{n\to \infty}\sum_{\gamma \in \mathcal{L}_{\gamma_1^{(n)} }} 
    	(\tanh 2\kappa)^{|\gamma|} 
    	\leq \lim_{n\to \infty}\sum_{j = T_n}^\infty (2m)^j (\tanh 2\kappa)^{j} 
    	= 0.
	\end{align*}
	Noting that
	\[
		\rho_N(\gamma_1^{(n)},\gamma_2^{(n)})^{1/2} = \sum_{\gamma \in \mathcal{L}_{\gamma_1^{(n)}}} (\tanh 2\kappa)^{|\gamma|} \rho_N(\gamma_1,\gamma_2,\gamma),
	\]
	the desired conclusion immediately follows. 
\end{proof}

\end{document}